\documentclass[11pt,twoside]{amsart}

\usepackage[T1]{fontenc}
\usepackage[latin1]{inputenc}% Codage du fichier en ISO-Latin-1 (accents...)
\usepackage[english]{babel}% Utilisation du Franà¯Â¿Â½is comme langue principale, de l'anglais comme langue secondaire

\usepackage{amsmath,amsthm,amssymb,amsfonts,amscd}
\usepackage{url}
\usepackage{color}
\usepackage{setspace}
\usepackage{enumerate}

\usepackage{epsfig}
\usepackage{graphicx}
\usepackage{mathrsfs}
\usepackage{pinlabel}
\usepackage[outercaption]{sidecap}

\definecolor{light-gray}{gray}{0.60}

\newcounter{notes}%

\theoremstyle{plain}
\newtheorem{theorem}{Theorem}
\newtheorem{proposition}[theorem]{Proposition}
\newtheorem{corollary}[theorem]{Corollary}

\newtheorem{lemma}[theorem]{Lemma}
\newtheorem{fact}[theorem]{Fact}
\newtheorem{notation}[theorem]{Notation}

\newtheoremstyle{theoremwithref}{}{}{\itshape}{}{\bfseries}{.}{.5em}{#1 #2 #3}
\theoremstyle{theoremwithref}

\theoremstyle{definition}
\newtheorem{definition}[theorem]{Definition}
\newtheorem{example}[theorem]{Example}
\newtheorem{remark}[theorem]{Remark}
\newtheorem{remarks}[theorem]{Remarks}

\numberwithin{theorem}{section}
\numberwithin{equation}{section}

\newcommand{\D}{\mathrm{d}}

\newcommand{\NN}{\mathbb{N}}
\newcommand{\ZZ}{\mathbb{Z}}

\newcommand{\RR}{\mathbb{R}}
\newcommand{\CC}{\mathbb{C}}
\newcommand{\HH}{\mathbb{H}}
\newcommand{\PP}{\mathbb{P}}

\renewcommand{\SS}{\mathbb{S}}

\newcommand{\SL}{\mathrm{SL}}
\newcommand{\GL}{\mathrm{GL}}
\newcommand{\SO}{\mathrm{SO}}
\newcommand{\OO}{\mathrm{O}}
\newcommand{\PO}{\mathrm{PO}}

\newcommand{\PSL}{\mathrm{PSL}}
\newcommand{\PGL}{\mathrm{PGL}}

\newcommand{\Hom}{\mathrm{Hom}}

\newcommand{\ie}{i.e.\ }
\newcommand{\eg}{e.g.\ }
\newcommand{\resp}{resp.\ }

\newcommand{\C}{\mathcal{C}}
\newcommand{\DD}{\mathcal{D}}

\newcommand*{\longhookrightarrow}{\ensuremath{\lhook\joinrel\relbar\joinrel\rightarrow}}

\title{Convex cocompactness in pseudo-Riemannian hyperbolic spaces}

\author{Jeffrey Danciger}
\address{Department of Mathematics, The University of Texas at Austin, 1 University Station C1200, Austin, TX 78712, USA}
\email{jdanciger@math.utexas.edu}

\author{Fran\c{c}ois Gu\'eritaud}
\address{CNRS and Universit\'e Lille 1, Laboratoire Paul Painlev\'e, 59655 Villeneuve d'Ascq Cedex, France}
\email{francois.gueritaud@math.univ-lille1.fr}

\author{Fanny Kassel}
\address{CNRS and Institut des Hautes \'Etudes Scientifiques, Laboratoire Alexander Grothendieck, 35 route de Chartres, 91440 Bures-sur-Yvette, France}
\email{kassel@ihes.fr}

\dedicatory{To Bill Goldman on his 60th birthday}

\thanks{J.D. was partially supported by an Alfred P. Sloan Foundation fellowship and by the National Science Foundation grant DMS 1510254.
F.G. and F.K. were partially supported by the Agence Nationale de la Recherche through the Labex CEMPI (ANR-11-LABX-0007-01).
F.K. was partially supported by the European Research Council under ERC Starting Grant 715982 (DiGGeS).
The authors also acknowledge support from the GEAR Network, funded by the National Science Foundation grants DMS 1107452, 1107263, and 1107367 (``RNMS: GEometric structures And Representation varieties").}

\begin{document}

\maketitle

\begin{abstract}
Anosov representations of word hyperbolic groups into\linebreak higher-rank semisimple Lie groups are representations with finite kernel and discrete image that have strong analogies with convex cocompact representations into rank-one Lie groups.
However, the most naive analogy fails: generically, Anosov representations do not act properly and cocompactly on a convex set in the associated Riemannian symmetric space.
We study representations into projective indefinite orthogonal groups $\PO(p,q)$ by considering their action on the associated pseudo-Riemannian hyperbolic space $\HH^{p,q-1}$ in place of the Riemannian symmetric space.
Following work of Barbot and M\'erigot in anti-de Sitter geometry, we find an intimate connection between Anosov representations and a natural notion of convex cocompactness in this setting.
\end{abstract}

%%%%%%%%%%%%%%%%%%%%%%%%%%%%%%%%%%%%%%%%%%%%%%%%%%%
\section{Introduction}

Convex cocompact subgroups of rank-one semisimple Lie groups are an important class of discrete groups whose actions on the associated Riemannian symmetric space (and its visual boundary at infinity) exhibit many desirable geometric and dynamical properties.
Their study has been particularly important in the setting of Kleinian groups and hyperbolic geometry. 
This paper studies a generalized notion of convex cocompactness in the higher-rank setting of projective indefinite orthogonal groups $\PO(p,q)$, described in terms of the action on the projective space $\PP(\RR^{p,q})$ and on the associated pseudo-Riemannian hyperbolic space $\HH^{p,q-1}$.
Our forthcoming papers \cite{dgk-cc,dgk-racg-cc} will extend many of these ideas to the setting of discrete subgroups of the projective general linear group $\PGL(\RR^n)$ which do not necessarily preserve any nonzero quadratic form.

%%%%%%%%%%%%%%%%%%%%%%%%%
\subsection{Convex cocompactness in projective orthogonal groups} \label{subsec:intro-Hpq-cc}

In the whole paper, we fix integers $p,q\geq 1$ and let $G=\PO(p,q)$ be the orthogonal group, modulo its center $\{\pm\mathrm{I}\}$, of a nondegenerate symmetric bilinear form $\langle\cdot,\cdot\rangle_{p,q}$ of signature $(p,q)$ on~$\RR^{p+q}$.
We denote by $\RR^{p,q}$ the space $\RR^{p+q}$ endowed with the symmetric bilinear form $\langle\cdot,\cdot\rangle_{p,q}$.
For any linear subspace $W$ of $\RR^{p,q}$, we denote by $W^{\perp}$ the orthogonal of $W$ for $\langle\cdot,\cdot\rangle_{p,q}$.
We use similar notation in $\PP(\RR^{p,q})$: in particular, for $z\in\PP(\RR^{p,q})$ the set $z^{\perp}$ is a projective hyperplane of $\PP(\RR^{p,q})$, which contains $z$ if and only if $\langle z,z\rangle_{p,q}=0$.

When $q=1$, the group $G$ is the group of isometries of the real hyperbolic space
$$\HH^p = \big\{[x]\in\PP(\RR^{p,1})\,|\,\langle x,x\rangle_{p,1}<0\big\},$$
which is also the Riemannian symmetric space $G/K$ associated with~$G$.
Recall that a
discrete
subgroup $\Gamma$ of $G=\PO(p,1)$ is said to be \emph{convex cocompact} if it acts cocompactly on some nonempty closed convex subset $\mathcal{C}$ of~$\HH^p$. Note that since $\Gamma$ is discrete and $\HH^p$ is Riemannian, the action is automatically properly discontinuous, and so the quotient $\Gamma \backslash \mathcal{C}$ is a hyperbolic orbifold, or a manifold if the action is free, with convex boundary.
Basic examples of convex cocompact subgroups include uniform lattices of~$G$ and Schottky subgroups of~$G$.
In the case $p=3$, for which the accidental isomorphism $\PO(3,1)_0 \simeq \PSL_2(\CC)$ makes $G$ a complex group, the realm of Kleinian groups gives an abundance of interesting examples coming both from complex analysis \`a la Ahlfors and Bers and from $3$-manifold topology and Thurston's geometrization program.
Notable are the quasi-Fuchsian groups (isomorphic to closed surface groups) which are deformations of Fuchsian subgroups of $\PO(2,1) \subset \PO(3,1)$.

Assume that $G=\PO(p,q)$ has real rank $\geq 2$, \ie $\min(p,q)\geq 2$, and let $K=\mathrm{P}(\OO(p)\times\OO(q))$ be a maximal compact subgroup of~$G$.
The group $G$ is the isometry group of the Riemannian symmetric space $G/K$, and it is natural to study the discrete subgroups $\Gamma$ of~$G$ that act cocompactly on some convex subset of $G/K$. This naive generalization of convex cocompactness turns out to be quite restrictive due to the following general result proved independently by Kleiner--Leeb \cite{kl06} and Quint~\cite{qui05}.

\begin{fact}[{\cite{kl06,qui05}}]
Let $G$ be a real semisimple Lie group of real rank $\geq 2$ and $K$ a maximal compact subgroup of~$G$.
Any Zariski-dense discrete subgroup of~$G$ acting cocompactly on some nonempty closed convex subset $\mathcal{C}$ of the Riemannian symmetric space $G/K$ is a uniform lattice in~$G$.
\end{fact}

In this paper, we propose instead a notion of convex cocompactness in $G = \PO(p,q)$ in terms of the action on the real projective space $\PP(\RR^{p,q})$, and in particular on the invariant open domain
$$\HH^{p,q-1} = \big\{[x]\in\PP(\RR^{p,q})\,|\,\langle x,x\rangle_{p,q}<0\big\} \simeq G/\OO(p,q-1)$$
which is the projective model for a pseudo-Riemannian symmetric space associated to~$G$.
Indeed, $\HH^{p,q-1}$ has a natural pseudo-Riemannian structure of signature $(p,q-\nolinebreak 1)$ with isometry group~$G$, induced by the symmetric bilinear form $\langle\cdot,\cdot\rangle_{p,q}$ (see Section~\ref{subsec:Hpq}).
Geodesics of $\HH^{p,q-1}$ are intersections of $\HH^{p,q-1}$ with straight lines of $\PP(\RR^{p,q})$.
For $q=1$, the space $\HH^{p,q-1}$ is the real hyperbolic space~$\HH^p$, in its projective model.
In general, $\HH^{p,q-1}$ is a pseudo-Riemannian analogue of~$\HH^p$ of signature $(p,q-1)$, with constant negative sectional curvature.

Recall that a subset of projective space is said to be \emph{convex} if it is contained and convex in some affine chart; in other words, any two points of the subset are connected inside the subset by a unique projective segment.
A subset of projective space is said to be \emph{properly convex} if its closure is convex.
Unlike real hyperbolic space, for $q>1$ the space $\HH^{p,q-1}$ is not a convex subset of the projective space $\PP(\RR^{p,q})$, and the basic operation of taking convex hulls is not well defined.
Nonetheless, the notion of convexity in $\HH^{p,q-1}$ makes sense: we shall say that a subset $\C$ of $\HH^{p,q-1}$ is \emph{convex} if it is convex as a subset of $\PP(\RR^{p,q})$ or, from an intrinsic point of view, if any two points of~$\C$ are connected inside~$\C$ by a unique segment which is geodesic for the pseudo-Riemannian structure.
We shall say that $\C$ is \emph{properly convex} if its closure in $\PP(\RR^{p,q})$ is convex.

For $q=2$, the Lorentzian space $\HH^{p,q-1}$ is the $(p+1)$-dimensional \emph{anti-de Sitter space} $\mathrm{AdS}^{p+1}$, where a notion of \emph{AdS quasi-Fuchsian group} has been studied by Mess \cite{mes90} (for $p=2$) and Barbot--M\'erigot \cite{bm12,bar15} (for $p\geq 3$).
Inspired by this notion, we make the following definition.

\begin{definition}\label{def:ccc-Hpq}
A discrete subgroup $\Gamma$ of $G=\PO(p,q)$ is \emph{$\HH^{p,q-1}$-convex cocompact} if it acts properly discontinuously and cocompactly on some properly convex closed subset $\C$ of~$\HH^{p,q-1}$ with nonempty interior whose ideal boundary $\partial_i\C := \overline{\C} \smallsetminus \C$ does not contain any nontrivial projective segment.
\end{definition}

Here $\overline \C$ denotes the closure of $\C$ in $\PP(\RR^{p,q})$.
We note that an $\HH^{p,q-1}$-convex cocompact group is always finitely generated.

\begin{remark}
We shall say that a subgroup $\Gamma$ of $G=\PO(p,q)$ is \emph{irreducible} if it does not preserve any projective subspace of $\PP(\RR^{p,q})$ of positive codimension.
In that case, any nonempty $\Gamma$-invariant convex subset of $\PP(\RR^{p,q})$ has nonempty interior, and so ``$\C$ with nonempty interior'' may be replaced by ``$\C$ nonempty'' in Definition~\ref{def:ccc-Hpq}.
\end{remark}

Note that a discrete subgroup $\Gamma$ of $\PO(p,q)$ need not act properly discontinuously on $\HH^{p,q-1}$, since the stabilizer $\OO(p,q-1)$ of a point is noncompact.
When $\Gamma$ preserves a properly convex subset $\C \subset \HH^{p,q-1}$, the action on the interior of~$\C$ is always properly discontinuous (see Section~\ref{subsec:prop-conv-proj}), but the action on the whole of~$\C$ need not be.
The requirement of proper discontinuity in Definition~\ref{def:ccc-Hpq} guarantees that the accumulation points of any $\Gamma$-orbit are contained in the ideal boundary $\partial_i\C$.
Since $\C$ is assumed to be closed in $\HH^{p,q-1}$, the set $\partial_i\C$ is contained in the boundary
$$\partial_{\scriptscriptstyle\PP}\HH^{p,q-1} := \big\{[x]\in\PP(\RR^{p,q})\,|\,\langle x,x\rangle_{p,q}=0\big\}$$
of $\HH^{p,q-1}$.
The condition that $\partial_i\C$ not contain any nontrivial projective segment is equivalent to the condition that $\partial_i\C$ be \emph{transverse}, \ie $y\notin z^{\perp}$ for all $y\neq z$ in $\partial_i\C$.

\begin{remark}
When $\Gamma$ is a discrete subgroup of $\PO(p,q)$ which is \emph{not} irreducible, it is possible that $\Gamma$ act properly discontinuously and cocompactly on some closed convex subset $\C$ of $\HH^{p,q-1}$ with nonempty interior but that $\partial_i\C$ contain nontrivial projective segments, as the following example shows.
It is not clear whether this is possible for irreducible~$\Gamma$.
\end{remark}

\begin{example} \label{ex:bad-cyclic}
Let $\gamma$ be an element of $\PO(p,q)$ whose top eigenvalue $\lambda >\nolinebreak 1$ has multiplicity larger than one.
The cyclic group $\Gamma = \langle \gamma \rangle$ acts properly discontinuously and cocompactly on a closed convex neighborhood $\C$ in $\HH^{p,q-1}$ of a line connecting two points of $\partial_{\scriptscriptstyle\PP} \HH^{p,q-1}$ corresponding to eigenvectors of eigenvalue $\lambda$ and $\lambda^{-1}$ respectively.
However, the ideal boundary $\partial_i \C$ must contain a nontrivial segment of the projectivization of the highest eigenspace.
Thus $\Gamma$ is not $\HH^{p,q-1}$-convex cocompact in the sense of Definition~\ref{def:ccc-Hpq}.
\end{example}

It is well known that (unlike in Example~\ref{ex:bad-cyclic}) an irreducible discrete subgroup $\Gamma$ of $\PGL(\RR^n)$ preserving a nonempty properly convex subset $\C$ of $\PP(\RR^n)$ always contains a \emph{proximal} element, \ie an element~$\gamma$ with a unique attracting fixed point in $\PP(\RR^n)$ (see \cite[Prop.\,3.1]{ben00}).
We shall call \emph{proximal limit set} of $\Gamma$ in $\PP(\RR^n)$ the closure $\Lambda_\Gamma \subset \PP(\RR^n)$ of the set of attracting fixed points of proximal elements of~$\Gamma$ (Definition~\ref{def:limit-set}).
If the action of $\Gamma$ on~$\C$ is properly discontinuous, then the proximal limit set $\Lambda_{\Gamma}$ is contained in the ideal boundary $\partial_i\C$: indeed, for any proximal element $\gamma\in\Gamma$ and any point $y$ in the interior of~$\C$, the sequence $(\gamma^m\cdot y)_{m\in\NN}$ converges to the attracting fixed point of $\gamma$ in $\PP(\RR^n)$, which belongs to $\partial_i\C$ since the action is properly discontinuous.
For $\Gamma$ and $\C$ as in Definition~\ref{def:ccc-Hpq}, we shall see (Theorem~\ref{thm:main}) that in fact $\Lambda_\Gamma = \partial_i\C$.

\begin{remark}\label{rem:Spq-Hpq}
The boundary $\partial_{\scriptscriptstyle\PP}\HH^{p,q-1}$ divides $\PP(\RR^{p,q})$ into two connected components.
One component is $\HH^{p,q-1}$ and the other is 
$$\SS^{p-1,q} = \{[x]\in\PP(\RR^{p,q})\,|\,\langle x,x\rangle_{p,q}>\nolinebreak 0\},$$
which inherits from $\langle \cdot, \cdot \rangle_{p,q}$ a pseudo-Riemannian metric of positive curvature.
However, multiplication by $-1$ transforms $\langle \cdot, \cdot \rangle_{p,q}$ into a form of signature $(q,p)$, and $\SS^{p-1,q}$ into the copy of $\HH^{q,p-1}$ defined by $-\langle \cdot, \cdot \rangle_{p,q}$.
Rather than study two very similar notions of convex cocompactness in pseudo-Riemannian hyperbolic spaces $\HH^{p,q-1}$ and pseudo-Riemannian ``spheres''\linebreak $\SS^{p-1,q}$, we will use the isomorphism $\PO(\langle \cdot, \cdot \rangle_{p,q}) = \PO(-\langle \cdot, \cdot \rangle_{p,q})\simeq \PO(q,p)$ to exchange $\SS^{p-1,q}$ with $\HH^{q,p-1}$ when convenient.
\end{remark}

%%%%%%%%%%%%%%%%%%%%%%%%%
\subsection{Goals of the paper}

There are three main goals.
First, we show that the notion of convex cocompactness introduced above is closely related to the notion of Anosov representation --- a notion that has become fundamental in the study of higher Teichm\"uller theory.
Second, we show that in the setting of discrete irreducible subgroups of $\PO(p,q)$, our notion of $\HH^{p,q-1}$-convex cocompactness is equivalent to a notion of strong convex cocompactness in $\PP(\RR^n)$ introduced by Crampon--Marquis \cite{cm14}, for $n=p+q$.
Third, we show that a natural construction of Coxeter groups in projective orthogonal groups going back to Tits gives rise to many examples of $\HH^{p,q-1}$-convex cocompact groups, hence to many new examples of Anosov representations into $\PO(p,q)$ and of strongly convex cocompact groups in $\PP(\RR^{p+q})$.

%%%%%%%%%%%%%%%%%%%%%%%%%
\subsection{Link with Anosov representations} \label{subsec:intro-Anosov}

The main result that we establish in this paper is a close connection between convex cocompactness in $\HH^{p,q-1}\subset\PP(\RR^{p+q})$ and Anosov representations.

Anosov representations of word hyperbolic groups into real semisimple Lie groups are representations with finite kernel and discrete image, defined by the dynamics of their action on some flag varieties.
They were introduced by Labourie \cite{lab06} for fundamental groups of compact negatively-curved manifolds, and generalized for arbitrary word hyperbolic groups by Guichard--Wienhard \cite{gw12}.
They have been extensively studied recently by many authors (see \eg \cite{bcls15,klp-survey,ggkw17,bps} to just name a few) and now play a crucial role in higher Teichm\"uller theory (see \eg \cite{biw14}); they share many dynamical properties with classical convex cocompact subgroups of rank-one simple Lie groups (see in particular \cite{lab06,gw12,klp14,klp-survey}).

Let $P_1^{p,q}$ be the stabilizer in $G=\PO(p,q)$ of an isotropic line of $\RR^{p,q}$; it is a parabolic subgroup of~$G$, and $G/P_1^{p,q}$ identifies with the boundary $\partial_{\scriptscriptstyle\PP}\HH^{p,q-1}$ of $\HH^{p,q-1}$.
By definition, a \emph{$P_1^{p,q}$-Anosov} representation of a word hyperbolic group $\Gamma$ into~$G$ is a representation $\rho : \Gamma\to G$ for which there exists a continuous, $\rho$-equivariant boundary map $\xi : \partial_{\infty}\Gamma\to\partial_{\scriptscriptstyle\PP}\HH^{p,q-1}$ which
\begin{enumerate}[(i)]
  \item\label{item:def-Ano-transv} is transverse (a strengthening of injectivity), meaning that $\xi(\eta)\notin\xi(\eta')^{\perp}$ for any $\eta\neq\eta'$ in $\partial_{\infty}\Gamma$,
  \item\label{item:flow} has an associated flow with some uniform contraction/expansion properties described in \cite{lab06,gw12}.
\end{enumerate}
Here $\partial_{\infty}\Gamma$ denotes the Gromov boundary of~$\Gamma$.
A consequence of \eqref{item:flow} is that $\xi$ is \emph{dynamics-preserving}: for any infinite-order element $\gamma\in\Gamma$, the element $\rho(\gamma)\in G$ is proximal in $\partial_{\scriptscriptstyle\PP}\HH^{p,q-1}$, and $\xi$ sends the attracting fixed point of $\gamma$ in $\partial_{\infty}\Gamma$ to the attracting fixed point of $\rho(\gamma)$ in $\partial_{\scriptscriptstyle\PP}\HH^{p,q-1}$.
In particular, by a density argument, the continuous map $\xi$ is unique, and the image $\xi(\partial_\infty \Gamma)$ is the proximal limit set $\Lambda_{\rho(\Gamma)}$ of $\rho(\Gamma)$ in $\partial_{\scriptscriptstyle\PP}\HH^{p,q-1}$ (Definition~\ref{def:limit-set} and Remark~\ref{rem:lim-set-POpq}).
By \cite[Prop.\,4.10]{gw12}, if $\rho(\Gamma)$ is irreducible, then condition~\eqref{item:flow} is automatically satisfied as soon as \eqref{item:def-Ano-transv} is.
If $\Gamma$ is finite, then $\partial_{\infty}\Gamma$ is empty and any representation $\rho : \Gamma\to G$ is $P_1^{p,q}$-Anosov.

In real rank~1, it is easy to see \cite[Th.\,5.15]{gw12} that a discrete subgroup $\Gamma$ of $G=\PO(p,1)$ is convex cocompact if and only if $\Gamma$ is word hyperbolic and the natural inclusion $\Gamma\hookrightarrow G$ is $P_1^{p,1}$-Anosov.
In this paper, we prove the following generalization to higher real rank.

\begin{theorem} \label{thm:main}
For $p,q\in\NN^*$, let $\Gamma$ be an irreducible discrete subgroup of $G=\PO(p,q)$.
\begin{enumerate}[(1)]
  \item\label{item:ccc-implies-Anosov} If $\Gamma$ is $\HH^{p,q-1}$-convex cocompact, then it is word hyperbolic and the natural inclusion $\Gamma\hookrightarrow G$ is $P_1^{p,q}$-Anosov.
  \item\label{item:Ano-connected-implies-ccc} Conversely, if $\Gamma$ is word hyperbolic with connected boundary $\partial_{\infty}\Gamma$ and if the natural inclusion $\Gamma\hookrightarrow G$ is $P_1^{p,q}$-Anosov, then $\Gamma$ is $\HH^{p,q-1}$-convex cocompact or $\HH^{q,p-1}$-convex cocompact (after identifying $\PO(p,q)$ with $\PO(q,p)$, see Remark~\ref{rem:Spq-Hpq}).
\end{enumerate}
If these conditions are satisfied, then for any nonempty properly convex closed subset $\C$ of $\HH^{p,q-1}$ on which $\Gamma$ acts properly discontinuously and cocompactly, the ideal boundary $\partial_i\C$ is the proximal limit set $\Lambda_\Gamma\subset\partial_{\scriptscriptstyle\PP}\HH^{p,q-1}$.
\end{theorem}

\begin{remark}
The special case when $q=2$ and $\Gamma$ is the fundamental group of a closed hyperbolic $p$-manifold follows from work of Mess \cite{mes90} for $p=2$ and work of Barbot--M\'erigot \cite{bm12} for $p\geq 3$.
In that case, if $\Omega_{\max}$ denotes a maximal $\Gamma$-invariant properly convex open subset of $\PP(\RR^{p,q})$ (see Proposition~\ref{prop:Lambda-non-pos-neg}), then the manifold $\Gamma\backslash\Omega_{\max}$ is a \emph{GHMC} spacetime (globally hyperbolic maximal Cauchy-compact) \cite[Th.\,4.3 \& Prop.\,4.5]{bm12}.
We refer to \cite[\S\,11]{dgk-cc} for further discussion of global hyperbolicity when $q=2$.
\end{remark}

%%%%%%%%%%%%%%%%%%%%%%%%%
\subsection{Anosov representations with negative or positive limit set}

We may replace the connectedness assumption of Theorem~\ref{thm:main}.\eqref{item:Ano-connected-implies-ccc} with the following simple consistency condition on the image of the boundary map.

\begin{definition} \label{def:pos-neg}
A subset $\Lambda$ of $\partial_{\scriptscriptstyle\PP}\HH^{p,q-1}$ is \emph{negative} (\resp \emph{positive}) if it lifts to a cone of $\RR^{p,q}\smallsetminus\{ 0\}$ on which all inner products $\langle\cdot,\cdot\rangle_{p,q}$ of noncollinear points are negative (\resp positive).
Equivalently (Lemma~\ref{lem:justif-neg-triple} and Remark~\ref{rem:pos-neg}.\eqref{item:other-equiv-pos-neg}), every triple of distinct points of~$\Lambda$ spans a triangle fully contained in $\HH^{p,q-1}$ (\resp $\SS^{p-1,q}$) outside of the vertices.
\end{definition}

By a \emph{cone} we mean a subset of $\RR^{p,q}\smallsetminus\{ 0\}$ which is invariant under multiplication by positive scalars.
Recall from Remark~\ref{rem:Spq-Hpq} that $\HH^{p,q-1}$ and $\SS^{p-1,q}$ are the two connected components of $\PP(\RR^{p,q})\smallsetminus\partial_{\scriptscriptstyle\PP}\HH^{p,q-1}$.
In the Lorentzian setting (\ie $q = 2$), a negative subset of $\partial_{\scriptscriptstyle\PP}\HH^{p,q-1}$ is also called an \emph{acausal} subset.

Since the connectedness of~$\Lambda$ implies the connectedness of the set of unordered distinct triples of~$\Lambda$ (Fact~\ref{fact:conn-ktuple}), the following holds (see Section~\ref{subsec:proof-conn-transv}).

\begin{proposition} \label{prop:conn-transv}
If a closed subset $\Lambda$ of $\partial_{\scriptscriptstyle\PP}\HH^{p,q-1}$ is connected and transverse, then it is negative or positive.
\end{proposition}

As above, we say that $\Lambda$ is \emph{transverse} if for any $y\neq z$ in~$\Lambda$ we have $y\notin z^{\perp}$.

Theorem~\ref{thm:main} is an immediate consequence of Proposition~\ref{prop:conn-transv} and of the following, which is the main result of the paper.

\begin{theorem} \label{thm:main-negative}
For any $p,q\in\NN^*$ and any irreducible discrete subgroup $\Gamma$ of $G = \PO(p,q)$, the following two conditions are equivalent:
\begin{enumerate}[(i)]
  \item\label{item:a1} $\Gamma$ is $\HH^{p,q-1}$-convex cocompact,
  \item\label{item:a2} $\Gamma$ is word hyperbolic, the natural inclusion $\Gamma\hookrightarrow G$ is $P_1^{p,q}$-Anosov, and the proximal limit set $\Lambda_{\Gamma}\subset\partial_{\scriptscriptstyle\PP}\HH^{p,q-1}$ is negative.
\end{enumerate}
Similarly, the following two conditions are equivalent:
\begin{enumerate}[(i)]\setcounter{enumi}{2}
  \item\label{item:a3} $\Gamma$ is $\HH^{q,p-1}$-convex cocompact (after identifying $\PO(p,q)$ with $\PO(q,p)$),
  \item\label{item:a4} $\Gamma$ is word hyperbolic, the natural inclusion $\Gamma\hookrightarrow G$ is $P_1^{p,q}$-Anosov, and the proximal limit set $\Lambda_{\Gamma}\subset\partial_{\scriptscriptstyle\PP}\HH^{p,q-1}$ is positive.
\end{enumerate}
If (\ref{item:a1}) (\resp (\ref{item:a3})) holds, then for any nonempty properly convex closed subset $\C$ of $\HH^{p,q-1}$ (\resp $\HH^{q,p-1}$) on which $\Gamma$ acts properly discontinuously and cocompactly, the ideal boundary $\partial_i\C$ is the proximal limit set $\Lambda_\Gamma\subset\partial_{\scriptscriptstyle\PP}\HH^{p,q-1}$.
\end{theorem}

By \cite{lab06,gw12}, the space of $P_1^{p,q}$-Anosov representations is open in $\Hom(\Gamma,G)$.
In fact, the space of $P_1^{p,q}$-Anosov representations with negative proximal limit set is also open (Proposition~\ref{prop:comp-Ano-neg}).
Moreover, the space of irreducible representations is open.
Therefore Theorem~\ref{thm:main-negative} implies the following.

\begin{corollary} \label{coro:cc-open}
For any $p,q\in\NN^*$ and any finitely generated group~$\Gamma$, the set of irreducible injective representations $\Gamma\to G=\PO(p,q)$ whose image is $\HH^{p,q-1}$-convex cocompact is open in $\Hom(\Gamma,G)$.
\end{corollary}

\begin{remark}
In the special case when $p\geq q=2$ (\ie $\HH^{p,q-1}$ is the Lorentzian \emph{anti-de Sitter space} $\HH^{p,1}=\mathrm{AdS}^{p+1}$) and $\Gamma$ is isomorphic to the fundamental group of a closed, negatively-curved Riemannian $p$-manifold, the following strengthening of Theorem~\ref{thm:main-negative} holds by work of Barbot \cite{bar15}: $\Gamma$ is $\HH^{p,1}$-convex cocompact if and only if its proximal limit set $\Lambda_{\Gamma}$ is a topological $(p-1)$-sphere which is negative (Definition~\ref{def:pos-neg}), if and only if $\Lambda_{\Gamma}$ is a topological $(p-1)$-sphere which is \emph{nonpositive} (\ie it lifts to a cone of $\RR^{p,2}\smallsetminus\nolinebreak\{ 0\}$ on which $\langle\cdot,\cdot\rangle_{p,2}$ is nonpositive); this property is called \emph{GH-regularity} \cite[\S\,1.3]{bar15}.
Using this, Barbot shows that the space of $P_1^{p,2}$-Anosov representations of~$\Gamma$ into $G=\PO(p,2)$ is not only open but also closed in $\Hom(\Gamma,G)$, hence it is a union of connected components of $\Hom(\Gamma,G)$ \cite[Th.\,1.2]{bar15}.
This becomes false when $\Gamma$ has virtual cohomological dimension~$<p$: for instance, when $\Gamma$ is a finitely generated free group the space $\Hom(\Gamma,\PO(p,q)_0)$ is connected but contains both Anosov and non-Anosov representations.
\end{remark}

\begin{remark}
For $\mathrm{rank}_{\RR}(G):=\min(p,q)\geq 2$, there are examples of irreducible $P_1^{p,q}$-Anosov representations $\rho : \Gamma\to G=\PO(p,q)$ for which the proximal limit set $\Lambda_{\rho(\Gamma)}\subset\partial_{\scriptscriptstyle\PP}\HH^{p,q-1}$ is neither negative nor positive: see Section~\ref{subsec:noncococo}.
By Theorem~\ref{thm:main-negative} the group $\rho(\Gamma)$ is neither $\HH^{p,q-1}$-convex cocompact nor $\HH^{q,p-1}$-convex cocompact in this case.
In such examples $\partial_\infty \Gamma$ is always disconnected.
Thus one cannot remove the connectedness assumption in Theorem~\ref{thm:main}.\eqref{item:Ano-connected-implies-ccc}.
This subtlety should be kept in mind when reading \cite[\S\,8.2]{bm12}.
\end{remark}

\begin{remark}
The irreducibility assumption in this paper makes properly discontinuous actions on properly convex sets more tractable (see Fact~\ref{fact:Yves} below) and the notion of Anosov representation simpler (condition~\eqref{item:flow} of Section~\ref{subsec:intro-Anosov} is automatically satisfied).
However, Theorems \ref{thm:main} and~\ref{thm:main-negative} and Corollary~\ref{coro:cc-open} hold even when $\Gamma$ is not irreducible, as we shall prove~in~\cite{dgk-cc}.
\end{remark}

%%%%%%%%%%%%%%%%%%%%%%%%%
\subsection{Link with strong projective convex cocompactness}

Let $n\geq 2$.
A properly convex open subset $\Omega$ of $\PP(\RR^{n})$ is said to be \emph{strictly convex} if its boundary does not contain any nontrivial segment.
It is said to have \emph{$C^1$ boundary} if every point of the boundary of~$\Omega$ has a unique supporting hyperplane.
In \cite{cm14}, Crampon--Marquis introduced a notion of geometrically finite subgroup $\Gamma$ of $\PGL(\RR^n)$, requiring $\Gamma$ to preserve and act with various nice properties on a strictly convex open subset of $\PP(\RR^n)$ with $C^1$ boundary.
If cusps are not allowed, the notion reduces to a natural notion of convex cocompactness.
We will refer to this notion as \emph{strong convex cocompactness} to distinguish it from Definition~\ref{def:ccc-Hpq} and from a more general notion of convex cocompactness that we study in~\cite{dgk-cc}.

\begin{definition}[\cite{cm14}] \label{def:cm}
A discrete subgroup $\Gamma$ of $\PGL(\RR^n)$ is \emph{strongly convex cocompact in $\PP(\RR^n)$} if it preserves a nonempty strictly convex open subset $\Omega$ of $\PP(\RR^n)$ with $C^1$ boundary and if the convex hull of the orbital limit set $\Lambda^{\mathsf{orb}}_{\Omega}(\Gamma)$ in~$\Omega$ has compact quotient by $\Gamma$.
\end{definition}

Here we call \emph{orbital limit set} the set $\Lambda^{\mathsf{orb}}_{\Omega}(\Gamma)$ of accumulation points in $\partial_{\scriptscriptstyle\PP}\Omega$ of a $\Gamma$-orbit of~$\Omega$; it does not depend on the orbit since $\Omega$ is strictly convex (Lemma~\ref{lem:Lambda-orb}).
For strongly convex cocompact groups, this set coincides with the proximal limit set $\Lambda_{\Gamma}$ (Lemma~\ref{lem:snowman}).
When $\Gamma$ is finite, $\Lambda^{\mathsf{orb}}_{\Omega}(\Gamma)=\emptyset$.

In the setting of Definition~\ref{def:cm}, the action of $\Gamma$ on~$\Omega$ is automatically properly discontinuous (see Section~\ref{subsec:prop-conv-proj}), and so for torsion-free~$\Gamma$ the quotient $\Gamma\backslash\Omega$ is a real projective manifold.
The image in $\Gamma\backslash\Omega$ of the convex hull of~$\Lambda^{\mathsf{orb}}_{\Omega}(\Gamma)$ in~$\Omega$ is a \emph{compact convex core} for this manifold.
Such convex cocompact real projective manifolds $\Gamma\backslash\Omega$ provide a natural generalization of the compact real projective manifolds which have been classified by Goldman \cite{gol90} in dimension~$2$ and investigated by Benoist \cite{ben04,ben03,ben05,ben06} in higher dimension.

We make the following link between Definitions \ref{def:ccc-Hpq} and~\ref{def:cm}.

\begin{proposition} \label{prop:CramponMarquis}
Let $p,q\in\NN^*$ and let $\Gamma$ be an irreducible discrete subgroup of $G=\PO(p,q)$.
\begin{enumerate}[(1)]
  \item\label{item:CM1} If $\Gamma$ is $\HH^{p,q-1}$-convex cocompact, then it is strongly convex cocompact in $\PP(\RR^{p+q})$.
  Moreover, the set $\Omega$ of Definition~\ref{def:cm} may be taken to be contained in $\HH^{p,q-1}$.
  \item\label{item:CM2} Conversely, if $\Gamma$ is strongly convex cocompact in $\PP(\RR^{p+q})$, then it is $\HH^{p,q-1}$-convex cocompact or $\HH^{q,p-1}$-convex cocompact (after identifying $\PO(p,q)$ with $\PO(q,p)$).
\end{enumerate}
\end{proposition}

The following observation is an easy consequence of the definitions.
We refer to \cite{gw12} for the notion of $P_1$-Anosov representation into $\PGL(\RR^n)$, sometimes also known as \emph{projective Anosov representation}.

\begin{fact}[{\cite[Th.\,4.3]{gw12}}] \label{fact:Ano-PO-PGL}
Let $p,q\in\NN^*$ with $p+q=n$.
A representation with values in $\PO(p,q)$ is $P_1^{p,q}$-Anosov if and only if it is $P_1$-Anosov as a representation into $\PGL(\RR^n)$, where $P_1$ is the stabilizer of a line of~$\RR^n$.
\end{fact}

Therefore, Theorems \ref{thm:main} and~\ref{thm:main-negative} and Proposition~\ref{prop:CramponMarquis} give an intimate relationship between $P_1$-Anosov representations into $\PGL(\RR^n)$ and discrete subgroups of $\PGL(\RR^n)$ which are strongly convex cocompact in $\PP(\RR^n)$, in the context where there is an invariant quadratic form on~$\RR^n$.
In~\cite{dgk-cc}, we shall generalize this relationship to the setting of subgroups of $\PGL(\RR^n)$ which do not necessarily preserve any quadratic form: indeed, the arguments in the proofs of Theorems~\ref{thm:main} and~\ref{thm:main-negative} take place in projective geometry, and with some work we will be able to remove the use of the quadratic form.

%%%%%%%%%%%%%%%%%%%%%%%%%
\subsection{Examples of $\HH^{p,q-1}$-convex cocompact subgroups coming from Anosov representations} \label{subsec:intro-ex-cc}

Theorems \ref{thm:main} and~\ref{thm:main-negative} imply that many well-known examples of Anosov representations yield $\HH^{p,q-1}$-convex cocompact groups.
In Section~\ref{sec:examples} we describe examples, generalizing quasi-Fuchsian representations, that come from deformations of a convex cocompact subgroup of a rank-one Lie subgroup $H$ of~$G$.
These include certain maximal representations of surface groups and Hitchin representations.
Applying Proposition~\ref{prop:CramponMarquis}, all of these examples are new examples of discrete subgroups of $\PGL(\RR^n)$ which are strongly convex cocompact in $\PP(\RR^n)$.

Here is a sample result from Section~\ref{subsec:Hitchin}.

\begin{proposition} \label{prop:Hitchin}
Let $\Gamma$ be the fundamental group of a closed orientable hyperbolic surface, and let $m\geq 1$ and $\ell\in\{m,m+1\}$.

If $m$ is odd (\resp even), then the group $\rho(\Gamma)$ is $\HH^{m+1,\ell-1}$-convex cocompact (\resp $\HH^{\ell,m}$-convex cocompact) for any irreducible representation $\rho$ in the Hitchin component of $\Hom(\Gamma,\PO(m+1,\ell))$.
\end{proposition}

The result is true also for nonirreducible representations: see \cite{dgk-cc}.

%%%%%%%%%%%%%%%%%%%%%%%%%
\subsection{New examples of Anosov representations}

Conversely, Theorem~\ref{thm:main} also enables us to give new examples of Anosov representations into higher-rank semisimple Lie groups.
While Anosov representations of free groups and surface groups are abundant in the literature, the same is not true for Anosov representations of more complicated hyperbolic groups outside the realm of Kleinian groups.
We show that certain natural and explicit representations of hyperbolic right-angled Coxeter groups, namely deformations of the Tits canonical representation studied by Krammer \cite{kra94} and others (see \eg Dyer--Hohlweg--Ripoll~\cite{dhr16}), are $\HH^{p,q-1}$-convex cocompact for some appropriate pair $(p,q)$; therefore, by Theorem~\ref{thm:main}, they are $P_1^{p,q}$-Anosov.

\begin{theorem} \label{thm:exist-Ano-RACG}
Let $W$ be an infinite word hyperbolic right-angled Coxeter group in $n$ generators.
Then $W$ admits a $P_1^{p,q}$-Anosov representation into $\PO(p,q)$ for some $p,q\in\NN^*$ with $p+q=n$.
Composing with the inclusion $\PO(p,q) \hookrightarrow \PGL(\RR^n)$ gives a $P_1$-Anosov representation of $W$ into $\PGL(\RR^n)$ (Fact~\ref{fact:Ano-PO-PGL}).
\end{theorem}

The class of infinite hyperbolic right-angled Coxeter groups is quite large.
It includes groups of arbitrarily large virtual cohomological dimension \cite{js03,hag03,osa13}, which can have exotic Gromov boundaries such as the Menger curve, Pontryagin surfaces, Menger compacta, or the Sierpi\'nski carpet \cite{ben92,dra99,dra01,swi16}.

Theorem~\ref{thm:exist-Ano-RACG} also provides (by restriction to a subgroup or induction to a finite-index overgroup) Anosov representations for all groups commensurable to hyperbolic right-angled Coxeter groups, as well as for all their quasi-isometrically embedded subgroups.

%%%%%%%%%%%%%%%%%%%%%%%%%
\subsection{Organization of the paper}

In Section~\ref{sec:reminders} we recall some well-known facts about the space $\HH^{p,q-1}$ and properly convex domains in projective space.
In Section~\ref{sec:non-pos-sets} we give a characterization of negative subsets of $\partial_{\scriptscriptstyle\PP}\HH^{p,q-1}$, from which we deduce Proposition~\ref{prop:conn-transv} and Corollary~\ref{coro:cc-open}, and we establish some general properties of properly convex domains of $\PP(\RR^{p,q})$ preserved by discrete subgroups of $\PO(p,q)$.
Sections \ref{sec:cc->Ano} and~\ref{sec:Ano-neg->cc} are devoted to the proofs of implications \eqref{item:a1}~$\Rightarrow$~\eqref{item:a2} and~\eqref{item:a2}~$\Rightarrow$ \eqref{item:a1} of Theorem~\ref{thm:main-negative}, respectively.
In Section~\ref{sec:proofCM} we prove Proposition~\ref{prop:CramponMarquis}, which makes the link between our notion of $\HH^{p,q-1}$-convex cocompactness and strong convex cocompactness in $\PP(\RR^{p+q})$ (Definition~\ref{def:cm}).
In Section~\ref{sec:examples} we give examples of $\HH^{p,q-1}$-convex cocompact representations coming from well-known families of Anosov representations.
Finally, in Section~\ref{sec:ex-Ano} we construct $\HH^{p,q-1}$-convex cocompact right-angled Coxeter groups and prove Theorem~\ref{thm:exist-Ano-RACG}.
In Appendix~\ref{app:connectivity} we provide a proof of a (surely well known) basic result in point-set topology.

%%%%%%%%%%%%%%%%%%%%%%%%%
\subsection*{Acknowledgements}

We are grateful to Yves Benoist and Anna Wienhard for motivating comments and questions, and for their encouragement.
We also thank Vivien Ripoll for interesting discussions on limit sets of Coxeter groups, the referee for useful suggestions, and Jean-Philippe Burelle, Virginie Charette and Son Lam Ho for pointing out a subtlety in Section~4.
The main results, examples, and ideas of proofs in this paper were presented by the third-named author in June 2016 at the conference \emph{Geometries, Surfaces and Representations of Fundamental Groups} in honor of Bill Goldman's 60th birthday; we would like to thank the organizers for a very interesting and enjoyable conference.
Finally, we thank Bill for being a constant source of inspiration and encouragement to us and many others in the field.

%%%%%%%%%%%%%%%%%%%%%%%%%%%%%%%%%%%%%%%%%%%%%%%%%%%
\section{Reminders and basic facts} \label{sec:reminders}

%%%%%%%%%%%%%%%%%%%%%%%%%
\subsection{Pseudo-Riemannian hyperbolic spaces} \label{subsec:Hpq}

Fix two integers $p,q\geq 1$.
Let $G=\PO(p,q)$ and let $P_1^{p,q}$ be the stabilizer in~$G$ of an isotropic line of~$\RR^{p,q}$.
The projective space $\PP(\RR^{p,q})$ is the disjoint union of
$$\HH^{p,q-1} = \big\{[x]\in\PP(\RR^{p,q})\,|\,\langle x,x\rangle_{p,q}<0\big\},$$
of
$$\SS^{p-1,q} = \big\{[x]\in\PP(\RR^{p,q})\,|\,\langle x,x\rangle_{p,q}>0\big\},$$
and of
$$\partial_{\scriptscriptstyle\PP}\HH^{p,q-1} = \partial_{\scriptscriptstyle\PP}\SS^{p-1,q} = \big\{[x]\in\PP(\RR^{p,q})\,|\,\langle x,x\rangle_{p,q}=0\big\} \simeq G/P_1^{p,q}.$$
For instance, Figure~\ref{fig:H123} shows
$$\PP(\RR^4) = \HH^{3,0} \sqcup \left (\partial_{\scriptscriptstyle\PP}\HH^{3,0} = \partial_{\scriptscriptstyle\PP}\SS^{2,1}\right ) \sqcup \SS^{2,1}$$
and
$$\PP(\RR^4) = \HH^{2,1} \sqcup \left (\partial_{\scriptscriptstyle\PP}\HH^{2,1} = \partial_{\scriptscriptstyle\PP}\SS^{1,2}\right ) \sqcup \SS^{1,2}.$$
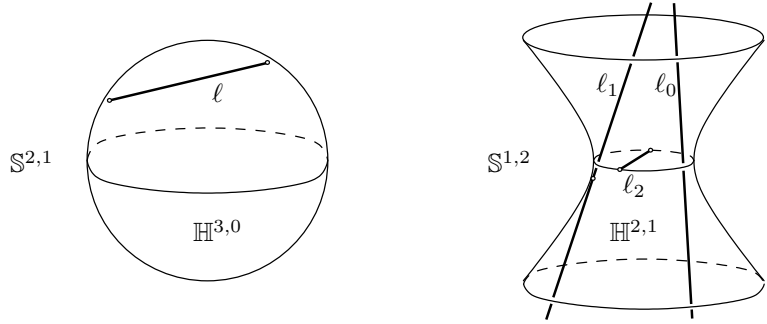
\begin{figure}[h!]
\centering
\labellist
\small\hair 2pt
\pinlabel $\SS^{1,2}$ [u] at 160 61
\pinlabel $\SS^{2,1}$ [u] at -20 61
\pinlabel $\HH^{3,0}$ [u] at 50 35
\pinlabel $\HH^{2,1}$ [u] at 207 35
\pinlabel $\ell$ [u] at 50 87
\pinlabel $\ell_2$ [u] at 207 52
\pinlabel $\ell_0$ [u] at 219 90 
\pinlabel $\ell_1$ [u] at 197 90
\endlabellist
\includegraphics[scale=1]{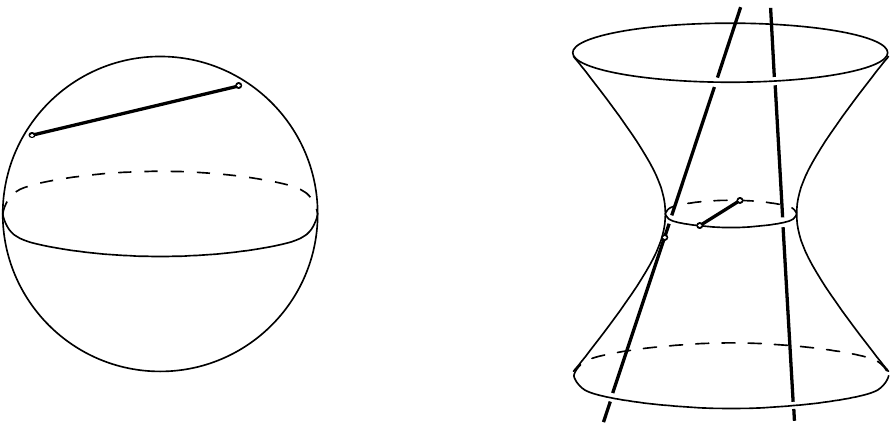}
\caption{Left: $\HH^{3,0}$ with a geodesic line $\ell$ (necessarily spacelike), and $\SS^{2,1}$. Right: $\HH^{2,1}$ with three geodesic lines $\ell_2$ (spacelike), $\ell_1$ (lightlike), and $\ell_0$ (timelike), and $\SS^{1,2}$.}
\label{fig:H123}
\end{figure}
The space $\HH^{p,q-1}$ is homeomorphic to $\RR^p\times\PP(\RR^q)$.
It has a natural pseudo-Riemannian structure of signature $(p,q-1)$ with isometry group~$G$.
To see this, consider the double covering
$$\widehat{\HH}^{p,q-1} = \big\{ x\in\RR^{p,q} ~|~ \langle x,x\rangle_{p,q}=-1\big\}.$$
The restriction of $\langle\cdot,\cdot\rangle_{p,q}$ to any tangent space to $\widehat{\HH}^{p,q-1}$ in $\RR^{p,q}$ has signature $(p,q-1)$ and defines a pseudo-Riemannian structure on~$\widehat{\HH}^{p,q-1}$ with isometry group $\OO(p,q)$, descending to a pseudo-Riemannian structure on~$\HH^{p,q-1}$ with isometry group $\PO(p,q)$.
The sectional curvature is constant negative for this pseudo-Riemannian structure.
The geodesic lines of the pseudo-Riemannian space $\HH^{p,q-1}$ are the intersections of $\HH^{p,q-1}$ with projective lines in $\PP(\RR^{p,q})$.
Such a line is called \emph{spacelike} (\resp \emph{lightlike}, \resp \emph{timelike}) if it meets $\partial_{\scriptscriptstyle\PP}\HH^{p,q-1}$ in two (\resp one, \resp zero) points: see Figure~\ref{fig:H123}.

Similarly, $\SS^{p-1,q}$ is homeomorphic to $\PP(\RR^p)\times\RR^q$ and has a natural pseudo-Riemannian structure of signature $(p-1,q)$ with isometry group~$G$, of constant positive curvature.
It identifies with $\HH^{q,p-1}$ as in Remark~\ref{rem:Spq-Hpq}.

\begin{remark} \label{rem:PO22}
For $(p,q)\notin\{(1,1),(2,2)\}$, the group $G=\PO(p,q)$ is simple and $P_1^{p,q}$ is a maximal proper parabolic subgroup of~$G$.
On the other hand, for $(p,q)=(1,1)$, the group $\PO(1,1)_0$ is isomorphic to $\RR_{>0}$ (hence reductive but not simple) and $P_1^{1,1}=\PO(1,1)_0$.
For $(p,q)=(2,2)$, the group $\PO(2,2)_0$ is isomorphic to $\PSL_2(\RR)\times\PSL_2(\RR)$ (hence semisimple but not simple) and its subgroup $P_1^{2,2}\cap\PO(2,2)_0$ is $B\times B$ where $B$ is a Borel subgroup of $\PSL_2(\RR)$.
To see this, observe that the space $\mathrm{M}_2(\RR)$ of $(2\times 2)$ real matrices is endowed with a natural nondegenerate quadratic form of signature $(2,2)$, namely the determinant; it thus identifies with $\RR^{2,2}$.
The group $\PO(2,2)_0$ acting on~$\PP(\RR^{2,2})$ identifies with the group $\PSL_2(\RR)\times\PSL_2(\RR)$ acting on $\PP(\mathrm{M}_2(\RR))$ by right and left multiplication.
The space $\partial_{\scriptscriptstyle\PP}\HH^{2,1} = \partial_{\scriptscriptstyle\PP}\SS^{1,2}$ identifies with the image in $\PP(\mathrm{M}_2(\RR))$ of the rank-one matrices in $\mathrm{M}_2(\RR)$, which identifies with $\PP^1\RR\times\PP^1\RR$ by taking the kernel and the image.
\end{remark}

The following notation, used in the introduction, will remain valid throughout the paper.

\begin{notation} \label{not:partial}
For a subset $X$ of $\PP(\RR^{p,q})$ (\eg a subset of $\HH^{p,q-1}$), we denote by
\begin{itemize}
  \item $\overline{X}$ the closure of $X$ in $\PP(\RR^{p,q})$;
  \item $\mathrm{Int}(X)$ the interior of $X$ in $\PP(\RR^{p,q})$ (or equivalently in $\HH^{p,q-1}$ if $X\subset\nolinebreak\HH^{p,q-1}$).
\end{itemize}
We set $\partial_{\scriptscriptstyle\PP} X:=\overline{X}\smallsetminus\mathrm{Int}(X)$ and, when $X\subset\HH^{p,q-1}$, we denote by
\begin{itemize}
  \item $\partial_{\scriptscriptstyle\HH} X:=\partial_{\scriptscriptstyle\PP} X\cap\HH^{p,q-1}$ the boundary of $X$ in $\HH^{p,q-1}$;
\item $\partial_i X:=\partial_{\scriptscriptstyle\PP} X \cap \partial_{\scriptscriptstyle\PP} \HH^{p,q-1}$ the boundary of $X$ in $\partial_{\scriptscriptstyle\PP} \HH^{p,q-1}$; if $X$ is closed in $\HH^{p,q-1}$, this coincides with the \emph{ideal boundary} of~$X$, namely $\overline{X}\smallsetminus X$.
\end{itemize}
We also denote by $\partial_\infty\Gamma$ the Gromov boundary of a word hyperbolic group~$\Gamma$.
\end{notation}

%%%%%%%%%%%%%%%%%%%%%%%%%
\subsection{Limit sets in projective space} \label{subsec:limit-set}

Let $V$ be a finite-dimensional real vector space of dimension $\geq 2$.
Recall that an element $g\in\PGL(V)$ is said to be \emph{proximal in $\PP(V)$} if it admits a unique attracting fixed point $\xi_g^+$ in $\PP(V)$.
Equivalently, $g$ has a unique complex eigenvalue of maximal modulus.

If $g$ is proximal in $\PP(V)$, then $g^{-1}$ is proximal in the dual projective space $\PP(V^*)$, for the dual action given by $g^{-1}\cdot \ell:=\ell\circ g$ for a linear form $\ell\in V^*$; the unique attracting fixed point of $g^{-1}$ in $\PP(V^*)$ corresponds to the projective hyperplane $H_g^-\subset\PP(V)$ which is the projectivization of the sum of the generalized eigenspaces of~$g$ for eigenvalues of nonmaximal modulus; we have $g^n\cdot y\to\xi_g^+$ for all $y\in\PP(V)\smallsetminus H_g^-$ as $n\to +\infty$.

We shall use the following terminology.

\begin{definition}\label{def:limit-set}
Let $\Gamma$ be a discrete subgroup of $\PGL(V)$.
The \emph{proximal limit set of $\Gamma$ in $\PP(V)$} is the closure $\Lambda_{\Gamma}$ of the set of attracting fixed points of elements of~$\Gamma$ which are proximal in $\PP(V)$.
\end{definition}

This set is a closed, $\Gamma$-invariant subset of $\PP(V)$.
When $\Gamma$ is irreducible and contains at least one proximal element, it was studied in \cite{gui90,ben97}.
In that setting, by \cite{ben97}, the action of $\Gamma$ on~$\Lambda_{\Gamma}$ is \emph{minimal}, \ie all $\Gamma$-orbits are dense; any nonempty, closed, $\Gamma$-invariant subset of $\PP(V)$ contains~$\Lambda_{\Gamma}$.

\begin{remark} \label{rem:lim-set-POpq}
For $p,q\in\NN^*$, an element $g\in G=\PO(p,q)$ is proximal in $\PP(\RR^{p,q})$ if and only if it is proximal in $\partial_{\scriptscriptstyle\PP}\HH^{p,q-1}$, in the sense that $g$ admits a unique attracting fixed point $\xi_g^+$ in $\partial_{\scriptscriptstyle\PP}\HH^{p,q-1}$.
In this case, $g^{-1}$ is automatically proximal too, and its unique attracting fixed point $\xi_g^-$ satisfies $\langle\xi_g^+,\xi_g^-\rangle_{p,q}\neq 0$.
In particular, for a discrete subgroup $\Gamma$ of $G=\PO(p,q)$, the proximal limit set $\Lambda_{\Gamma}$ of $\Gamma$ in $\PP(\RR^{p,q})$ is contained in $\partial_{\scriptscriptstyle\PP}\HH^{p,q-1}$, and called the proximal limit set of $\Gamma$ in $\partial_{\scriptscriptstyle\PP}\HH^{p,q-1}$.
\end{remark}

%%%%%%%%%%%%%%%%%%%%%%%%%
\subsection{Properly convex domains in projective space} \label{subsec:prop-conv-proj}

Let $\Omega$ be a properly convex open subset of $\PP(V)$, with boundary $\partial_{\scriptscriptstyle\PP}\Omega$.
Recall the \emph{Hilbert metric} $d_{\Omega}$ on~$\Omega$:
$$d_{\Omega}(y,z) := \frac{1}{2} \log \, [a,y,z,b] $$
for all distinct $y,z\in\Omega$, where $a,b$ are the intersection points of $\partial_{\scriptscriptstyle\PP}\Omega$ with the projective line through $y$ and~$z$, with $a,y,z,b$ in this order.
Here $[\cdot,\cdot,\cdot,\cdot]$ denotes the cross-ratio on $\PP^1\RR$, normalized so that $[0,1,t,\infty]=t$ for all~$t$. 
The metric space $(\Omega,d_{\Omega})$ is complete and proper, and the automorphism group
$$\mathrm{Aut}(\Omega) := \{g\in\PGL(V) ~|~ g\cdot\Omega=\Omega\}$$
acts on~$\Omega$ by isometries for~$d_{\Omega}$.
As a consequence, any discrete subgroup of $\mathrm{Aut}(\Omega)$ acts properly discontinuously on~$\Omega$.

Let $V^*$ be the dual vector space of~$V$.
By definition, the \emph{dual convex set} of~$\Omega$ is
$$\Omega^* := \PP\big(\big\{ \ell\in V^* ~|~ \ell(x)<0\quad \forall x\in\overline{\widetilde{\Omega}}\big\}\big),$$
where $\overline{\widetilde{\Omega}}$ is the closure in $V\smallsetminus \{0\}$ of an open convex cone $\widetilde{\Omega}$ of $V\smallsetminus\{0\}$ lifting~$\Omega$.
The set $\Omega^*$ is a properly convex open subset of $\PP(V^*)$ which is preserved by the dual action of $\mathrm{Aut}(\Omega)$ on $\PP(V^*)$.

Straight lines (contained in projective lines) are always geodesics for the Hilbert metric $d_{\Omega}$.
When $\Omega$ is not strictly convex, there can be other geodesics as well, by the following well-known and easy fact.

\begin{fact} \label{fact:triang-ineq-Hilbert}
For pairwise distinct points $w_1,w_2,w_3\in\Omega$, we have $d_{\Omega}(w_1,w_2)=d_{\Omega}(w_1,w_3)+d_{\Omega}(w_3,w_2)$ if and only if there are segments $[y,y']$ and $[z,z']$ in the boundary of~$\Omega$ such that $y,w_1,w_3,z$ on the one hand, and $y',w_3,w_2,z'$ on the other hand, are aligned in this order.
In this case, there exist points $y''\in [y,y']$ and $z''\in [z,z']$ such that $y'',w_1,w_2,z''$ are aligned in this order.
\end{fact}

(For an illustration, we refer to the left panel of Figure~\ref{fig:hairpin}, where $w_1,w_2,w_3$ correspond to $\mathcal{G}(s),\mathcal{G}(t),\mathcal{G}(u)$.)

However, the following fact is always true, and will be used in Section~\ref{subsec:C-hyperb}.
It is proved by Foertsch--Karlsson \cite[Th.\,3]{fk05}, as was pointed out to us by Constantin Vernicos.
Here we provide a short proof for the reader's convenience.

\begin{lemma} \label{lem:geod-loop}
\begin{enumerate}
  \item\label{item:end-geod-ray} Any geodesic ray of $(\Omega,d_{\Omega})$ has a well-defined endpoint in the boundary $\partial_{\scriptscriptstyle\PP}\Omega$.
  \item\label{item:end-biinf-geod} Any biinfinite geodesic of $(\Omega,d_{\Omega})$ has two distinct endpoints in $\partial_{\scriptscriptstyle\PP}\Omega$.
\end{enumerate}
\end{lemma}

\begin{proof}
We work in an affine Euclidean chart where $\Omega$ is bounded.
Let $I$ be $\RR_{\geq 0}$ or~$\RR$ and let $\mathcal{G}=(\mathcal{G}(t))_{t\in I}$ be a geodesic ray or biinfinite geodesic of $(\Omega,d_{\Omega})$.
For any $s<t$ in~$I$, let $y_{s,t}\in\partial_{\scriptscriptstyle\PP}\Omega$ and $z_{s,t}\in\partial_{\scriptscriptstyle\PP}\Omega$ be such that $y_{s,t},\mathcal{G}(s),\mathcal{G}(t),z_{s,t}$ are aligned in this order.

Recall that for any $y\in\partial_{\scriptscriptstyle\PP}\Omega$, the \emph{face} of $\partial_{\scriptscriptstyle\PP}\Omega$ at~$y$ is by definition the intersection of $\partial_{\scriptscriptstyle\PP} \Omega$ with all supporting hyperplanes to~$\Omega$ at~$y$.

We claim that all points $y_{s,t}$ for $s<t$ in~$I$ are contained in a common face $P$ of $\partial_{\scriptscriptstyle\PP}\Omega$ (of arbitrary dimension).
Indeed, for any $s<t$, let $P_{s,t}$ be the face of $\partial_{\scriptscriptstyle\PP}\Omega$ at~$y_{s,t}$; it is a nonempty compact convex subset of $\partial_{\scriptscriptstyle\PP}\Omega$.
For any $s < u < t$, observe that by  Fact~\ref{fact:triang-ineq-Hilbert} we have $y_{s,t}\in [y_{s,u},y_{u,t}]$ (see Figure~\ref{fig:hairpin}, left), hence any supporting hyperplane to $\Omega$ at $y_{s,t}$ is also a supporting hyperplane to $\Omega$ at $y_{s,u}$ and $y_{u,t}$, \ie $P_{s,u}$ and $P_{u,t}$ are contained in $P_{s,t}$.
For any four points $s < u < v <t$, we apply the previous statement to the triples $(u,v,t)$ and $(s,u,t)$ to obtain $P_{u,v} \subset P_{u,t} \subset P_{s,t}$.
Thus the sequence $(P_{0,m})_{m\geq 1}$ (if $I=\RR_{\geq 0}$) or $(P_{-m,m})_{m\geq 1}$ (if $I=\RR$) of nonempty compact convex subsets of $\partial_{\scriptscriptstyle\PP}\Omega$ is nondecreasing for inclusion; it must have a limit~$P$.

Similarly, all points $z_{s,t}$ for $s<t$ are contained in a common face $Q$ of~$\partial_{\scriptscriptstyle\PP}\Omega$.

Any forward accumulation point $a$ of $\mathcal{G}$ in the boundary of $\Omega$ is an accumulation point of the $z_{0,t}$ as $t\to +\infty$, and therefore belongs to~$Q$.
In the Euclidean metric, $z_{0,t}$ for $t>0$ is further away from the span of $P$ than $\mathcal{G}(0)$ is: therefore $a$ belongs in fact to $Q\smallsetminus P$.
(In particular $P\neq Q$.) 
Similarly, in the case $I=\RR$, all backward accumulation points of $\mathcal{G}$ are in $P\smallsetminus Q$.

Suppose by contradiction that there are two sequences $(s_m)_{m\in\NN}$ and $(t_m)_{m\in\NN}$ of positive numbers tending to $+\infty$ such that $\mathcal{G}(s_m)$ and $\mathcal{G}(t_m)$ converge respectively to some points $a\neq b$ in $Q\smallsetminus P$.
Up to taking subsequences, we may assume that $s_m<t_m<s_{m+1}$ for all~$m$.

Consider the triangle $T_m$ spanned by $y_{s_m,t_m}$, $\mathcal{G}(t_m)$, and $y_{t_m,s_{m+1}}$ (see Figure~\ref{fig:hairpin}, right).
Its angle at $\mathcal{G}(t_m)$ goes to $\pi$ in the chosen chart as $m\to +\infty$, because $\mathcal{G}(s_m), \mathcal{G}(s_{m+1})\to a$ and $\mathcal{G}(t_m) \to b$.
The opposite edge $[y_{s_m,t_m}, y_{t_m,s_{m+1}}]$ of $T_m$ converges to a segment of~$P$ as $m\to +\infty$: therefore $\lim_m \mathcal{G}(t_m)\in P$.
But $\mathcal{G}(t_m) \to b \notin P$: contradiction. 
Thus $\mathcal{G}$ has a unique forward endpoint $a$ in the boundary of~$\Omega$, belonging to $Q\smallsetminus P$.
Similarly, in the case $I=\RR$, it has a unique backward endpoint $a'\neq a$ in $P\smallsetminus Q$.
\end{proof}

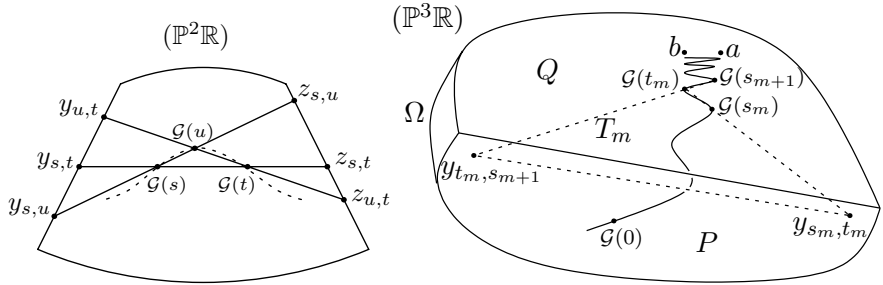
\begin{figure}[h]
\centering
\labellist
\small\hair 2pt
\pinlabel $y_{u,t}$ [u] at 15 60
\pinlabel $y_{s,t}$ [u] at 7 42
\pinlabel $y_{s,u}$ [u] at -2 26
\pinlabel $z_{s,u}$ [u] at 97 65
\pinlabel $z_{s,t}$ [u] at 109 42
\pinlabel $z_{u,t}$ [u] at 115 29
\pinlabel $\scriptstyle\mathcal{G}(s)$ [u] at 45 34
\pinlabel $\scriptstyle\mathcal{G}(u)$ [u] at 54 52
\pinlabel $\scriptstyle\mathcal{G}(t)$ [u] at 69 34
\pinlabel $(\PP^2\RR)$ [u] at 55 85
\endlabellist
\includegraphics[scale=1.1]{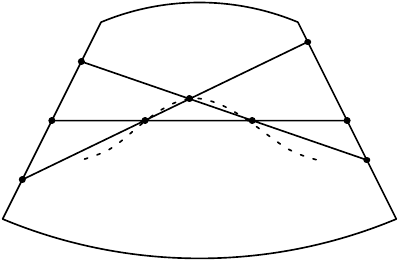}
\hspace{0.5cm}
\labellist
\pinlabel $y_{s_m,t_m}$ [u] at 152 21
\pinlabel $y_{t_m,s_{m+1}}$ [u] at 24 42
\pinlabel $T_m$ [u] at 70 58
\pinlabel $\scriptstyle\mathcal{G}(0)$ [u] at 73 17
\pinlabel $a$ [u] at 115 88
\pinlabel $b$ [u] at 93 90
\pinlabel $P$ [u] at 105 16
\pinlabel $Q$ [u] at 45 80
\pinlabel $\Omega$ [u] at -5 65
\pinlabel $\scriptstyle\mathcal{G}(s_m)$ [u] at 121 67
\pinlabel $\scriptstyle\mathcal{G}(s_{m+1})$ [u] at 126 78
\pinlabel $\scriptstyle\mathcal{G}(t_m)$ [u] at 84 77
\pinlabel $(\PP^3\RR)$ [u] at 0 100
\endlabellist
\includegraphics[scale=1]{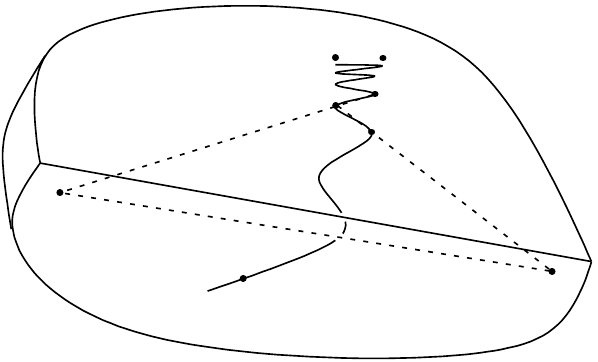}
\caption{Illustration for the proof of Lemma~\ref{lem:geod-loop}. Left: definition of the points $y_{s,t}$ and $z_{s,t}$. Right: absurd situation where $\mathcal{G}$ would have two forward accumulation points $a\neq b$.}
\label{fig:hairpin}
\end{figure}

In Section~\ref{subsec:proofCM1} we shall also use the following elementary observation.

\begin{lemma} \label{lem:Conv-C1}
Let $\mathcal{O},\mathcal{O}_1,\dots,\mathcal{O}_k$ be properly convex open subsets of $\PP(V)$ such that $\mathcal{O}$ is the convex hull of the~$\mathcal{O}_i$ in some affine chart of $\PP(V)$.
If the boundary $\partial_{\scriptscriptstyle\PP}\mathcal{O}_i$ is $C^1$ for every~$i$, then so is the boundary $\partial_{\scriptscriptstyle\PP}\mathcal{O}$.
\end{lemma}

\begin{proof}
For any~$i$, the boundary $\partial_{\scriptscriptstyle\PP}\mathcal{O}_i$ is $C^1$ if and only if the dual convex set $\mathcal{O}_i^*$ is strictly convex.
In this case, the intersection $\bigcap_{i=1}^k \mathcal{O}_i^*$ is also strictly convex.
But this intersection is the dual $\mathcal{O}^*$ of~$\mathcal{O}$.
Therefore, $\partial_{\scriptscriptstyle\PP}\mathcal{O}$ is~$C^1$.
\end{proof}

%%%%%%%%%%%%%%%%%%%%%%%%%
\subsection{Irreducible groups preserving properly convex domains}

We shall use the following general properties due to Benoist \cite[Prop.\,3.1]{ben00}.
We denote by $\Lambda_{\Gamma}^*$ the proximal limit set of $\Gamma$ in $\PP(V^*)$.

\begin{fact}[\cite{ben00}] \label{fact:Yves}
An irreducible discrete subgroup $\Gamma$ of $\PGL(V)$ preserves a nonempty properly convex subset of $\PP(V)$ if and only if the following two conditions are satisfied:
\begin{enumerate}[(i)]
  \item\label{item:i-lifts} $\Gamma$ contains an element of $\PGL(V)$ which is proximal both in $\PP(V)$ and in $\PP(V^*)$,
  \item\label{item:ii-lifts} $\Lambda_{\Gamma}$ and~$\Lambda_{\Gamma}^*$ lift respectively to cones $\widetilde{\Lambda}_{\Gamma}$ of $V\smallsetminus\{0\}$ and $\widetilde{\Lambda}_{\Gamma}^*$ of $V^*\smallsetminus\{0\}$ such that $\ell(x)\leq 0$ for all $x\in\widetilde{\Lambda}_{\Gamma}$ and $\ell\in\widetilde{\Lambda}_{\Gamma}^*$.
\end{enumerate}
In this case, for any $\Gamma$-invariant properly convex open subset $\Omega\neq\emptyset$ of $\PP(V)$,
\begin{enumerate}
  \item the proximal limit set $\Lambda_{\Gamma}$ (\resp $\Lambda_{\Gamma}^*$) is contained in the boundary of $\Omega$ (\resp $\Omega^*$) in $\PP(V)$ (\resp $\PP(V^*)$);
  \item\label{item:2-lifts} more specifically, $\Omega$ and $\Lambda_{\Gamma}$ lift to cones $\widetilde{\Omega}$ and $\widetilde{\Lambda}_\Gamma$ of $V\smallsetminus\{0\}$ with $\widetilde{\Omega}$ properly convex containing $\widetilde{\Lambda}_{\Gamma}$ in its boundary, and $\Omega^*$ and $\Lambda_{\Gamma}^*$ lift to cones $\widetilde{\Omega}^*$ and $\widetilde{\Lambda}_\Gamma^*$ of $V^*\smallsetminus\{0\}$ with $\widetilde{\Omega}^*$ properly convex containing $\widetilde{\Lambda}_{\Gamma}^*$ in its boundary, such that $\ell(x)\leq 0$ for all $x\in\widetilde{\Lambda}_{\Gamma}$ and $\ell\in\widetilde{\Lambda}_{\Gamma}^*$;
  \item there is a unique smallest nonempty $\Gamma$-invariant properly convex open subset $\Omega_{\min}$ of $\PP(V)$ contained in~$\Omega$, namely the projectivization of the interior of the $\RR^+$-span of $\widetilde{\Lambda}_{\Gamma}$ for $\widetilde{\Lambda}_{\Gamma}$ as in (\ref{item:2-lifts}); it is the interior of the convex hull of $\Lambda_{\Gamma}$ in~$\Omega$;
  \item there is a unique largest $\Gamma$-invariant properly convex open subset $\Omega_{\max}$ of $\PP(V)$ containing~$\Omega$, namely the dual convex set to the projectivization of the interior of the $\RR^+$-span of $\widetilde{\Lambda}_{\Gamma}^*$ for $\widetilde{\Lambda}_{\Gamma}^*$ as in (\ref{item:2-lifts}).
\end{enumerate}
\end{fact}

\begin{remark} \label{rem:inv-dom}
When $V=\RR^{p,q}$ and $\Gamma$ is contained in $\PO(p,q)$ for some $p,q\in\NN^*$, the map $x\mapsto\langle x,\cdot\rangle_{p,q}$ from $V$ to~$V^*$ induces a homeomorphism $\Lambda_{\Gamma}\simeq\Lambda_{\Gamma}^*$ and $\Omega_{\max}$ is a connected component of the complement in~$V$ of the union of the projective hyperplanes $z^{\perp}$ for $z\in\Lambda_{\Gamma}$.
In the Lorentzian case (\ie $q=2$), the terminology \emph{invisible domain of~$\Lambda_{\Gamma}$} is often used for $\Omega_{\max}$.
\end{remark}

%%%%%%%%%%%%%%%%%%%%%%%%%
\subsection{Strictly convex open domains}

We now make two elementary observations about strictly convex domains.

\begin{lemma} \label{lem:Lambda-orb}
Let $\Gamma$ be a discrete subgroup of $\PGL(V)$ preserving a nonempty strictly convex open subset $\Omega$ of $\PP(V)$.
The set $\Lambda^{\mathsf{orb}}_{\Omega}(\Gamma)$ of accumulation points in $\partial_{\scriptscriptstyle\PP}\Omega$ of a $\Gamma$-orbit of~$\Omega$ does not depend on the $\Gamma$-orbit.
\end{lemma}

We call this set the \emph{orbital limit set} of $\Gamma$ in~$\Omega$.

\begin{proof}
We may assume that $\Gamma$ is infinite, otherwise $\Lambda^{\mathsf{orb}}_{\Omega}(\Gamma)=\emptyset$.
Consider two points $y\neq z$ in~$\Omega$ and a sequence $(\gamma_m)\in\Gamma^{\NN}$ of pairwise distinct elements such that $\gamma_m\cdot y\to y_{\infty}\in\partial_{\scriptscriptstyle\PP}\Omega$ and $\gamma_m\cdot z\to z_{\infty}\in\partial_{\scriptscriptstyle\PP}\Omega$.
By proper discontinuity of the action of $\Gamma$ on~$\Omega$ (see Section~\ref{subsec:prop-conv-proj}), the Hilbert distance $d_{\Omega}$ from a given point of~$\Omega$ to any point of the segment $\gamma_m\cdot [y,z]$ tends to infinity with~$m$.
Therefore the segment $[y_{\infty},z_{\infty}]=\lim_m \gamma_m\cdot [y,z]$ is fully contained in $\partial_{\scriptscriptstyle\PP}\Omega$.
Since $\Omega$ is strictly convex, $y_{\infty}=z_{\infty}$.
\end{proof}

\begin{lemma}\label{lem:snowman}
Let $\Gamma$ be a discrete subgroup of $\PGL(V)$ preserving a nonempty strictly convex open subset $\Omega$ of $\PP(V)$.
If the proximal limit set $\Lambda_\Gamma$ of $\Gamma$ in $\PP(V)$ contains at least two points, then it coincides with the orbital limit set $\Lambda^{\mathsf{orb}}_\Omega(\Gamma)$.
\end{lemma}

\begin{proof}
For any element $\gamma\in\Gamma$ which is proximal in $\PP(V)$, the element $\gamma^{-1}$ is proximal in $\PP(V^*)$, and its attracting fixed point is contained in $\partial\Omega^*$ by Fact~\ref{fact:Yves}.
Thus the projective hyperplane $H_{\gamma}^-\subset\PP(V)$ from Section~\ref{subsec:limit-set} does not meet~$\Omega$, and for any $y\in\Omega$ the sequence $(\gamma^m\cdot y)_{m\in\NN}$ converges to the attracting fixed point $\xi_{\gamma}^+$ of $\gamma$ in $\PP(V)$.
This shows that $\Lambda_\Gamma \subset \Lambda^{\mathsf{orb}}_\Omega(\Gamma)$.

For the reverse inclusion, suppose that $\Lambda_\Gamma$ contains two points $a\neq b$.
In particular, $\Gamma$ is infinite.
Consider a point $y \in \Omega$ on the open segment $(a,b)$ and a sequence $(\gamma_m)\in\Gamma^{\NN}$ of pairwise distinct elements such that $\gamma_m\cdot y\to y_{\infty}\in\partial_{\scriptscriptstyle\PP}\Omega$.
Up to passing to a subsequence, we may assume that $\gamma_m\cdot a \to a_\infty$ and $\gamma_m\cdot b \to b_\infty$ for some $a_\infty,b_\infty \in \Lambda_{\Gamma}$, with $y_\infty \in [a_\infty, b_\infty]$.
Since $\Omega$ is strictly convex, $y_\infty\in \{a_\infty, b_\infty\} \subset \Lambda_{\Gamma}$.
Thus $\Lambda_\Omega^{\mathsf{orb}}(\Gamma) \subset \Lambda_\Gamma$.
\end{proof}

If $\Gamma$ is irreducible, then $\Lambda_{\Gamma}$ always contains at least two points since it is nonempty (Fact~\ref{fact:Yves}) and not contained in a projective subspace of $\PP(V)$ of positive codimension.

%%%%%%%%%%%%%%%%%%%%%%%%%%%%%%%%%%%%%%%%%%%%%%%%%%%
\section{Nonpositive subsets of $\partial_{\scriptscriptstyle\PP}\HH^{p,q-1}$} \label{sec:non-pos-sets}

We shall use the following terminology which extends Definition~\ref{def:pos-neg}.

\begin{definition} \label{def:non-pos-neg}
For $p,q\in\NN^*$, a subset $\Lambda$ of $\partial_{\scriptscriptstyle\PP}\HH^{p,q-1}$ is \emph{negative} (\resp \emph{nonpositive}, \emph{nonnegative}, \emph{positive}) if it lifts to a cone of $\RR^{p,q}\smallsetminus\{ 0\}$ on which all inner products $\langle\cdot,\cdot\rangle_{p,q}$ of noncollinear points are negative (\resp nonpositive, nonnegative, positive).
\end{definition}

In the Lorentzian case ($q=2$), the usual terminology for a negative (\resp nonpositive) subset of $\partial_{\scriptscriptstyle\PP}\HH^{p,q-1}$ is \emph{acausal} (\resp \emph{achronal}).

%%%%%%%%%%%%%%%%%%%%%%%%%
\subsection{Reading the sign on triples}

The following characterization will be used only to prove Proposition~\ref{prop:conn-transv} and Corollary~\ref{coro:cc-open}, in Section~\ref{subsec:proof-conn-transv} below.

\begin{lemma} \label{lem:justif-neg-triple}
Let $\Lambda$ be a subset of $\partial_{\scriptscriptstyle\PP}\HH^{p,q-1}$ with at least three points.
Then the following are equivalent:
\begin{enumerate}[(i)]
  \item\label{item:i-triple} $\Lambda$ is negative,
  \item\label{item:ii-triple} every triple of distinct points of~$\Lambda$ is negative,
  \item\label{item:iii-triple} every triple of distinct points of~$\Lambda$ spans a triangle fully contained in $\HH^{p,q-1}$ outside of the vertices.
\end{enumerate}
\end{lemma}

The equivalence \eqref{item:ii-triple}~$\Leftrightarrow$~\eqref{item:iii-triple} is contained in the following immediate remark.

\begin{remark} \label{rem:neg-triple-explicit}
For any pairwise distinct points $y_1,y_2,y_3$ of $\partial_{\scriptscriptstyle\PP}\HH^{p,q-1}$, the following are equivalent:
\begin{itemize}
  \item there exist lifts $x_i\in\RR^{p,q}\smallsetminus\{0\}$ of~$y_i$ such that $\langle x_i,x_j\rangle_{p,q}<0$ for all $i\neq j$,
  \item for any lifts $x_i\in\RR^{p,q}\smallsetminus\{0\}$ of~$y_i$, we have
  $$\langle x_1,x_2\rangle_{p,q} \, \langle x_1,x_3\rangle_{p,q} \, \langle x_2,x_3\rangle_{p,q} <\nolinebreak 0,$$
  \item there exist lifts $x_i\in\RR^{p,q}\smallsetminus\{0\}$ of~$y_i$ such that for any $t_i\geq 0$, if at least two of the $t_i$ are nonzero, then
  $$\Big\langle \sum_{i=1}^3 t_i x_i, \sum_{i=1}^3 t_i x_i \Big\rangle_{p,q} = \sum_{1\leq i<j\leq 3} 2 t_i t_j \langle x_i,x_j\rangle_{p,q} < 0,$$
  \item $(y_1,y_2,y_3)$ spans a triangle fully contained in $\HH^{p,q-1}$ outside of the vertices.
\end{itemize}
\end{remark}

\begin{proof}[Proof of Lemma~\ref{lem:justif-neg-triple}]
If $\Lambda$ is negative, then any subset of~$\Lambda$ is as well, and so \eqref{item:i-triple}~$\Rightarrow$~\eqref{item:ii-triple} holds.
We now check \eqref{item:ii-triple}~$\Rightarrow$~\eqref{item:i-triple}.

Suppose that every triple of distinct points of~$\Lambda$ is negative.
Choose two distinct points $y_1,y_2\in\Lambda$ and respective lifts $x_1,x_2\in\RR^{p,q}\smallsetminus\{0\}$ with $\langle x_1,x_2\rangle_{p,q}<0$.
We now define a map $f : \Lambda\to\RR^{p,q}\smallsetminus\{0\}$ as follows.
We set $f(y_i):=x_i$ for $i\in\{1,2\}$.
For each $y\in\Lambda\smallsetminus\{y_1,y_2\}$, we choose a lift $x\in\RR^{p,q}\smallsetminus\{0\}$ of~$y$; by Remark~\ref{rem:neg-triple-explicit}, we have $\langle x_1,x_2\rangle_{p,q} \, \langle x_1,x\rangle_{p,q} \, \langle x_2,x\rangle_{p,q} <\nolinebreak 0$, and so $\langle x_1,x\rangle_{p,q}$ and $\langle x_2,x\rangle_{p,q}$ are both nonzero of the same sign; we set $f(y):=x$ if this sign is negative, and $f(y):=-x$ otherwise.
We claim that $\langle f(y),f(y')\rangle_{p,q}<0$ for any $y\neq y'$ in~$\Lambda$.
Indeed, this is true by construction if $y$ or~$y'$ is equal to $y_1$, so we assume this is not the case.
By Remark~\ref{rem:neg-triple-explicit}, we have $\langle x_1,f(y)\rangle_{p,q} \, \langle x_1,f(y')\rangle_{p,q} \, \langle f(y),f(y')\rangle_{p,q} <\nolinebreak 0$.
Since $\langle x_1,f(y)\rangle_{p,q}<0$ and $\langle x_1,f(y')\rangle_{p,q}<0$ by construction, we have $\langle f(y),f(y')\rangle_{p,q}<0$.
Thus $\{ tf(y) \,|\, t>0\,\mathrm{and}\,y\in\Lambda\}$ is a cone of $\RR^{p,q}\smallsetminus\{0\}$ lifting~$\Lambda$ on which all inner products $\langle\cdot,\cdot\rangle_{p,q}$ of noncollinear points are negative, and so $\Lambda$ is negative.
\end{proof}

\begin{remarks} \label{rem:pos-neg}
\begin{enumerate}
  \item\label{item:other-equiv-pos-neg} Similar equivalences to Lemma~\ref{lem:justif-neg-triple} hold after replacing \emph{negative} with \emph{positive} in conditions \eqref{item:i-triple} and~\eqref{item:ii-triple}, and $\HH^{p,q-1}$ with $\SS^{p-1,q}$ in condition~\eqref{item:iii-triple}.
  \item\label{item:not-neg-and-pos} It follows from Remark~\ref{rem:neg-triple-explicit} that a triple of distinct points of $\partial_{\scriptscriptstyle\PP}\HH^{p,q-1}$ cannot be both negative and positive.
Therefore, by Lemma~\ref{lem:justif-neg-triple}, an arbitrary subset of $\partial_{\scriptscriptstyle\PP}\HH^{p,q-1}$ with at least three points cannot be both negative and positive.
  \item Our notion of positivity should not be confused with that of \cite{ben00}.
  For $V=\RR^{p,q}$, consider a subset of $\PP(V)\times\PP(V^*)$ of the form $\underline{\Lambda}=\{ (z,z^{\perp})\,|\,z\in\Lambda\}$ where $\Lambda\subset\partial_{\scriptscriptstyle\PP}\HH^{p,q-1}$.
  The set $\underline{\Lambda}$ is \emph{positive} in the sense of \cite{ben00} if and only if $\Lambda$ is nonnegative or nonpositive in the sense of Definition~\ref{def:non-pos-neg}.
  The set $\underline{\Lambda}$ is \emph{$3$-by-$3$ positive} in the sense of \cite{ben00} if and only if every triple of points of~$\Lambda$ is nonnegative or every triple of distinct points of $\Lambda$ is nonpositive in the sense of Definition~\ref{def:non-pos-neg}.
  We shall not use the terminology of \cite{ben00} in this paper.
\end{enumerate}
\end{remarks}

%%%%%%%%%%%%%%%%%%%%%%%%%
\subsection{Consequences of Lemma~\ref{lem:justif-neg-triple}} \label{subsec:proof-conn-transv}

\begin{proof}[Proof of Proposition~\ref{prop:conn-transv}]
Let $\Lambda$ be a closed, connected, transverse subset of $\partial_{\scriptscriptstyle\PP}\HH^{p,q-1}$, and let $\Lambda^{(3)}$ be the set of unordered triples of distinct points of~$\Lambda$.
By Remark~\ref{rem:neg-triple-explicit}, we may define a function $\Lambda^{(3)}\to\{\pm 1\}$ by sending $(y_1,y_2,y_3)\in\Lambda^{(3)}$ to the sign of $\langle x_1,x_2\rangle_{p,q} \, \langle x_1,x_3\rangle_{p,q} \, \langle x_2,x_3\rangle_{p,q}$, where $x_i\in\RR^{p,q}\smallsetminus\{0\}$ is an arbitrary lift of~$y_i$.
This function is continuous and $\Lambda^{(3)}$ is connected by Fact~\ref{fact:conn-ktuple}, hence the function is constant.
In other words, every triple of distinct points of~$\Lambda$ is negative or every triple of distinct points of~$\Lambda$ is positive.
By Lemma~\ref{lem:justif-neg-triple} and Remark~\ref{rem:pos-neg}.\eqref{item:other-equiv-pos-neg}, the set $\Lambda$ is negative or positive.
\end{proof}

Here is another consequence of Lemma~\ref{lem:justif-neg-triple}.

\begin{proposition} \label{prop:comp-Ano-neg}
Let $\Gamma$ be a word hyperbolic group and $\mathcal{T}$ a connected component in the space of $P_1^{p,q}$-Anosov representations of $\Gamma$ into $G=\PO(p,q)$.
If the proximal limit set $\Lambda_{\rho(\Gamma)}$ is negative (\resp positive) for some $\rho\in\mathcal{T}$, then it is negative (\resp positive) for all $\rho\in\mathcal{T}$.
\end{proposition}

\begin{proof}
We may assume $\#\partial_{\infty}\Gamma\geq 3$, otherwise for any $P_1^{p,q}$-Anosov representation $\rho : \Gamma\to G$ the proximal limit set $\Lambda_{\rho(\Gamma)}$ is both negative and positive.

Suppose $\Lambda_{\rho_0(\Gamma)}$ is negative for some $\rho_0\in\mathcal{T}$, and let $(\rho_t)_{t\in [0,1]}$ be a continuous path in~$\mathcal{T}$.
For $t\in [0,1]$, let $\xi_t : \partial_{\infty}\Gamma\to\partial_{\scriptscriptstyle\PP}\HH^{p,q-1}$ be the boundary map of the Anosov representation~$\rho_t$.
For any triple $\{\eta_1,\eta_2,\eta_3\}$ of distinct points of $\partial_{\infty}\Gamma$ and any $t\in [0,1]$, the triple $\{\xi_t(\eta_1),\xi_t(\eta_2),\xi_t(\eta_3)\}\subset\Lambda_{\rho_t(\Gamma)}$ is either negative or positive, by transversality of~$\xi_t$.
Since $\{\xi_0(\eta_1),\xi_0(\eta_2),\xi_0(\eta_3)\}$ is negative and $t\mapsto\xi_t(\eta_i)$ is continuous for all~$i$ (see \cite[Th.\,5.13]{gw12}), we deduce as in the proof of Proposition~\ref{prop:conn-transv} that $\{\xi_t(\eta_1),\xi_t(\eta_2),\xi_t(\eta_3)\}$ is negative for all $t\in [0,1]$.
By Lemma~\ref{lem:justif-neg-triple}, the set $\Lambda_{\rho_t(\Gamma)}$ is negative for all $t\in [0,1]$.

The case that $\Lambda_{\rho_0(\Gamma)}$ is positive is similar.
\end{proof}

%%%%%%%%%%%%%%%%%%%%%%%%%
\subsection{Boundaries of convex subsets of $\HH^{p,q-1}$}

The following lemma makes a link between convexity in $\HH^{p,q-1}$ and nonpositivity in $\partial_{\scriptscriptstyle\PP}\HH^{p,q-1}$.

\begin{lemma} \label{lem:nonpos-Hpq}
\begin{enumerate}[(1)]
  \item\label{item:nonpos-Hpq-Lambda0} Let $\Lambda_0$ be a closed nonpositive (\resp nonnegative) subset of $\partial_{\scriptscriptstyle\PP}\HH^{p,q-1}$ which is \emph{not} contained in a projective hyperplane.
Then $\Lambda_0$ spans a nonempty convex open subset $\Omega$ of $\PP(\RR^{p,q})$ which is contained in $\HH^{p,q-1}$ (\resp in $\SS^{p-1,q}$).
Moreover, if $\Lambda_1\supset\Lambda_0$ is the intersection of $\partial_{\scriptscriptstyle\PP}\HH^{p,q-1}$ with the closure of~$\Omega$, then $\Lambda_1$ is still nonpositive (\resp nonnegative), and it is equal to~$\Lambda_0$ if $\Lambda_0$ is negative (\resp positive).
  \item\label{item:nonpos-Hpq-Omega} Conversely, for any nonempty properly convex open subset $\Omega$ of $\HH^{p,q-1}$ (\resp $\SS^{p-1,q}$), the intersection of $\partial_{\scriptscriptstyle\PP}\HH^{p,q-1}$ with the closure of~$\Omega$ is nonpositive (\resp nonnegative).
\end{enumerate}
\end{lemma}

\begin{proof}
\eqref{item:nonpos-Hpq-Lambda0} Let $\widetilde{\Lambda}_0$ be a cone of $\RR^{p,q}\smallsetminus\{ 0\}$ on which $\langle\cdot,\cdot\rangle_{p,q}$ is nonpositive (the nonnegative case is similar).
Using the equality
\begin{equation} \label{eqn:polar}
\Big\langle \sum_i t_i x_i, \sum_i t_i x_i \Big\rangle_{p,q} = \sum_{i,j} t_i t_j \langle x_i, x_j\rangle_{p,q}
\end{equation}
for $t_i\in\RR^+$ and $x_i\in\widetilde{\Lambda}_0$, we see that $\langle\cdot,\cdot\rangle_{p,q}$ is still nonpositive on the $\RR^+$-span of~$\widetilde{\Lambda}_0$.
In particular, $\Lambda_1$ is nonpositive since it is contained in the projectivization of the $\RR^+$-span of~$\widetilde{\Lambda}_0$.
Let $\Omega$ be the projectivization of the interior of this $\RR^+$-span.
Then $\Omega$ is convex, contained in $\HH^{p,q-1}\cup\partial_{\scriptscriptstyle\PP}\HH^{p,q-1}$ and open, hence contained in $\HH^{p,q-1}$ (since $\partial_{\scriptscriptstyle\PP}\HH^{p,q-1}$ is a hypersurface of $\PP(\RR^{p,q})$).
Suppose $\Lambda_0$ is negative, \ie all inner products $\langle\cdot,\cdot\rangle_{p,q}$ of noncollinear points of $\widetilde{\Lambda}_0$ are negative.
Any point $z\in\Lambda_1$ admits a lift to $\RR^{p,q}\smallsetminus\{0\}$ of the form $\sum_{i=1}^k t_i x_i$ where $x_1,\dots,x_k\in\widetilde{\Lambda}_0$ are pairwise noncollinear and $t_1,\dots,t_k\geq 0$.
Since $z\in\partial_{\scriptscriptstyle\PP}\HH^{p,q-1}$, we see from \eqref{eqn:polar} that $\langle x_i,x_j\rangle_{p,q}=0$ for all $1\leq i<j\leq k$, hence $k=1$ and $z\in\Lambda_0$.
This shows that the inclusion $\Lambda_1\supset\Lambda_0$ is an equality when $\Lambda_0$ is negative.

\eqref{item:nonpos-Hpq-Omega} Let $\Omega$ be a nonempty properly convex open subset of $\HH^{p,q-1}$ (the $\SS^{p-1,q}$ case is similar).
We can lift it to a properly convex open cone $\widetilde{\Omega}$ of $\RR^{p,q}\smallsetminus\{0\}$ such that $\langle x,x\rangle_{p,q}<0$ for all $x\in\widetilde{\Omega}$.
Let $\Lambda_1$ be the intersection of $\partial_{\scriptscriptstyle\PP}\HH^{p,q-1}$ with the closure of~$\Omega$, and let $\widetilde{\Lambda}_1$ be a cone of $\RR^{p,q}\smallsetminus\{0\}$ lifting~$\Lambda_1$, contained in the closure of~$\widetilde{\Omega}$.
Using the equality \eqref{eqn:polar} for $t_i\in\RR^+$ and $x_i\in\widetilde{\Lambda}_1$, we see that $\langle\cdot,\cdot\rangle_{p,q}$ is nonpositive on~$\widetilde{\Lambda}_1$.
Thus $\Lambda_1$ is nonpositive.
\end{proof}

%%%%%%%%%%%%%%%%%%%%%%%%%
\subsection{Irreducible subgroups of $\PO(p,q)$ preserving properly convex domains}

Here is a consequence of Fact~\ref{fact:Yves} (see Figure~\ref{fig:Omega-min-max}).

\begin{proposition} \label{prop:Lambda-non-pos-neg}
For $p,q\in\NN^*$, an irreducible discrete subgroup $\Gamma$ of $G = \PO(p,q)$ preserves a nonempty properly convex subset of $\PP(\RR^{p,q})$ if and only if the following two conditions are satisfied:
\begin{enumerate}[(i)]
  \item\label{item:i-lifts-Hpq} $\Gamma$ contains an element of $G$ which is proximal in $\partial_{\scriptscriptstyle\PP}\HH^{p,q-1}$,
  \item\label{item:ii-lifts-Hpq} $\Lambda_{\Gamma}$ is nonpositive or nonnegative (Definition~\ref{def:non-pos-neg}).
\end{enumerate}
In this case, let $\Omega$ be a nonempty $\Gamma$-invariant properly convex open subset of $\PP(\RR^{p,q})$ and $\widetilde{\Omega}$ a properly convex cone of $\RR^{p,q}\smallsetminus\{0\}$ lifting~$\Omega$.
If $\Lambda_{\Gamma}$ is nonpositive (\resp nonnegative), then $\Lambda_{\Gamma}$ lifts to a cone $\widetilde{\Lambda}_{\Gamma}$ in the boundary of $\widetilde{\Omega}$ on which $\langle\cdot,\cdot\rangle_{p,q}$ is nonpositive (\resp nonnegative).
There is a unique smallest nonempty $\Gamma$-invariant properly convex open subset $\Omega_{\min}$ of $\PP(\RR^{p,q})$ contained in~$\Omega$, namely the interior of the convex hull of $\Lambda_{\Gamma}$ in~$\Omega$.
There is a unique largest $\Gamma$-invariant properly convex open subset $\Omega_{\max}$ of $\PP(\RR^{p,q})$ containing~$\Omega$, namely the projectivization of the interior of the set of $x'\in\RR^{p,q}$ such that $\langle x,x'\rangle_{p,q}\leq\nolinebreak 0$ for all $x\in\widetilde{\Lambda}_{\Gamma}$ (\resp$\langle x,x'\rangle_{p,q}\geq 0$ for all $x\in\widetilde{\Lambda}_{\Gamma}$).
\end{proposition}

\begin{figure}
[h!]
\centering
\includegraphics[scale=1]{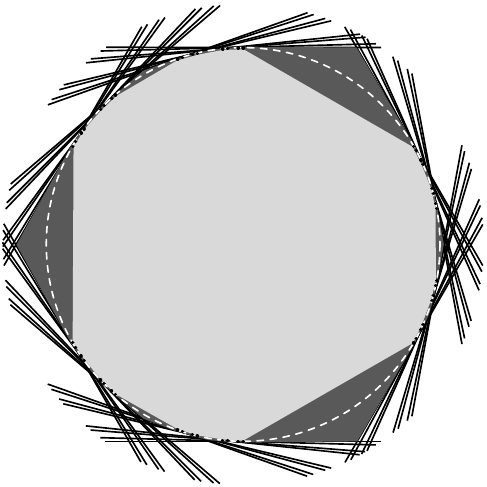}
\hspace{0.8cm}
\includegraphics[scale=1.2]{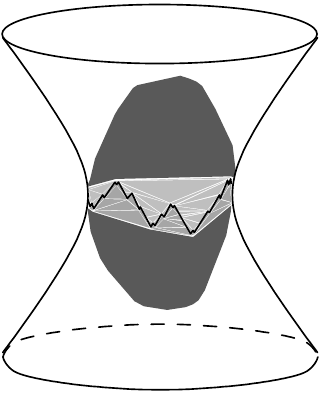}
\caption{The sets $\Omega_{\min}\subset\HH^{p,q-1}$ (light gray) and $\Omega_{\max}\subset\PP(\RR^{p,q})$ (dark gray), for a negative proximal limit set $\Lambda_{\Gamma}$.
On the left $(p,q)=(2,1)$, and on the right $(p,q)=(2,2)$.}
\label{fig:Omega-min-max}
\end{figure}

\begin{proof}
Taking $V=\RR^{p,q}$, we only need to check that condition~\eqref{item:i-lifts-Hpq} of Proposition~\ref{prop:Lambda-non-pos-neg} is equivalent to condition~\eqref{item:i-lifts} of Fact~\ref{fact:Yves}, that condition~\eqref{item:ii-lifts-Hpq} of Proposition~\ref{prop:Lambda-non-pos-neg} implies condition~\eqref{item:ii-lifts} of Fact~\ref{fact:Yves}, and that the existence of a nonempty $\Gamma$-invariant properly convex open subset of $\PP(\RR^{p,q})$ implies condition~\eqref{item:ii-lifts-Hpq} of Proposition~\ref{prop:Lambda-non-pos-neg}.

An element $g\in G=\PO(p,q)$ is proximal in $\PP(\RR^{p,q})$ if and only if $g^{-1}$ is proximal in $\PP(\RR^{p,q})$, because the set of eigenvalues of~$g$ is stable under taking inverses, and so $g$ has a unique eigenvalue of maximal modulus if and only if $g^{-1}$ has.
On the other hand, $g^{-1}$ is proximal in $\PP(\RR^{p,q})$ if and only if $g$ is proximal in $\PP((\RR^{p,q})^*)$ (see Section~\ref{subsec:limit-set}).
Using Remark~\ref{rem:lim-set-POpq}, we see that condition~\eqref{item:i-lifts} of Fact~\ref{fact:Yves} is equivalent to condition~\eqref{item:i-lifts-Hpq} of Proposition~\ref{prop:Lambda-non-pos-neg} for $\Gamma\subset G=\PO(p,q)$.

Suppose that $\Lambda_{\Gamma}$ is nonpositive (\resp nonnegative), \ie we can lift it to a cone $\widetilde{\Lambda}_{\Gamma}$ of $\RR^{p,q}\smallsetminus\{0\}$ on which $\langle\cdot,\cdot\rangle_{p,q}$ is nonpositive (\resp nonnegative).
Recall from Remark~\ref{rem:inv-dom} that the map $\psi : x\mapsto\langle x,\cdot\rangle_{p,q}$ identifies $\RR^{p,q}$ with $(\RR^{p,q})^*$ and $\Lambda_{\Gamma}\subset\PP(\RR^{p,q})$ with $\Lambda_{\Gamma}^*\subset\PP((\RR^{p,q})^*)$.
The set $\widetilde{\Lambda}_{\Gamma}^* := \psi(\widetilde{\Lambda}_{\Gamma})$ (\resp $\widetilde{\Lambda}_{\Gamma}^* := - \psi(\widetilde{\Lambda}_{\Gamma})$) is a cone of $(\RR^{p,q})^*\smallsetminus\nolinebreak\{0\}$ lifting $\Lambda_{\Gamma}^*$, and by construction $\ell(x)\leq 0$ for all $x\in\widetilde{\Lambda}_{\Gamma}$ and $\ell\in\widetilde{\Lambda}_{\Gamma}^*$.
Thus condition~\eqref{item:ii-lifts-Hpq} of Proposition~\ref{prop:Lambda-non-pos-neg} implies condition~\eqref{item:ii-lifts} of Fact~\ref{fact:Yves} for $\Gamma\subset G=\PO(p,q)$.

Suppose that there exists a nonempty $\Gamma$-invariant properly convex open subset $\Omega$ of $\PP(\RR^{p,q})$.
It lifts to a properly convex cone $\widetilde{\Omega}$ of $\RR^{p,q}\smallsetminus\{0\}$, and $\Lambda_{\Gamma}$ lifts to a cone $\widetilde{\Lambda}_{\Gamma}$ of $\RR^{p,q}\smallsetminus\{0\}$ contained in the boundary of~$\widetilde{\Omega}$.
Let
$$\widetilde{\Omega}^* := \big\{ \ell\in (\RR^{p,q})^* ~|~ \ell(x)<0\quad \forall x\in\overline{\widetilde{\Omega}}\big\},$$
where $\overline{\widetilde{\Omega}}$ is the closure of $\widetilde{\Omega}$ in $\RR^{p,q}\smallsetminus \{0\}$.
The set $\widetilde{\Omega}^*$ is a properly convex cone of $(\RR^{p,q})^*\smallsetminus\{0\}$ lifting~$\Omega^*$. 
The set~$\Lambda_{\Gamma}^*$ lifts to a cone $\widetilde{\Lambda}_{\Gamma}^*$ of $(\RR^{p,q})^*\smallsetminus\{0\}$ contained in the boundary of~$\widetilde{\Omega}^*$.
By construction, $\ell(x)\leq 0$ for all $x\in\widetilde{\Lambda}_{\Gamma}$ and $\ell\in\widetilde{\Lambda}_{\Gamma}^*$.
Let $\widehat\Gamma \subset \OO(p,q)$ be the lift of~$\Gamma$ leaving invariant $\widetilde{\Omega}$ (hence also $\widetilde{\Omega}^*$, $\widetilde{\Lambda}_\Gamma$, and $\widetilde{\Lambda}_\Gamma^*$). 
Since $\psi$ induces an identification between $\Lambda_{\Gamma}$ and $\Lambda_{\Gamma}^*$, we have $\psi(x)\in\widetilde{\Lambda}_{\Gamma}^*\cup -\widetilde{\Lambda}_{\Gamma}^*$ for all $x\in\widetilde{\Lambda}_{\Gamma}$.
Let $F^-$ (\resp $F^+$) be the subcone of $\widetilde{\Lambda}_{\Gamma}$ consisting of those vectors $x$ such that $\psi(x)\in\widetilde{\Lambda}_{\Gamma}^*$ (\resp $\psi(x)\in -\widetilde{\Lambda}_{\Gamma}^*$).
By construction, we have $x\in F^-$ if and only if $\langle x,x'\rangle_{p,q}\leq 0$ for all $x'\in\widetilde{\Lambda}_{\Gamma}$; in particular, $F^-$ is closed in~$\widetilde{\Lambda}_{\Gamma}$ and $\widehat\Gamma$-invariant.
Similarly, $F^+$ is closed and $\widehat\Gamma$-invariant.
The sets $F^-$ and~$F^+$ are disjoint since no $x\in\RR^{p,q}\smallsetminus\{0\}$ can satisfy $\langle x,x'\rangle_{p,q}=\nolinebreak 0$ for all $x'\in\nolinebreak\widetilde{\Lambda}_{\Gamma}$, otherwise the $\Gamma$-invariant subset ${\Lambda}_{\Gamma}$ of $\PP(\RR^{p,q})$ would be contained in the hyperplane $\PP(x^{\perp})$, contradicting the irreducibility of~$\Gamma$.
Thus $F^-$ and~$F^+$ are disjoint, $\widehat \Gamma$-invariant, closed subcones of~$\widetilde{\Lambda}_{\Gamma}$, whose projections to $\PP(\RR^{p,q})$ are disjoint, $\Gamma$-invariant, closed subsets of~$\Lambda_{\Gamma}$.
Since $\Gamma$ is irreducible, $\Lambda_{\Gamma}$ is the smallest nonempty $\Gamma$-invariant closed subset of $\PP(\RR^{p,q})$ (see Section~\ref{subsec:limit-set}), and so $\widetilde{\Lambda}_{\Gamma} = F^-$ or $\widetilde{\Lambda}_{\Gamma} = F^+$.
In the first case $\Lambda_{\Gamma}$ is nonpositive, and in the second case it is nonnegative.
\end{proof}

In the setting of Proposition~\ref{prop:Lambda-non-pos-neg}, if $\Lambda_{\Gamma}\subset\partial_{\scriptscriptstyle\PP}\HH^{p,q-1}$ is nonpositive (\resp nonnegative), then $\Omega_{\min}$ is contained in $\HH^{p,q-1}$ (\resp $\SS^{p-1,q}$) by Lemma~\ref{lem:nonpos-Hpq}.\eqref{item:nonpos-Hpq-Lambda0}.
We shall use the following in the stronger situation that $\Lambda_{\Gamma}$ is negative (\resp positive).

\begin{lemma} \label{lem:Omega-min-Hpq}
In the setting of Proposition~\ref{prop:Lambda-non-pos-neg}, if $\Lambda_{\Gamma}\subset\partial_{\scriptscriptstyle\PP}\HH^{p,q-1}$ is negative (\resp positive), then the closure $\C_{\min}$ of $\Omega_{\min}$ in $\HH^{p,q-1}$ (\resp $\SS^{p-1,q}$) is contained in~$\Omega_{\max}$.
\end{lemma}

\begin{proof}
Suppose $\Lambda_{\Gamma}$ is negative.
For any $x\in\widetilde{\Lambda}_{\Gamma}$, using the equality
$$\Big\langle x, \sum_i t_i x_i \Big\rangle_{p,q} = \sum_i t_i \langle x, x_i\rangle_{p,q}$$
for $t_i\in\RR^+$ and $x_i\in\widetilde{\Lambda}_{\Gamma}$, we see that $\langle x,\cdot\rangle_{p,q}$ is negative on the $\RR^+$-span of~$\widetilde{\Lambda}_{\Gamma}$ minus $\{0\}$.
In particular, the set $\C_{\min}$, which is the projectivization of this $\RR^+$-span minus $\{0\}$, is contained in $\Omega_{\max}$, which is the projectivization of the interior of the set of $x'\in\RR^{p,q}$ such that $\langle x,x'\rangle_{p,q}\leq 0$ for all $x\in\widetilde{\Lambda}_{\Gamma}$.

The case that $\Lambda_{\Gamma}$ is positive is analogous.
\end{proof}

%%%%%%%%%%%%%%%%%%%%%%%%%%%%%%%%%%%%%%%%%%%%%%%%%%%
\section{$\HH^{p,q-1}$-convex cocompact groups are Anosov} \label{sec:cc->Ano}

The goal of this section is to prove the implications \eqref{item:a1}~$\Rightarrow$~\eqref{item:a2} and \eqref{item:a3}~$\Rightarrow$~\eqref{item:a4} of Theorem~\ref{thm:main-negative}, which contain Theorem~\ref{thm:main}.\eqref{item:ccc-implies-Anosov}.
By the following observation, which is immediate from the definitions, we can focus on \eqref{item:a1}~$\Rightarrow$~\eqref{item:a2} only.

\begin{remark}\label{rem:pos-impl}
A representation $\rho : \Gamma\to\PO(p,q)$ is $P_1^{p,q}$-Anosov if and only if it is $P_1^{q,p}$-Anosov under the identification $\PO(p,q)\simeq\PO(q,p)$.
A subset of $\partial_{\scriptscriptstyle\PP}\HH^{p,q-1}$ is positive if and only if it is negative under the identification $\partial_{\scriptscriptstyle\PP}\HH^{p,q-1}\simeq\partial_{\scriptscriptstyle\PP}\HH^{q,p-1}$.
\end{remark}

We also prove (Lemma~\ref{lem:Lambda-Lambda-Gamma} and Remark~\ref{rem:conseq-Lambda-Lambda-Gamma}) that if $\Gamma$ is $\HH^{p,q-1}$-convex cocompact, then for any nonempty, closed, properly convex subset $\C$ of $\HH^{p,q-1}$ on which $\Gamma$ acts properly discontinuously and cocompactly, the ideal boundary $\partial_i\C$ is the proximal limit set $\Lambda_\Gamma\subset\partial_{\scriptscriptstyle\PP}\HH^{p,q-1}$.

%%%%%%%%%%%%%%%%%%%%%%%%%
\subsection{Working inside a properly convex \emph{open} domain}

The following lemma will enable us to use the restriction to~$\C$ of the Hilbert metric of $\Omega_{\max}$.

\begin{lemma} \label{lem:no-degenerate-faces}
Let $\Gamma$ be an irreducible discrete subgroup of $G=\PO(p,q)$ acting properly discontinuously and cocompactly on some nonempty properly convex closed subset $\C$ of~$\HH^{p,q-1}$.
Then $\C$ is contained in a maximal $\Gamma$-invariant properly convex open subset $\Omega_{\max}$ of $\PP(\RR^{p,q})$.
\end{lemma}

The point of Lemma~\ref{lem:no-degenerate-faces} is that, not only the interior of~$\C$, but also its boundary in $\HH^{p,q-1}$ is contained in such a set~$\Omega_{\max}$.

Recall from Fact~\ref{fact:Yves} and Remark~\ref{rem:inv-dom} that in this setting the proximal limit set $\Lambda_{\Gamma}$ is contained in the ideal boundary $\partial_i\C$ of~$\C$; a maximal $\Gamma$-invariant properly convex open subset $\Omega_{\max}$ of $\PP(\RR^{p,q})$ containing~$\C$ is unique.

\begin{proof}
Let $\Omega_{\max}$ be the largest $\Gamma$-invariant properly convex open subset of $\PP(\RR^{p,q})$ containing $\mathrm{Int}(\C)$, as in Fact~\ref{fact:Yves} and Remark~\ref{rem:inv-dom}.
Suppose by contradiction that $\C$ is not contained in $\Omega_{\max}$.
By Fact~\ref{fact:Yves} and Remark~\ref{rem:inv-dom}, this means that some point $y\in\C$ belongs to $z^{\perp}$ for some $z \in \Lambda_{\Gamma} \subset \partial_i\C$.
The interval $[y,z)$ is a lightlike ray of~$\HH^{p,q-1}$.
By convexity of~$\C$, it is fully contained in~$\C$.
Let $(a_m)_{m\in\NN}$ be a sequence of points of $[y,z)$ converging to~$z$ (see Figure~\ref{fig:CinOhm}).
Since $\Gamma$ acts cocompactly on~$\C$, for any $m$ there exists $\gamma_m\in\Gamma$ such that $\gamma_m\cdot a_m$ belongs to a fixed compact subset of~$\mathcal{C}$.
Up to taking a subsequence, the sequences $(\gamma_m\cdot a_m)_m$ and $(\gamma_m\cdot y)_m$ and $(\gamma_m\cdot z)_m$ converge respectively to some points $a_{\infty}, y_{\infty}, z_{\infty}$ in~$\overline{\C}$, with $a_{\infty}\in\C$ and $z_{\infty}\in\nolinebreak\Lambda_{\Gamma}$.
Since $a_{\infty}\in [y_{\infty},z_{\infty}]\subset z_{\infty}^{\perp}$, the intersection of $[y_{\infty},z_{\infty}]$ with $\HH^{p,q-1}$ is contained in a lightlike geodesic, hence can meet $\partial_{\scriptscriptstyle\PP}\HH^{p,q-1}$ only at~$z_\infty$.
Thus $y_{\infty}$ cannot belong to $\partial_{\scriptscriptstyle\PP}\HH^{p,q-1}$, lest $y_{\infty}=z_{\infty}$ and the closure of $\C$ in $\PP(\RR^{p,q})$ contain a full projective line, contradicting the proper convexity of~$\C$.
Therefore, $y_{\infty}\in\C$.
But this contradicts the proper discontinuity of the action of $\Gamma$ on~$\mathcal{C}$. 
\end{proof}
\begin{figure}[h]
\centering
\labellist
\small\hair 2pt
\pinlabel {$z$} [u] at 37 52
\pinlabel {$a_m$} [u] at 37 60
\pinlabel {$y$} [u] at 43 75
\pinlabel {$\gamma_m \cdot z$} [u] at 73.6 35.5
\pinlabel {$\gamma_m \cdot a_m$} [u] at 71 59
\pinlabel {$\gamma_m \cdot y$} [u] at 73 74
\pinlabel {$\mathcal{C}$} [u] at 100 65
\pinlabel {$\Omega_{\mathrm{max}}$} [u] at 70 98
\pinlabel {$\Lambda_\Gamma$} [u] at 106 38
\pinlabel {$\partial_{\scriptscriptstyle\PP} \HH^{p,q-1}$} [u] at 8 90
\endlabellist
\includegraphics[scale=1.5]{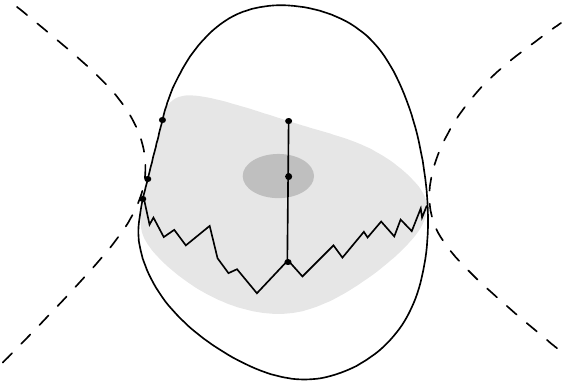}
\caption{Illustration for the proof of Lemma~\ref{lem:no-degenerate-faces}}
\label{fig:CinOhm}
\end{figure}
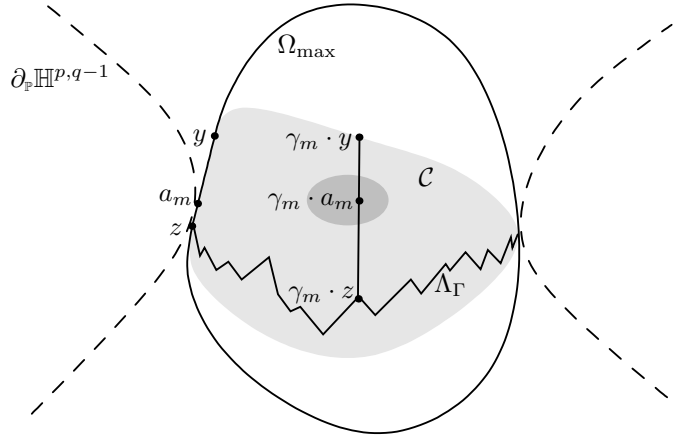

%%%%%%%%%%%%%%%%%%%%%%%%%
\subsection{Equality $\partial_i\C=\Lambda_{\Gamma}$}

Using Lemma~\ref{lem:no-degenerate-faces}, we establish the following.

\begin{lemma}\label{lem:Lambda-Lambda-Gamma}
Let $\Gamma$ be an irreducible discrete subgroup of $G=\PO(p,q)$ acting properly discontinuously and cocompactly on a nonempty properly convex closed subset $\C$ of~$\HH^{p,q-1}$.
If the proximal limit set $\Lambda_{\Gamma}$ is transverse, then the inclusion $\Lambda_{\Gamma} \subset \partial_i\C$ is an equality, and this set is negative (Definition~\ref{def:pos-neg}).
\end{lemma}

Recall that the transversality of $\Lambda_{\Gamma}$ means that $y\notin z^{\perp}$ for all $y\neq z$~in~$\Lambda_{\Gamma}$.

\begin{proof}
Let $\Omega_{\min} \subset \C \subset \HH^{p,q-1}$ be the interior of the convex hull of $\Lambda_\Gamma$ in~$\C$, and $\C_{\min}$ its closure in $\HH^{p,q-1}$.
We have
$$\Lambda_{\Gamma} \subset \partial_i\C_{\min} = \partial_{\scriptscriptstyle\PP}\Omega_{\min} \cap \partial_{\scriptscriptstyle\PP}\HH^{p,q-1} \subset \partial_i\C.$$
The set $\Lambda_{\Gamma}$ is transverse by assumption, hence negative by Lemma~\ref{lem:nonpos-Hpq}.\eqref{item:nonpos-Hpq-Omega}, and so $\Lambda_{\Gamma} = \partial_i\C_{\min}$ by Lemma~\ref{lem:nonpos-Hpq}.\eqref{item:nonpos-Hpq-Lambda0}.
We now check that $\partial_i\C_{\min} = \partial_i\C$.
For this we use the fact that, by Lemma~\ref{lem:no-degenerate-faces}, the set $\C$ is contained in a $\Gamma$-invariant properly convex open subset $\Omega_{\max}$ of $\PP(\RR^{p,q})$; we denote by $d$ the Hilbert metric on $\Omega_{\max}$.

Suppose by contradiction that there exists $z \in \partial_i\C \smallsetminus \partial_i\C_{\min}$.
Let $(z_m)_{m\in\NN}$ be a sequence in $\C\smallsetminus\C_{\min}$ converging to~$z$.
By cocompactness of the action of $\Gamma$ on~$\C$ and~$\C_{\min}$, we may find a sequence $(y_m)_{m\in\NN}$ in $\Omega_{\min}$ such that $d(y_m, z_m)$ is uniformly bounded and $d(y_m, \partial_{\scriptscriptstyle\HH} \C_{\min})$ is uniformly bounded away from zero.
The segment $[y_m, z_m]$ contains a unique point $u_m$ of $\partial_{\scriptscriptstyle\HH} \C_{\min}$, as depicted in Figure~\ref{fig:Nested}.
Let $(a_m, b_m)$ be the maximal interval of $\Omega_{\max}$ containing $y_m, z_m$, so that $d(y_m, z_m) = \frac{1}{2} \log [a_m, y_m, z_m, b_m]$ and $d(y_m, u_m) = \frac{1}{2} \log [a_m, y_m, u_m, b_m]$.
Up to passing to a subsequence, we may assume that $a_m \to a$, $b_m \to b$, $u_m \to u$, and $y_m \to y$ where $u$ and~$y$ belong to the line segment $[a,b] \subset \partial_{\scriptscriptstyle\PP}\Omega_{\max}$ and $u,y \in \partial_i\C_{\min}$.
By assumption $y \neq z$; since the cross ratios $[a_m, y_m, z_m, b_m] = e^{2d(y_m,z_m)}$ are bounded away from $0$, $1$, and $+\infty$, the points $a,y,z,b$ are pairwise distinct and $[a_m,y_m,z_m,b_m] \to [a,y,z,b]$. On the other hand, the cross ratios $[a_m, y_m, u_m, b_m] = e^{2d(y_m,u_m)}$ are bounded away from~$1$, hence $[a,y,u,b] \neq 1$.
Since the points $a,y,b$ are pairwise distinct, we conclude $y \neq u$.
But the segment $[y,u]$ is contained in $\partial_{\scriptscriptstyle\PP}\Omega_{\max}$, hence contained in $\partial_i\C_{\min}$, contradicting the transversality of $\partial_i\C_{\min} = \Lambda_{\Gamma}$.
\end{proof}

\begin{figure}
\centering
\labellist
\small\hair 2pt
\pinlabel {$a$} [u] at 32 32
\pinlabel {$y$} [u] at 39 62
\pinlabel {$u$} [u] at 39 68
\pinlabel {$z$} [u] at 42 76
\pinlabel {$b$} [u] at 46 87
\pinlabel {$a_m$} [u] at 71 -2
\pinlabel {$y_m$} [u] at 77 66
\pinlabel {$u_m$} [u] at 78 73
\pinlabel {$z_m$} [u] at 78 80
\pinlabel {$b_m$} [u] at 79 116
\pinlabel {$\mathcal{C}_{\mathrm{min}}$} [u] at 90 40
\pinlabel {$\mathcal{C}$} [u] at 84 32
\pinlabel {$\Omega_{\mathrm{max}}$} [u] at 84 7
\pinlabel {$\partial_{\scriptscriptstyle\PP} \HH^{p,q-1}$} [u] at 11 90 
\endlabellist
\includegraphics[scale=1.5]{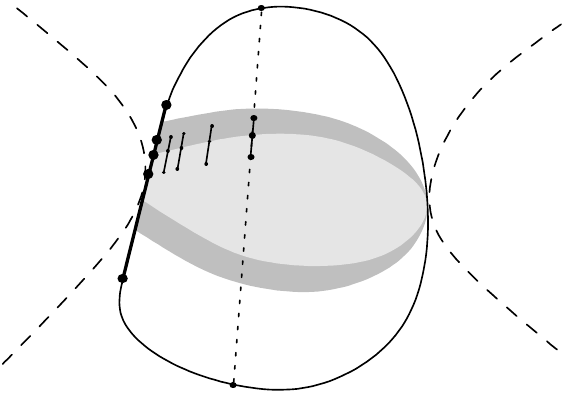}
\caption{Illustration of the proof of Lemma~\ref{lem:Lambda-Lambda-Gamma}}
\label{fig:Nested}
\end{figure}
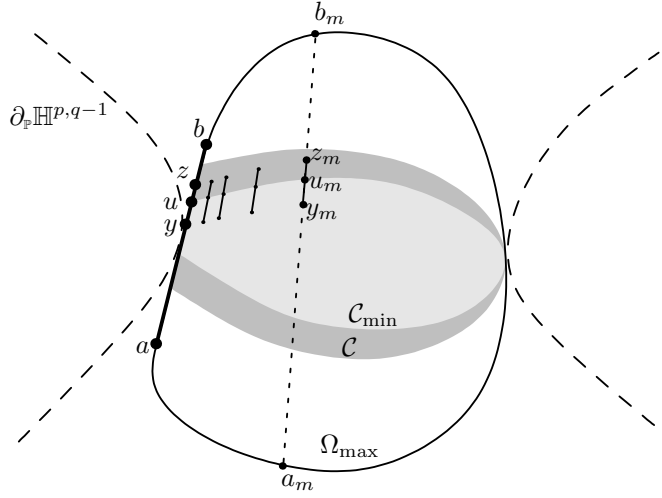

\begin{remark} \label{rem:conseq-Lambda-Lambda-Gamma}
Lemma~\ref{lem:Lambda-Lambda-Gamma} shows that if an irreducible discrete subgroup $\Gamma$ of $\PO(p,q)$ acts properly discontinuously and cocompactly on some nonempty properly convex closed subset $\C$ of~$\HH^{p,q-1}$ and if the proximal limit set $\Lambda_{\Gamma}$ is transverse, then $\Gamma$ is $\HH^{p,q-1}$-convex cocompact.
It also shows that if $\Gamma$ is $\HH^{p,q-1}$-convex cocompact, then for any nonempty properly convex closed subset $\C$ of $\HH^{p,q-1}$ on which $\Gamma$ acts properly discontinuously and cocompactly, the ideal boundary $\partial_i\C$ is the proximal limit set $\Lambda_\Gamma\subset\partial_{\scriptscriptstyle\PP}\HH^{p,q-1}$.
\end{remark}

%%%%%%%%%%%%%%%%%%%%%%%%%
\subsection{Gromov hyperbolicity of $(\C,d)$} \label{subsec:C-hyperb}

In the setting of Lemma~\ref{lem:Lambda-Lambda-Gamma}, we denote by $d$ the Hilbert metric on a $\Gamma$-invariant properly convex open subset $\Omega_{\max}$ of $\PP(\RR^{p,q})$ containing~$\C$, as given by Lemma~\ref{lem:no-degenerate-faces}.
Using arguments inspired from \cite{ben04}, we now prove that in this setting the metric space $(\C,d)$ is Gromov hyperbolic with Gromov boundary $\partial_i\C=\Lambda_{\Gamma}$.

We start with the following lemma; recall from Lemma~\ref{lem:geod-loop}.\eqref{item:end-geod-ray} that any geodesic ray of $(\C,d)$ has a well-defined endpoint in $\partial_i\C=\Lambda_{\Gamma}$.

\begin{lemma} \label{lem:geod-non-strict-conv}
In the setting of Lemma~\ref{lem:Lambda-Lambda-Gamma}, there exists $R>0$ such that any geodesic ray of $(\C,d)$ lies at Hausdorff distance $\leq R$ from the projective interval with the same endpoints.
\end{lemma}

\begin{proof}
Suppose by contradiction that for any $m\in\NN$ there is a geodesic ray $\mathcal{G}_m$ with endpoints $a_m\in\C$ and $b_m\in\Lambda_{\Gamma}$ and a point $y_m\in\C$ on that geodesic ray which lies at distance $\geq m$ from the projective interval $[a_m,b_m)$.
By cocompactness of the action of $\Gamma$ on~$\C$, for any $m\in\NN$ there exists $\gamma_m\in\Gamma$ such that $\gamma_m\cdot y_m$ belongs to a fixed compact set of~$\mathcal{C}$.
Up to taking a subsequence, $(\gamma_m\cdot y_m)_m$ converges to some $y_{\infty}\in\mathcal{C}$, and $(\gamma_m\cdot a_m)_m$ and $(\gamma_m\cdot b_m)_m$ converge respectively to some $a_{\infty}\in\overline{C}=\C\cup\Lambda_{\Gamma}$ and $b_{\infty}\in\Lambda_{\Gamma}$.
Since the distance $d$ from $y_m$ to $[a_m,b_m)$ goes to infinity, we have $[a_{\infty},b_{\infty}]\subset\partial_i\C=\Lambda_{\Gamma}$, hence $a_{\infty}=b_{\infty}$ by transversality of~$\Lambda_{\Gamma}$.
Therefore, up to extracting, the geodesic rays $\mathcal{G}_m$ converge to a biinfinite geodesic of $(\Omega_{\max},d)$ with both endpoints equal, contradicting Lemma~\ref{lem:geod-loop}.\eqref{item:end-biinf-geod}.
\end{proof}

\begin{lemma} \label{lem:C-hyp}
In the setting of Lemma~\ref{lem:Lambda-Lambda-Gamma}, the metric space $(\mathcal{C},d)$ is Gromov hyperbolic.
\end{lemma}

\begin{proof}
Suppose by contradiction that triangles of $(\C,d)$ are not uniformly thin.
By Lemma~\ref{lem:geod-non-strict-conv}, triangles of $(\C,d)$ whose sides are projective segments are not uniformly thin: namely, there exist $a_m,b_m,c_m\in\mathcal{C}$ and $y_m\in [a_m,b_m]$ such that
\begin{equation} \label{eqn:u_m-far}
d(y_m, [a_m,c_m]\cup [c_m,b_m]) \underset{n\to +\infty}{\longrightarrow} +\infty.
\end{equation}
By cocompactness, for any $m$ there exists $\gamma_m\in\Gamma$ such that $\gamma_m\cdot y_m$ belongs to a fixed compact set of~$\mathcal{C}$, as shown in Figure~\ref{fig:QuasiGeo}.
Up to taking a subsequence, $(\gamma_m\cdot y_m)_m$ converges to some $y_{\infty}\in\mathcal{C}$, and $(\gamma_m\cdot a_m)_m$ and $(\gamma_m\cdot b_m)_m$ and $(\gamma_m\cdot\nolinebreak c_m)_m$ converge respectively to some $a_{\infty},b_{\infty},c_{\infty}\in\overline{\C}$.
By \eqref{eqn:u_m-far} we have $[a_{\infty},c_{\infty}]\cup [c_{\infty},b_{\infty}]\subset\partial_i\C$, hence $a_{\infty}=b_{\infty}=c_{\infty}$ by transversality of $\partial_i\C = \Lambda_{\Gamma}$.
This contradicts the fact that $y_{\infty}\in (a_{\infty},b_{\infty})$.
\end{proof}
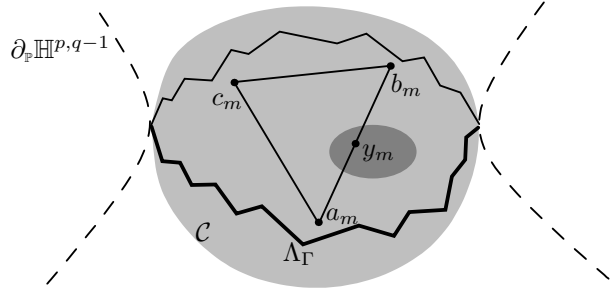
\begin{figure}[h]
\centering
\labellist
\small\hair 2pt
\pinlabel {$a_m$} [u] at 75 18
\pinlabel {$b_m$} [u] at 91 51
\pinlabel {$c_m$} [u] at 46 47
\pinlabel {$y_m$} [u] at 84 34
\pinlabel {$\mathcal{C}$} [u] at 40 14
\pinlabel {$\Lambda_\Gamma$} [u] at 64 8.5
\pinlabel {$\partial_{\scriptscriptstyle\PP} \HH^{p,q-1}$} [u] at 4 60 
\endlabellist
\includegraphics[scale=1.5]{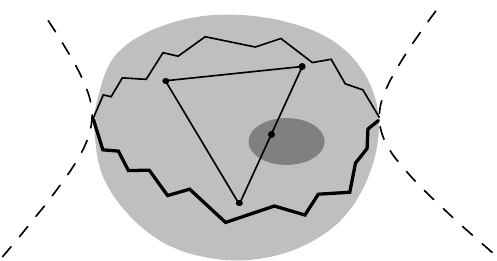}
\caption{Illustration of the proof of Lemma~\ref{lem:C-hyp}}
\label{fig:QuasiGeo}
\end{figure}

\begin{lemma} \label{lem:C-hyp-boundary}
In the setting of Lemma~\ref{lem:Lambda-Lambda-Gamma}, the Gromov boundary of $(\mathcal{C},d)$ is $\Gamma$-equivariantly homeomorphic to $\partial_i\C = \Lambda_{\Gamma}$.
\end{lemma}

\begin{proof}
Fix a basepoint $y\in\mathcal{C}$.
The Gromov boundary of $(\mathcal{C},d)$ is the set of equivalence classes of infinite geodesic rays in~$\mathcal{C}$ starting at~$y$, for the equivalence relation ``to remain at bounded distance for~$d$''.
Consider the $\Gamma$-equivariant continuous map from $\partial_i\C$ to the Gromov boundary of $(\C,d)$ sending $z\in\partial_i\C$ to the class of the straight geodesic ray (with image a projective interval) from $y$ to~$z$.
This map is surjective by Lemma~\ref{lem:geod-non-strict-conv}.
Moreover it is injective, since transversality implies that no two points of $\partial_i \C = \Lambda_\Gamma$ lie in a common face of $\partial \Omega_{\max}$, hence the Hilbert distance $d$ between rays going out to two different points of $\partial_i \C$ goes to infinity.
We conclude using the fact that a continuous bijection between two compact Hausdorff spaces is a homeomorphism.
\end{proof}

%%%%%%%%%%%%%%%%%%%%%%%%%
\subsection{Proof of the implication (\ref{item:a1})~$\Rightarrow$~(\ref{item:a2}) of Theorem~\ref{thm:main-negative}}

Let $\Gamma$ be an irreducible discrete subgroup of $G=\PO(p,q)$ acting properly discontinuously and cocompactly on some nonempty properly convex closed subset $\C$ of~$\HH^{p,q-1}$ whose ideal boundary $\partial_i\C$ is transverse.
By Lemma~\ref{lem:nonpos-Hpq}.\eqref{item:nonpos-Hpq-Omega}, the set $\partial_i\C$ is negative.
By Lemma~\ref{lem:no-degenerate-faces}, the closed set $\C$ is contained in a $\Gamma$-invariant properly convex open subset $\Omega_{\max}$ of $\PP(\RR^{p,q})$, and so we may consider the restriction to~$\C$ of the Hilbert metric $d$ of~$\Omega_{\max}$.
The discrete group $\Gamma$ acts cocompactly by isometries on the metric space $(\mathcal{C},d)$, which by Lemmas \ref{lem:C-hyp} and~\ref{lem:C-hyp-boundary} is Gromov hyperbolic with boundary $\Gamma$-equivariantly homeomorphic to $\partial_i\C=\Lambda_{\Gamma}$.
Therefore $\Gamma$ is word hyperbolic and any orbit map $\Gamma\to\mathcal{C}$ is a quasi-isometry, extending to a $\Gamma$-equivariant homeomorphism $\xi : \partial_{\infty}\Gamma\to\Lambda_{\Gamma}$.
This homeomorphism is transverse since $\Lambda_{\Gamma}$ is a transverse subset of $\partial_{\scriptscriptstyle\PP} \HH^{p,q-1}$.
Since $\Gamma$ is irreducible, we conclude (using \cite[Prop.\,4.10]{gw12}) that the natural inclusion $\Gamma\hookrightarrow G=\PO(p,q)$ is $P_1^{p,q}$-Anosov.

%%%%%%%%%%%%%%%%%%%%%%%%%%%%%%%%%%%%%%%%%%%%%%%%%%%
\section{Anosov subgroups with negative proximal limit set are $\HH^{p,q-1}$-convex cocompact} \label{sec:Ano-neg->cc}

In this section we prove the implication \eqref{item:a2}~$\Rightarrow$~\eqref{item:a1} of Theorem~\ref{thm:main-negative}. 
The implication \eqref{item:a4}~$\Rightarrow$~\eqref{item:a3} of Theorem~\ref{thm:main-negative} immediately follows by Remark~\ref{rem:pos-impl}.
Together with Proposition~\ref{prop:conn-transv}, this yields Theorem~\ref{thm:main}.\eqref{item:Ano-connected-implies-ccc}.
We also show that the connectedness assumption in Theorem~\ref{thm:main}.\eqref{item:Ano-connected-implies-ccc} cannot be removed, by providing a counterexample.

%%%%%%%%%%%%%%%%%%%%%%%%%
\subsection{Proof of the implication (\ref{item:a2})~$\Rightarrow$~(\ref{item:a1}) of Theorem~\ref{thm:main-negative}}

Let $\Gamma$ be an irreducible discrete subgroup of $G=\PO(p,q)$.
Suppose that $\Gamma$ is word hyperbolic, that the natural inclusion $\Gamma\hookrightarrow G$ is $P_1^{p,q}$-Anosov, and that the proximal limit set $\Lambda_{\Gamma}\subset\partial_{\scriptscriptstyle\PP}\HH^{p,q-1}$ is negative (Definition~\ref{def:pos-neg}).

By Proposition~\ref{prop:Lambda-non-pos-neg}, the group $\Gamma$ preserves a nonempty properly convex open subset of $\PP(\RR^{p,q})$; there is a maximal such subset, namely
$$\Omega_{\max} := \PP\big(\mathrm{Int}\big\{ x'\in\RR^{p,q} ~|~ \langle x,x'\rangle_{p,q}\leq 0\ \forall x\in\widetilde{\Lambda}_{\Gamma}\big\}\big),$$
where $\widetilde{\Lambda}_{\Gamma}$ is a cone of $\RR^{p,q}\smallsetminus\{0\}$ lifting $\Lambda_{\Gamma}$ on which all inner products $\langle\cdot,\cdot\rangle_{p,q}$ of noncollinear points are negative.
There is also a minimal such subset $\Omega_{\min}\subset\Omega_{\max}$, namely the interior of the convex hull of $\Lambda_{\Gamma}$ in~$\Omega_{\max}$.
By Lemma~\ref{lem:nonpos-Hpq}.\eqref{item:nonpos-Hpq-Lambda0} we have $\Omega_{\min}\subset\HH^{p,q-1}$, and by Lemma~\ref{lem:Omega-min-Hpq} the closure $\C_{\min}$ of $\Omega_{\min}$ in~$\HH^{p,q-1}$ is contained in $\Omega_{\max}$.
In particular, the action of $\Gamma$ on $\C_{\min}$ is properly discontinuous.
Moreover, the ideal boundary $\partial_i\C_{\min}$ is equal to~$\Lambda_{\Gamma}$ by Lemma~\ref{lem:nonpos-Hpq}.\eqref{item:nonpos-Hpq-Lambda0}; in particular, it is transverse.

To see that $\Gamma$ is $\HH^{p,q-1}$-convex cocompact and thus complete the proof of the implication \eqref{item:a2}~$\Rightarrow$~\eqref{item:a1} of Theorem~\ref{thm:main-negative}, it only remains to prove the following.

\begin{lemma}
In this setting, the action of $\Gamma$ on $\C_{\min}$ is cocompact.
\end{lemma}

\begin{proof}
By Fact~\ref{fact:Ano-PO-PGL}, the natural inclusion $\Gamma\hookrightarrow\PO(p,q)\hookrightarrow\PGL(\RR^{p+q})$ is $P_1$-Anosov, where $P_1$ is the stabilizer in $\PGL(\RR^{p+q})$ of a line of~$\RR^{p+q}$.
By \cite[Th.\,1.7]{klp14} (see also \cite[Rem.\,5.15]{ggkw17}), the action of $\Gamma$ on $\PP(\RR^{p,q})$ is \emph{expanding}: for any point $z\in\Lambda_{\Gamma}$ there exist an element $\gamma\in\Gamma$, a neighborhood $\mathcal{U}$ of $z$ in $\PP(\RR^{p,q})$, and a constant $C>1$ such that $\gamma$ is $C$-expanding~on~$\mathcal{U}$ for the metric
\begin{equation*}
d_{\PP}([x],[x']) := | \sin \measuredangle (x,x') |
\end{equation*}
on $\PP(\RR^{p,q})$.
We now use a version of the argument of \cite[Prop.\,2.5]{klp14}, inspired by Sullivan's dynamical characterization \cite{sul85} of convex cocompactness in Riemannian hyperbolic spaces.
(The argument in \cite{klp14} is a little more involved because it deals with bundles, whereas we work directly in~$\PP(\RR^{p,q})$.)

Suppose by contradiction that the action of $\Gamma$ on $\C_{\min}$ is \emph{not} cocompact, and let $(\varepsilon_m)_{n\in\NN}$ be a sequence of positive reals converging to~$0$.
For any~$m$, the set $K_m := \{ z\in\C_{\min} \,|\, d_{\PP}(z,\Lambda_{\Gamma}) \geq \varepsilon_m\}$ is compact, hence there exists a $\Gamma$-orbit contained in $\C_{\min}\smallsetminus (K_m \cup \Lambda_{\Gamma})$. 
By proper discontinuity of the action on $\C_{\min}$, the supremum of $d_{\PP}(\cdot,\Lambda_{\Gamma})$ on this orbit is achieved at some point $z_m\in\C_{\min}$, and by construction $0 < d_{\PP}(z_m,\Lambda_{\Gamma}) \leq \varepsilon_m$.
Then, for all $\gamma\in\Gamma$,
$$d_{\PP}(\gamma\cdot z_m,\Lambda_{\Gamma}) \leq d_{\PP}(z_m,\Lambda_{\Gamma}). $$
Up to extracting, we may assume that \((z_m)_{n\in\NN}\) converges to some \(z\in\Lambda_{\Gamma}\).
Consider an element $\gamma\in\Gamma$, a neighborhood $\mathcal{U}$ of $z$ in $\partial_{\scriptscriptstyle\PP}\HH^{p,q-1}$, and a constant $C>1$ such that $\gamma$ is $C$-expanding on~$\mathcal{U}$.
For any $n\in\NN$, there exists $z'_m\in\Lambda_{\Gamma}$ such that $d_{\PP}(\gamma\cdot z_m,\Lambda_{\Gamma}) = d_{\PP}(\gamma\cdot z_m,\gamma\cdot z'_m)$.
We then have
$$d_{\PP}(\gamma\cdot z_m,\Lambda_{\Gamma}) \geq C\, d_{\PP}( z_m, z'_m) \geq C\, d_{\PP}( z_m,\Lambda_{\Gamma}) \geq  
C\, d_{\PP}(\gamma\cdot z_m,\Lambda_{\Gamma}).$$
This is impossible since $C>1$.
\end{proof}

%%%%%%%%%%%%%%%%%%%%%%%%%
\subsection{Disconnected limit sets} \label{subsec:noncococo}

Let $\Gamma$ be a free group on two generators.
For $\mathrm{rank}_{\RR}(G):=\min(p,q)\geq 2$, let us give an example of an irreducible $P_1^{p,q}$-Anosov representation $\rho : \Gamma\to G=\PO(p,q)$ for which the proximal limit set $\Lambda_{\rho(\Gamma)}$ is neither negative nor positive (Definition~\ref{def:pos-neg}).
This shows that Theorem~\ref{thm:main}.\eqref{item:Ano-connected-implies-ccc} is not true when $\partial_{\infty}\Gamma$ is not connected.
We first work in $\PO(2,2)$ (Example~\ref{ex:noncococo}), then in $\PO(p,q)$ for any $p,q\geq 2$ (Example~\ref{ex:noncococo-higher}).

\begin{example} \label{ex:noncococo}
Let $\Gamma$ be a free group on two generators. 
Consider two injective and discrete representations $\rho_1,\rho_2:\Gamma\rightarrow \PSL_2(\RR)$ with convex cocompact images, such that $\rho_1$ is the holonomy of a hyperbolic 3-holed sphere and $\rho_2$ the holonomy of a hyperbolic one-holed torus. 
For $i\in\{1,2\}$, let $\xi_i :\nolinebreak\partial_\infty\Gamma\rightarrow \PP^1\RR$ be the boundary map associated to $\rho_i$, with image $\Lambda_i$ (a Cantor set). 
Let $\psi:=\xi_2\circ\xi_1^{-1}: \Lambda_1\rightarrow \Lambda_2$ be the unique $(\rho_1, \rho_2)$-equivariant homeomorphism. 
This map $\psi$ does not preserve the cyclic order of $\PP^1\RR$: there exist $x_1,x_2,x_3,x_4 \in \partial_\infty \Gamma$ such that  the quadruples $(\xi_1(x_1),\xi_1(x_2),\xi_1(x_3),\xi_1(x_4))$ and $(\xi_2(x_1),\xi_2(x_2),\xi_2(x_4),\xi_2(x_3))$ are both cyclically ordered.

The identification $\PO(2,2)_0 \simeq \PSL_2(\RR)\times \PSL_2(\RR)$ of Remark~\ref{rem:PO22} lets us see $(\rho_1, \rho_2)$ as a single representation $\rho:\Gamma\rightarrow \PO(2,2)$. 
The boundary maps $\xi_1, \xi_2:\partial_\infty\Gamma\rightarrow \PP^1\RR$ associated to $\rho_1, \rho_2$ combine into a map $\xi$ from $\partial_\infty\Gamma$ to the doubly ruled quadric $\partial_{\scriptscriptstyle\PP} \HH^{2,1}\simeq \PP^1\RR\times \PP^1\RR$: under this identification, the image $\Lambda$ of $\xi$ is the graph of $\psi$, and $\rho$ is $P_1^{2,2}$-Anosov with boundary map~$\xi$.
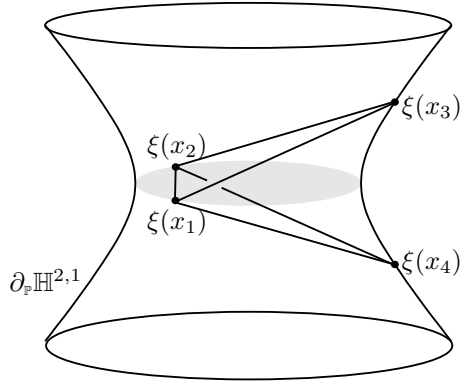
\begin{figure}
\centering
\labellist
\small\hair 2pt
\pinlabel {$\xi(x_1)$} [u] at 34 40
\pinlabel {$\xi(x_2)$} [u] at 34 59 
\pinlabel {$\xi(x_3)$} [u] at 98 69
\pinlabel {$\xi(x_4)$} [u] at 98 30
\pinlabel {$\partial_{\scriptscriptstyle\PP} \HH^{2,1}$} [u] at 1 25 
\endlabellist
\includegraphics[scale=1.5]{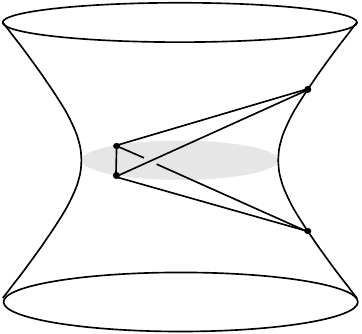}
\caption{Illustration for Example~\ref{ex:noncococo}. The triples $\{\xi(x_1),\xi(x_2),\xi(x_3)\}$ and $\{\xi(x_1),\xi(x_2),\xi(x_4)\}$ are negative, \ie span an ideal triangle contained in~$\HH^{2,1}$. The triples $\{\xi(x_1),\xi(x_3),\xi(x_4)\}$ and $\{\xi(x_2),\xi(x_3),\xi(x_4)\}$ are positive, \ie span ideal triangles contained in $\SS^{1,2}$ (which go through infinity in the picture).}
\label{fig:Mixed}
\end{figure}
However, $\rho(\Gamma)$ is not $\HH^{2,1}$-convex cocompact nor is it $\HH^{1,2}$-convex cocompact (with respect to $-\langle \cdot, \cdot \rangle_{2,2}$).
As depicted in Figure~\ref{fig:Mixed}, the triples $\xi(x_1),\xi(x_2),\xi(x_3)$ and $\xi(x_1),\xi(x_2),\xi(x_4)$ are negative, while the triples $\xi(x_1),\xi(x_3),\xi(x_4)$ and $\xi(x_2),\xi(x_3),\xi(x_4)$ are positive.
Alternatively, observe that the six segments connecting the $\xi(x_i)$ (for $1\leq i \leq 4$) inside $\HH^{2,1}$ (\resp inside $\SS^{1,2}$) carry the generator of $\pi_1(\PP(\RR^{2,2}))$, precluding the possibility that these segments could be part of a properly convex subset of $\PP(\RR^{2,2})$.
\end{example}

\begin{example} \label{ex:noncococo-higher}
Take any $p,q\geq 2$ and consider the embedding
$$\tau : \PO(2,2)_0 \simeq \SO(2,2)_0 \longhookrightarrow \SO(p,q)_0 \simeq \PO(p,q)_0$$
coming from the natural inclusion $\RR^{2,2}\subset\RR^{p,q}$.
The corresponding $\tau$-equiva\-riant embedding $\iota : \partial_{\scriptscriptstyle\PP}\HH^{2,1}\hookrightarrow\partial_{\scriptscriptstyle\PP}\HH^{p,q-1}$ has the property that a subset $\Lambda$ of $\partial_{\scriptscriptstyle\PP}\HH^{2,1}$ is negative (\resp positive) if and only $\iota(\Lambda)$ is negative (\resp positive) as a subset of $\partial_{\scriptscriptstyle\PP}\HH^{p,q-1}$.
Let $\rho : \Gamma\to\PO(2,2)_0$ be as in Example~\ref{ex:noncococo}.
By \cite[Prop.\,4.7]{gw12}, the composition $\tau\circ\rho : \Gamma\to\PO(p,q)$ is $P_1^{p,q}$-Anosov with proximal limit set $\Lambda_{\tau\circ\rho(\Gamma)} = \iota(\Lambda_{\rho(\Gamma)})$.
By Example~\ref{ex:noncococo}, this limit set is neither negative nor positive.
Since being Anosov is an open property \cite{lab06,gw12} and since a small deformation of a negative (\resp positive) triple in $\partial_{\scriptscriptstyle\PP}\HH^{p,q-1}$ is still negative (\resp positive), any small deformation of $\tau\circ\rho$ in $\Hom(\Gamma,\PO(p,q))$ is still a $P_1^{p,q}$-Anosov representation with a proximal limit set that is neither negative nor positive.
Since $\Gamma$ is a free group, such deformations are abundant, including many which are irreducible.
The image of any such representation fails to be $\HH^{p,q-1}$-convex cocompact.
\end{example}

%%%%%%%%%%%%%%%%%%%%%%%%%%%%%%%%%%%%%%%%%%%%%%%%%%%
\section{Link with strong projective convex cocompactness} \label{sec:proofCM}

The goal of this section is to prove Proposition~\ref{prop:CramponMarquis}.
We start with some general lemmas.

%%%%%%%%%%%%%%%%%%%%%%%%%
\subsection{Supporting hyperplanes at the limit set}

Recall that a \emph{supporting hyperplane} of a properly convex set $\Omega$ at a point $z\in\partial_{\scriptscriptstyle\PP}\Omega$ is a projective hyperplane whose intersection with $\overline{\Omega}$ is a subset of $\partial_{\scriptscriptstyle\PP}\Omega$ containing~$z$.
In the following lemma, we denote by $\Omega_{\min}\subset\Omega_{\max}$ a pair of minimal and maximal nonempty $\Gamma$-invariant properly convex open subsets of $\PP(\RR^{p,q})$, given by Proposition~\ref{prop:Lambda-non-pos-neg}, and by $d$ the Hilbert metric on~$\Omega_{\max}$.

\begin{lemma} \label{lem:C1-at-limit-set}
Let $\Gamma$ be an irreducible discrete subgroup of $G = \PO(p,q)$ preserving a nonempty properly convex open subset $\Omega$ of $\PP(\RR^{p,q})$.
Suppose that the proximal limit set $\Lambda_{\Gamma} \subset \partial_{\scriptscriptstyle\PP}\Omega$ is transverse.
If $\Omega$ contains a uniform neighborhood of $\Omega_{\min}$ in $(\Omega_{\max},d)$, then $\Omega$ has a unique supporting hyperplane at any point $z$ of~$\Lambda_{\Gamma}$, namely~$z^{\perp}$.
\end{lemma}

\begin{proof}
We first check that $\Omega_{\max}$ has a unique supporting hyperplane at any point $z$ of~$\Lambda_{\Gamma}$, namely~$z^{\perp}$.
By Proposition~\ref{prop:Lambda-non-pos-neg}, the set $\Lambda_{\Gamma}$ is negative or positive.
We assume that it is negative (the positive case is analogous).
Let $\widetilde{\Lambda}_{\Gamma}$ be a cone of $\RR^{p,q}\smallsetminus\{0\}$ lifting~$\Lambda_{\Gamma}$ on which all inner products $\langle\cdot,\cdot\rangle_{p,q}$ of noncollinear points are negative.
By Proposition~\ref{prop:Lambda-non-pos-neg}, the set $\Omega_{\max}$ is the projectivization of the interior $\widetilde{\Omega}_{\max}$ of the set of $x'\in\RR^{p,q}$ such that $\langle x,x'\rangle_{p,q}\leq 0$ for all $x\in\widetilde{\Lambda}_{\Gamma}$.
A supporting hyperplane to $\widetilde{\Omega}_{\max}$ in~$\RR^{p,q}$ is the kernel of a linear form $\ell=\sum_{i=1}^k \langle x_i,\cdot\rangle$ with $x_1,\dots,x_k \in \widetilde{\Lambda}_{\Gamma}$.
For any $x\in\widetilde{\Lambda}_{\Gamma}$ in such a hyperplane, we have $\ell(x)=0$ and $\langle x_i,x\rangle_{p,q}\leq 0$ for all~$i$, hence $\langle x_i,x\rangle_{p,q}=0$ for all~$i$.
But the set $\Lambda_\Gamma$ is transverse by assumption, and so all $x_i$ are  collinear to~$x$.
Thus the unique supporting hyperplane to $\widetilde{\Omega}_{\max}$ at $x\in\widetilde{\Lambda}_{\Gamma}$ is~$x^{\perp}$.
Taking images in $\PP(\RR^{p,q})$, the unique supporting hyperplane to $\Omega_{\max}$ at a point $z\in\Lambda_{\Gamma}$ is~$z^{\perp}$.

We now assume that $\Omega$ contains a uniform neighborhood of $\Omega_{\min}$ in $(\Omega_{\max},d)$.
Let us check that $\Omega$ has a unique supporting hyperplane at any point $z$ of~$\Lambda_{\Gamma}$, namely~$z^{\perp}$.
For $z\in\Lambda_{\Gamma}$, let $(y_t)_{t\geq 0}$ and $(z_t)_{t\geq 0}$ be two straight geodesic rays in $\Omega_{\max}$ with endpoint~$z$, such that $(z_t)_{t\geq 0}$ is contained in~$\Omega_{\min}$.
Since $\Omega_{\max}$ has a unique supporting hyperplane at~$z$, up to reparametrization we have $d(y_t,z_t)\to 0$ as $t\to +\infty$.
Since $\Omega$ contains some uniform neighborhood of $\Omega_{\min}$ in $(\Omega_{\max},d)$, we deduce that some subray $(y_t)_{t\geq t_0}$ is contained in~$\Omega$.
This shows that $z^\perp$ is also the unique supporting hyperplane of $\Omega$ at~$z$.
\end{proof}

%%%%%%%%%%%%%%%%%%%%%%%%%
\subsection{Extreme points at the limit set}

We next prove the following.

\begin{proposition}\label{prop:possibly-superfluous}
Let $\Gamma$ be an irreducible discrete subgroup of $\PO(p,q)$
preserving a properly convex open subset $\Omega$ of $\PP(\RR^{p,q})$ and acting cocompactly on some nonempty properly convex closed subset $\C$ of~$\Omega$.
If the proximal limit set $\Lambda_{\Gamma}$ is transverse, then every point of $\Lambda_\Gamma$ is an extreme point of $\overline{\Omega}$.
\end{proposition}

Proposition~\ref{prop:possibly-superfluous} relies on the following lemma.

\begin{lemma} \label{lem:unif-neighb}
Let $\Gamma$ be a discrete subgroup of $\PO(p,q)$ preserving a properly convex open subset $\Omega$ of $\PP(\RR^{p,q})$ and acting cocompactly on some nonempty properly convex closed subset $\C$ of $\Omega$ contained in $\HH^{p,q-1}$.
For any $r>0$, the closed uniform neighborhood $\C_r$ of $\C$ in $\Omega$ for the Hilbert metric $d_{\Omega}$ is properly convex, and the action of $\Gamma$ on~$\C_r$ is properly discontinuous and cocompact.
If $r>0$ is small enough, then $\C_r$ is contained in~$\HH^{p,q-1}$, and its ideal boundary $\partial_i\C_r$ is equal to the proximal limit set $\Lambda_{\Gamma}$ if $\Lambda_{\Gamma}$ is transverse.
\end{lemma}

\begin{proof}[Proof of Lemma~\ref{lem:unif-neighb}]
The action of $\Gamma$ is properly discontinuous on~$\Omega$ (see Section~\ref{subsec:prop-conv-proj}), hence also on~$\C$ and on $\C_r$ for any $r>0$.
Let $\DD\subset\HH^{p,q-1}$ be a compact fundamental domain for the action of $\Gamma$ on~$\C$.
For any $r>0$, the set $\C_r$ is the union of the $\Gamma$-translates of the closed uniform $r$-neighborhood $\DD_r$ of $\DD$ in $(\Omega,d_{\Omega})$.
Since $\DD_r$ is compact, the action of $\Gamma$ on $\C_r$ is cocompact. 
The proper convexity of~$\C_r$ follows from \cite[(18.12)]{bus55}.
Since $\DD$ is compact and contained in  the open set $\HH^{p,q-1}$, for small enough $r > 0$ we have $\DD_r \subset \HH^{p,q-1}$, hence $\C_r\subset\HH^{p,q-1}$.
In that case, if $\Lambda_{\Gamma}$ is transverse, then $\partial_i\C_r = \Lambda_{\Gamma}$ by Lemma~\ref{lem:Lambda-Lambda-Gamma}.
\end{proof}

\begin{proof}[Proof of Proposition~\ref{prop:possibly-superfluous}]
Suppose by contradiction that $\Lambda_{\Gamma}$ is transverse but that there exists $z \in \Lambda_\Gamma$ which is contained in a nontrivial open segment $I$ of $\partial_{\scriptscriptstyle\PP} \Omega$.
Fix $y \in \C$.
For any $z_1, z_2\in I$ with $z \in (z_1, z_2)$, the open triangle with vertices $y,z_1,z_2$ is contained in a uniform neighborhood of the ray $[y,z)$ in $(\Omega,d_{\Omega})$; in particular, it is contained in a uniform neighborhood $\C_r$ of~$\C$ in $(\Omega,d_{\Omega})$, for some $r>0$, and $(z_1,z_2)$ is contained in the ideal boundary $\partial_i\C_r$.
By choosing $z_1,z_2$ close enough to~$z$, we may make $r$ as small as we like.
By Lemma~\ref{lem:unif-neighb}, if $r> 0$ is small enough, then $\partial_i\C_r=\Lambda_{\Gamma}$ is transverse: contradiction. 
\end{proof}

%%%%%%%%%%%%%%%%%%%%%%%%%
\subsection{Proof of Proposition~\ref{prop:CramponMarquis}.\eqref{item:CM1}} \label{subsec:proofCM1}

Suppose $\Gamma$ is $\HH^{p,q-1}$-convex cocompact: it acts properly discontinuously and cocompactly on some nonempty properly convex closed subset $\C$ of~$\HH^{p,q-1}$ whose ideal boundary $\partial_i\C$ is transverse.
By Lemma~\ref{lem:Lambda-Lambda-Gamma}, the set $\partial_i\C$ is equal to the proximal limit set $\Lambda_{\Gamma}$.
Let $\C_{\min}$ be the convex hull of $\Lambda_{\Gamma}$ in~$\C$: it is a closed convex subset of~$\C$, which has compact quotient by~$\Gamma$.
By Lemma~\ref{lem:no-degenerate-faces}, the set $\C_{\min}$ is contained in a maximal $\Gamma$-invariant properly convex open subset $\Omega_{\max}$ of $\PP(\RR^{p,q})$.
By Lemma~\ref{lem:unif-neighb}, there exists $r>0$ such that the closed uniform neighborhood $\C_r$ of $\C_{\min}$ in $\Omega_{\max}$ for the Hilbert metric $d_{\Omega_{\max}}$ is properly convex, with ideal boundary $\partial_i\C_r=\Lambda_{\Gamma}$, and the action of $\Gamma$ on~$\C_r$ is properly discontinuous and cocompact.
In order to prove that $\Gamma$ is strongly convex cocompact in $\PP(\RR^{p+q})$ (Definition~\ref{def:cm}), it is sufficient to prove the following.

\begin{lemma} \label{lem:strict-conv-C1}
In this setting, there is a $\Gamma$-invariant open neighborhood $\Omega$ of $\C_{\min}$ in $\mathcal{U}_r:=\mathrm{Int}(\C_r)$ which is strictly convex with $C^1$ boundary $\partial_{\scriptscriptstyle\PP}\Omega$.
\end{lemma}

Indeed, for such an $\Omega$, the orbital limit set $\Lambda^{\mathsf{orb}}_{\Omega}(\Gamma)$ is equal to $\Lambda_\Gamma$ by Lemma~\ref{lem:snowman}, hence the convex hull of $\Lambda^{\mathsf{orb}}_{\Omega}(\Gamma)$ in $\Omega$ is $\C_{\min}$, which has compact quotient by~$\Gamma$.
Thus if such an~$\Omega$ exists, then $\Gamma$ is strongly convex cocompact in $\PP(\RR^{p+q})$.

Recall that, by definition, $\Omega$ is strictly convex with $C^1$ boundary $\partial_{\scriptscriptstyle\PP}\Omega$ if every point of $\partial_{\scriptscriptstyle\PP}\Omega$ has the property that it is an extreme point of $\overline{\Omega}$ and $\Omega$ has a unique supporting hyperplane at that point.
If $\Omega$ is any $\Gamma$-invariant properly convex open neighborhood of $\C_{\min}$ in~$\Omega_{\max}$, then every point of~$\Lambda_\Gamma$ has this property by Lemma~\ref{lem:C1-at-limit-set} and Proposition~\ref{prop:possibly-superfluous}.
Therefore, in order to prove Lemma~\ref{lem:strict-conv-C1}, we only need to focus on $\partial_{\scriptscriptstyle\PP}\Omega \smallsetminus \Lambda_\Gamma$, that is we must construct $\Omega$ such that each point of $\partial_{\scriptscriptstyle\PP}\Omega \smallsetminus \Lambda_\Gamma$ is an extreme point with unique supporting hyperplane.
Constructing such a neighborhood $\Omega$ clearly involves arbitrary choices; here is one of many possible constructions.
Cooper--Long--Tillmann~\cite[Prop.\,8.3]{clt} give a different construction yielding, in the case $\Gamma$ is torsion-free, a convex set $\Omega$ as in Lemma~\ref{lem:strict-conv-C1} with the slightly stronger property that $\partial_{\scriptscriptstyle\PP}\Omega \smallsetminus \Lambda_{\Gamma}$ is locally the graph of a smooth function with positive definite Hessian.

\begin{proof}[Proof of Lemma~\ref{lem:strict-conv-C1}]
We set $V:=\RR^{p+q}$.
The following argument does not use the quadratic form of signature $(p,q)$.
Let $\Gamma_0$ be a subgroup of~$\Gamma$ which is torsion-free; such a subgroup exists by the Selberg lemma \cite[Lem.\,8]{sel60}. 

We proceed in three steps.
Firstly, we construct a $\Gamma$-invariant open neighborhood $\Omega_1\subset\mathcal{U}_r$ of $\C_{\min}$ in~$\Omega_{\max}$ which has $C^1$ boundary, but which is not necessarily strictly convex.
Secondly, we construct a small deformation $\Omega_2\subset\mathcal{U}_r$ of~$\Omega_1$ which still has $C^1$ boundary and which is strictly convex, but only $\Gamma_0$-invariant, not necessarily $\Gamma$-invariant.
Finally, we use an averaging procedure over translates $\gamma\cdot\Omega_2$ of~$\Omega_2$, for $\gamma\Gamma_0$ ranging over the $\Gamma_0$-cosets of~$\Gamma$, to construct a $\Gamma$-invariant open neighborhood $\Omega\subset\mathcal{U}_r$ of $\C_{\min}$ which has $C^1$ boundary and is strictly convex.

\smallskip
\noindent
$\bullet$ \textbf{Construction of~$\Omega_1$:}
Consider a compact fundamental domain $\DD$ for the action of~$\Gamma$ on~$\C_{\min}$.
The convex hull of $\DD$ in~$\Omega_{\max}$ is still contained in~$\C_{\min}$.
Let $\DD'\subset\C_r$ be a closed neighborhood of this convex hull in~$\mathcal{U}_r$ which has $C^1$ boundary $\partial_{\scriptscriptstyle\PP}\DD'$, and let $\Omega_1$ be the interior of the convex hull of $\Gamma\cdot\DD'$ in~$\mathcal{U}_r$. 
We have $\overline{\Omega_1}\smallsetminus \Lambda_\Gamma \subset \overline{\mathcal{U}_r}\smallsetminus\Lambda_{\Gamma} = \C_r \subset \Omega_{\max}$ by choice of~$r$; in particular, the action of $\Gamma$ on $\overline{\Omega_1}\smallsetminus \Lambda_\Gamma$ is properly discontinuous.

Let us check that $\Omega_1$ has $C^1$ boundary $\partial_{\scriptscriptstyle\PP}\Omega_1$.
We first observe that any supporting hyperplane $\Pi_y$ to $\Omega_1$ at a point $y\in\partial_{\scriptscriptstyle\PP}\Omega_1 \smallsetminus\Lambda_{\Gamma}$ stays away from~$\Lambda_\Gamma$: indeed, if $\Pi_y$ contained a point $z\in \Lambda_{\Gamma}$, then by Lemma~\ref{lem:C1-at-limit-set} it would be equal to the unique supporting hyperplane to $\Omega_1$ at~$z$, namely $z^{\perp}$, contradicting $y\in\Omega_{\max}$.
On the other hand, since the action of $\Gamma$ on $\overline{\Omega_1} \smallsetminus\Lambda_{\Gamma}$ is properly discontinuous, for any neighborhood 
$\mathcal{N}$ of $\Lambda_\Gamma$ in $\PP(V)$ and any infinite sequence of distinct elements $\gamma_j\in\Gamma$, the translates $\gamma_j \cdot \DD'$ are eventually all contained in~$\mathcal{N}$.
Therefore, in a neighborhood of $y$, the hypersurface $\partial_{\scriptscriptstyle\PP}\Omega_1$ coincides with the convex hull of a \emph{finite} union of translates $\gamma\cdot\DD'$, which has $C^1$ boundary by Lemma~\ref{lem:Conv-C1}.

\smallskip
\noindent
$\bullet$ \textbf{Construction of~$\Omega_2$:}
For any $y\in\partial_{\scriptscriptstyle\PP}\Omega_1\smallsetminus \Lambda_\Gamma$, let $\mathcal{F}_y$ be the \emph{face of $y$ in $\partial_{\scriptscriptstyle\PP}\Omega$}, namely the intersection of $\partial_{\scriptscriptstyle\PP}\Omega_1$ with the unique supporting hyperplane $\Pi_y$ to $\Omega_1$ at~$y$.
By the above observation, $\mathcal{F}_y$ is a closed convex subset of $\partial_{\scriptscriptstyle\PP}\Omega_1\smallsetminus\Lambda_\Gamma$.

We claim that $\mathcal{F}_y$ is disjoint from $\gamma\cdot \mathcal{F}_y=\mathcal{F}_{\gamma\cdot y}$ for all $\gamma\in\Gamma_0\smallsetminus\{1\}$.
Indeed, if there existed $y'\in \mathcal{F}_y\cap \mathcal{F}_{\gamma\cdot y}$, then by uniqueness the supporting hyperplanes would satisfy $\Pi_y=\Pi_{y'}=\Pi_{\gamma\cdot y}$, hence $\mathcal{F}_y=\mathcal{F}_{y'}=\mathcal{F}_{\gamma\cdot y}=\gamma\cdot \mathcal{F}_y$.
This would imply $\mathcal{F}_y=\gamma^m\cdot \mathcal{F}_y$ for all $m\in\NN$, hence $\gamma^m\cdot y\in \mathcal{F}_y$ for all $m\in\NN$.
Using the fact that the action of $\Gamma_0$ on $\partial_{\scriptscriptstyle\PP}\Omega_1\smallsetminus\Lambda_{\Gamma}$ is properly discontinuous and taking a limit, we see that $\mathcal{F}_y$ would contain a point of~$\Lambda_{\Gamma}$, which we have seen is not true.
Therefore $\mathcal{F}_y$ is disjoint from $\gamma\cdot \mathcal{F}_y$ for all $\gamma\in\Gamma_0\smallsetminus\{1\}$.

For any $y\in\partial_{\scriptscriptstyle\PP}\Omega_1\smallsetminus\Lambda_{\Gamma}$, the subset of $\PP(V^*)$ consisting of those projective hyperplanes near the supporting hyperplane $\Pi_y$ that separate $\mathcal{F}_y$ from $\Lambda_\Gamma$ is open and nonempty, hence $(n-1)$-dimensional where $n:=p+q=\dim(V)$.
Choose $n-1$ such hyperplanes $\Pi_y^1,\dots, \Pi_y^{n-1}$ in generic position, with $\Pi_y^i$ cutting off a compact region $\mathcal{Q}_y^i\supset \mathcal{F}_y$ from~$\overline{\Omega_1}\smallsetminus \Lambda_\Gamma$.
One may imagine each $\Pi_y^i$ is obtained by pushing $\Pi_y$ normally into $\Omega_1$ and then tilting slightly in one of $n-1$ independent directions.
The intersection $\bigcap_{i=1}^{n-1} \Pi_y^i \subset \PP(V)$ is reduced to a singleton.
By taking each hyperplane $\Pi_y^i$ very close to $\Pi_y$, we may assume that the union $\mathcal{Q}_y:=\bigcup_{i=1}^{n-1} \mathcal{Q}_y^i$ is disjoint from all its $\gamma$-translates for $\gamma \in \Gamma_0 \smallsetminus \{1\}$. In addition, we ensure that $\mathcal{F}_y$ has a neighborhood $\mathcal{Q}'_y$ contained in $\bigcap_{i=1}^{n-1} \mathcal{Q}_y^i$.

Since the action of $\Gamma_0$ on $\partial_{\scriptscriptstyle\PP} \Omega_1\smallsetminus\Lambda_{\Gamma}$ is cocompact, there exist finitely many points $y_1,\dots,y_m\in\partial_{\scriptscriptstyle\PP} \Omega_1\smallsetminus\Lambda_{\Gamma}$ such that $(\partial_{\scriptscriptstyle\PP}\Omega_1\smallsetminus\Lambda_{\Gamma})\subset \Gamma_0\cdot (\mathcal{Q}'_{y_1}\cup\dots\cup \mathcal{Q}'_{y_m})$.
 
We now explain, for any $y\in\partial_{\scriptscriptstyle\PP}\Omega_1\smallsetminus\Lambda_{\Gamma}$, how to deform $\Omega_1$ into a new, smaller properly convex $\Gamma_0$-invariant open neighborhood of~$\C_{\min}$ with $C^1$ boundary, in a way that destroys all segments in $\mathcal{Q}'_y$.
Repeating for $y=y_1, \dots, y_m$, this will produce a strictly convex $\Gamma_0$-invariant open neighborhood $\Omega_2 \subset \Omega_1$ of~$\C_{\min}$ with $C^1$ boundary $\partial_{\scriptscriptstyle\PP}\Omega_2$.

Choose an affine chart containing~$\Omega_{\max}$, an auxiliary Euclidean metric $g$ on this chart, and a smooth strictly concave function $h : \RR^+\rightarrow \RR^+$ with $h(0)=0$ and $\frac{\D}{\D t}\big|_{t=0}\, h(t)=1$ (\eg $h=\tanh$).
We may assume that for every $1\leq i \leq n-1$ the $g$-orthogonal projection $\pi_y^i$ onto~$\Pi_y^i$ satisfies $\pi_y^i(\mathcal{Q}_y^i)\subset \Pi_y^i \cap \Omega_1$, with $(\pi_y^i|_{\mathcal{Q}_y^i})^{-1}(\Pi_y^i\cap \partial_{\scriptscriptstyle\PP} \Omega_1)\subset \Pi_y^i$. 
Define maps $\varphi_y^i: \mathcal{Q}_y^i \rightarrow \mathcal{Q}_y^i$ by the property that $\varphi_y^i$ preserves each fiber $(\pi_y^i)^{-1}(y')$ (a segment), taking the point at distance $t$ from $y'$ to the point at distance~$h(t)$.
Then $\varphi_y^i$ takes any segment $\sigma$ of $\mathcal{F}_y$ to a strictly convex curve, unless $\sigma$ is parallel to $\Pi_i$.
The image $\varphi_y^i(\mathcal{Q}_y^i\cap \partial_{\scriptscriptstyle\PP} \Omega_1)$ is still a convex hypersurface.
Extending $\varphi_y^i$ by the identity on $\mathcal{Q}_y\smallsetminus \mathcal{Q}_y^i$ and repeating with varying $i$, we find that the composition $\varphi_y:=\varphi_y^1\circ\dots\circ\varphi_y^{n-1}$, defined on $\mathcal{Q}_y$, 
takes $\mathcal{Q}'_y\cap \partial_{\scriptscriptstyle\PP} \Omega_1$ to a strictly convex hypersurface.
We can extend $\varphi_y$ in a $\Gamma_0$-equivariant fashion to $\Gamma_0\cdot\mathcal{Q}_y$, and extend it further by the identity on the rest of~$\Omega_1$: the set $\varphi_y(\Omega_1)$ is still $\Gamma_0$-invariant, with $C^1$ boundary, and is still contained in~$\mathcal{U}_r$.

Repeating with finitely many points $y_1, \dots, y_m$ as above, we obtain a strictly convex, $\Gamma_0$-invariant open neighborhood $\Omega_2\subset\mathcal{U}_r$ of~$\C_{\min}$ with $C^1$ boundary $\partial_{\scriptscriptstyle\PP}\Omega_2$.

\smallskip
\noindent
$\bullet$ \textbf{Construction of~$\Omega$:}
Consider the finitely many $\Gamma_0$-cosets $\gamma_1\Gamma_0, \dots, \gamma_k\Gamma_0$ of~$\Gamma$ and the corresponding translates $\Omega'_i:=\gamma_i\cdot\Omega_2$.
Let $\Omega''$ be a $\Gamma$-invariant properly convex (not necessarily strictly convex) open neighborhood of~$\C_{\min}$ in $\mathcal{U}_r$ which has $C^1$ boundary $\partial_{\scriptscriptstyle\PP}\Omega''$ and is contained in all $\Omega'_i$, $1\leq i\leq k$.
(Such a neighborhood $\Omega''$ can be constructed for instance by the same method as for $\Omega_1$ above.) 
Since $\Omega'_i$ is strictly convex, uniform neighborhoods of $\Omega''$ in $(\Omega'_i,d_{\Omega'_i})$ are strictly convex \cite[(18.12)]{bus55}.
Therefore, by cocompactness, if $h : [0,1]\to [0,1]$ is a convex function with sufficiently fast growth (\eg $h(t)=t^{\alpha}$ for large enough $\alpha>0$), then the $\Gamma_0$-invariant function $H_i:=h\circ d_{\Omega'_i}(\cdot, \Omega'')$ is convex on the convex region $H_i^{-1}([0,1])$, and in fact smooth and strictly convex near every point outside $\Omega''$.
The function $H:=\sum_{i=1}^k H_i$ is $\Gamma$-invariant and its sublevel set $\Omega:=H^{-1}([0,1))$ is a $\Gamma$-invariant open neighborhood of $\C_{\min}$ in~$\mathcal{U}_r$ which is strictly convex with $C^1$ boundary $\partial_{\scriptscriptstyle\PP}\Omega$.
\end{proof}

%%%%%%%%%%%%%%%%%%%%%%%%%
\subsection{Proof of Proposition~\ref{prop:CramponMarquis}.\eqref{item:CM2}} \label{subsec:proofCM2}

Suppose $\Gamma$ preserves a nonempty strictly convex open subset $\Omega$ of $\PP(\RR^n)$.
By Lemma~\ref{lem:snowman}, the orbital limit set $\Lambda^{\mathsf{orb}}_{\Omega}(\Gamma)$ coincides with the proximal limit set $\Lambda_\Gamma$.
Suppose that the convex hull $\C_{\min}$ of $\Lambda_\Gamma$ in~$\Omega$ has compact quotient by~$\Gamma$.
By Proposition~\ref{prop:Lambda-non-pos-neg}, the set $\Lambda_{\Gamma}\subset\partial_{\scriptscriptstyle\PP}\HH^{p,q-1}$ is nonpositive or nonnegative (Definition~\ref{def:non-pos-neg}).
Moreover, $\Lambda_{\Gamma}$ is transverse, because it is contained (Fact~\ref{fact:Yves}) in the boundary of the strictly convex set~$\Omega$.
Therefore, $\Lambda_{\Gamma}$ is negative or positive.
If $\Lambda_\Gamma$ is negative, then $\C_{\min} \subset \HH^{p,q-1}$ by Lemma~\ref{lem:nonpos-Hpq}.\eqref{item:nonpos-Hpq-Lambda0} and Lemma~\ref{lem:Omega-min-Hpq}, and $\Gamma$ is $\HH^{p,q-1}$-convex cocompact by Remark~\ref{rem:conseq-Lambda-Lambda-Gamma}.
Similarly, if $\Lambda_{\Gamma}$ is positive, then the image of~$\Gamma$ under the natural isomorphism $\PO(p,q)\simeq\PO(q,p)$ is $\HH^{q,p-1}$-convex cocompact.

%%%%%%%%%%%%%%%%%%%%%%%%%%%%%%%%%%%%%%%%%%%%%%%%%%%
\section{Examples of $\HH^{p,q-1}$-convex cocompact subgroups}\label{sec:examples}

In this section we consider the following general construction.

\begin{proposition} \label{prop:quasiFuchsian}
Let $H$ be a real semisimple Lie group of real rank~$1$.
For $p,q\in\NN^*$, let $\tau : H\to G:=\PO(p,q)$ be a linear representation which is proximal, in the sense that $\tau(H)$ contains an element which is proximal in $\partial_{\scriptscriptstyle\PP}\HH^{p,q-1}$.
Then for any word hyperbolic group $\Gamma$ and any representation $\sigma_0 : \Gamma\to H$ with finite kernel and convex cocompact image (in the classical sense) in the rank-one group~$H$,
\begin{enumerate}
  \item\label{item:qF1} the composition $\rho_0 := \tau \circ \sigma_0 : \Gamma \rightarrow G$ is $P_1^{p,q}$-Anosov and the proximal limit set $\Lambda_{\rho_0(\Gamma)}\subset\partial_{\scriptscriptstyle\PP}\HH^{p,q-1}$ is negative or positive (Definition~\ref{def:pos-neg});
  \item\label{item:qF2} the connected component $\mathcal{T}_{\rho_0}$ of~$\rho_0$ in the space of $P_1^{p,q}$-Anosov representations from $\Gamma$ to~$G$ is a neighborhood of $\rho_0$ in $\Hom(\Gamma,G)$ consisting entirely of $P_1^{p,q}$-Anosov representations with negative proximal limit set or entirely of $P_1^{p,q}$-Anosov representations with positive proximal limit set.
\end{enumerate}
\end{proposition}

\begin{proof}
Since $H$ has real rank~$1$, the convex cocompact representation $\sigma_0$ is $P$-Anosov where $P$ is a minimal parabolic subgroup of~$H$ \cite[Th.\,5.15]{gw12}; in particular, there is an injective, continuous, $\sigma_0$-equivariant boundary map $\xi_{\sigma_0} : \partial_{\infty}\Gamma\to H/P$.
By \cite[Prop.\,4.7]{gw12} (see also \cite[Prop.\,3.1]{lab06}), since $\tau$ is proximal, there is a $\tau$-equivariant embedding $\iota : H/P \hookrightarrow \partial_{\scriptscriptstyle\PP}\HH^{p,q-1}$ and $\rho_0 = \tau\circ\sigma_0$ is $P_1^{p,q}$-Anosov with boundary map $\iota\circ\xi_{\sigma_0} : \partial_{\infty}\Gamma\to\partial_{\scriptscriptstyle\PP}\HH^{p,q-1}$.
In particular, the proximal limit set $\Lambda_{\rho_0(\Gamma)} = \iota\circ\xi_{\sigma_0}(\partial_{\infty}\Gamma)$ is contained in $\Lambda:=\iota(H/P)$, which is a closed, connected subset of $\partial_{\scriptscriptstyle\PP}\HH^{p,q-1}$.
If $\sigma_0(\Gamma)$ is a uniform lattice in~$H$, then $\Lambda_{\rho_0(\Gamma)}=\Lambda$; since uniform lattices of~$H$ exist, we deduce that $\Lambda$ is transverse.
By Proposition~\ref{prop:conn-transv}, the set $\Lambda$ is negative or positive.
In particular, for arbitrary $\sigma_0(\Gamma)$ (not necessarily a uniform lattice), the set $\Lambda_{\rho_0(\Gamma)}\subset\Lambda$ is negative or positive, proving~\eqref{item:qF1}.

Statement~\eqref{item:qF2} follows from \eqref{item:qF1} and from Proposition~\ref{prop:comp-Ano-neg}.
\end{proof}

Here is an immediate consequence of Theorem~\ref{thm:main-negative}, Proposition~\ref{prop:CramponMarquis}, and Proposition~\ref{prop:quasiFuchsian}.

\begin{corollary} \label{cor:T-Gamma-rho-0}
In the setting of Proposition~\ref{prop:quasiFuchsian}, the group $\rho(\Gamma)$ is strongly convex cocompact in $\PP(\RR^{p+q})$ (Definition~\ref{def:cm}) for all irreducible $\rho\in\mathcal{T}_{\rho_0}$.

More precisely, either $\rho(\Gamma)$ is $\HH^{p,q-1}$-convex cocompact for all irreducible $\rho\in\mathcal{T}_{\rho_0}$, or $\rho(\Gamma)$ is $\HH^{q,p-1}$-convex cocompact (after identifying $\PO(p,q)$ with $\PO(q,p)$) for all irreducible $\rho\in\mathcal{T}_{\rho_0}$.
\end{corollary}

Corollary~\ref{cor:T-Gamma-rho-0} also holds for representations $\rho\in\mathcal{T}_{\rho_0}$ that are not irreducible: see \cite{dgk-cc}.

We now make explicit a few examples to which Corollary~\ref{cor:T-Gamma-rho-0} applies.

%%%%%%%%%%%%%%%%%%%%%%%%%
\subsection{$\HH^{p,q-1}$-quasi-Fuchsian groups} \label{subsec:quasiFuchsian}

Let $\Gamma$ be the fundamental group of a convex cocompact (\eg closed) hyperbolic manifold $M$ of dimension $m\geq 2$, with holonomy $\sigma_0 : \Gamma\to\PO(m,1)=\mathrm{Isom}(\HH^m)$.
The representation $\sigma_0$ is $P_1^{m,1}$-Anosov \cite[Th.\,5.15]{gw12}.
The proximal limit set $\Lambda_{\sigma_0(\Gamma)} \subset \partial_\infty \HH^m$ is negative since any subset of $\partial_\infty \HH^m$ is.

For $p,q\in\NN^*$ with $p\geq m$, the natural embedding $\RR^{m,1}\hookrightarrow\RR^{p,q}$ induces a representation $\tau : \OO(m,1)\to\OO(p,q)\to\PO(p,q)$ which is proximal, and a $\tau$-equivariant embedding $\iota : \partial_{\infty}\HH^m\hookrightarrow\partial_{\scriptscriptstyle\PP}\HH^{p,q-1}$.
The set $\Lambda:=\iota(\partial_{\infty}\HH^m)\subset\partial_{\scriptscriptstyle\PP}\HH^{p,q-1}$ is negative by construction.

The representation $\sigma_0$ lifts to a representation $\widetilde{\sigma}_0 : \Gamma\to H:=\OO(m,1)$.
Let $\rho_0 := \tau\circ\widetilde{\sigma}_0 : \Gamma \to G:=\PO(p,q)$.
The proximal limit set $\Lambda_{\rho_0(\Gamma)} = \iota(\Lambda_{\sigma_0(\Gamma)}) \subset \Lambda$ is negative.
Thus Corollary~\ref{cor:T-Gamma-rho-0} implies the following.

\begin{proposition} \label{prop:quasi-Fuchsian}
In this setting, the representation $\rho_0 : \Gamma\to G=\PO(p,q)$ is $P_1^{p,q}$-Anosov.

The connected component $\mathcal{T}_{\rho_0}$ of~$\rho_0$ in the space of $P_1^{p,q}$-Anosov representations from $\Gamma$ to~$G$ is a neighborhood of $\rho_0$ in $\Hom(\Gamma,G)$ consisting entirely of $P_1^{p,q}$-Anosov representations with negative proximal limit set.

For any irreducible $\rho\in\mathcal{T}_{\rho_0}$, the group $\rho(\Gamma)$ is $\HH^{p,q-1}$-convex cocompact, hence strongly convex cocompact in $\PP(\RR^{p+q})$.
\end{proposition}

For $p=m+1=3$ and $q=1$, when the hyperbolic surface $M$ is closed, the representation $\rho_0 : \Gamma\to\OO(2,1)\hookrightarrow\PO(3,1)$ is called \emph{Fuchsian}, and $\mathcal{T}_{\rho_0}$ is the classical space of \emph{quasi-Fuchsian} representations of $\Gamma=\pi_1(M)$ into $\PO(3,1)$, which Bers parametrized by the product of two copies of the Teichm\"uller space of~$M$.

Suppose $p=m$ and $q=2$.
The space $\HH^{p,1}$ is the $(p+1)$-dimensional (Lorentzian) \emph{anti-de Sitter} space $\mathrm{AdS}^{p+1}$.
When the hyperbolic $m$-manifold $M$ is closed, Proposition~\ref{prop:quasi-Fuchsian} follows from work of Mess \cite{mes90} (for $p=2$) and Barbot--M\'erigot \cite{bm12} (for $p\geq 3$).
In that case $\mathcal{T}_{\rho_0}$ is actually a full connected component of $\Hom(\Gamma,\PO(p,2))$, by Mess \cite{mes90} (for $p=2$) and Barbot \cite{bar15} (for $p\geq 3$).
The terminology \emph{AdS quasi-Fuchsian} is used for $\HH^{p,1}$-convex cocompact representations of $\Gamma$ into $\PO(p,2)$.
For $p=2$, these are exactly the elements of $\mathcal{T}_{\rho_0}$, and they are parametrized by the product of two copies of the Teichm\"uller space of~$M$ \cite{mes90}.
For $p\geq 3$, it is conjectured \cite{bar15} that any $\HH^{p,1}$-convex cocompact representation of~$\Gamma$ lies in $\mathcal{T}_{\rho_0}$.

%%%%%%%%%%%%%%%%%%%%%%%%%
\subsection{Hitchin representations into $\PO(m,m+1)$ and $\PO(m+1,m+1)$, and maximal representations into $\PO(2,q)$} \label{subsec:Hitchin}

For $n\geq 2$, let
$$\tau_n : \SL_2(\RR) \longrightarrow \SL_n(\RR)$$
be the irreducible $n$-dimensional linear representation of $\SL_2(\RR)$, obtained from the action of $\SL_2(\RR)$ on the $(n-1)^{\mathrm{st}}$ symmetric power $\mathrm{Sym}^{n-1}(\RR^2) \simeq \RR^n$.
The image of~$\tau_n$ preserves the nondegenerate bilinear form $B_n := -\omega^{\otimes (n-1)}$ induced from the area form $\omega$ of~$\RR^2$.
This form is symmetric if $n$ is odd, and antisymmetric (\ie symplectic) if $n$ is even.

Suppose $n = 2m+1$ is odd.
The symmetric bilinear form $B_n$ has signature
\begin{equation} \label{eqn:kn-ln}
(k_n, \ell_n) := \left\{ \begin{array}{ll} (m+1,m) & \text{ if $m$ is odd,}\\
(m,m+1) & \text{ if $m$ is even.} \end{array} \right.
\end{equation}
If we identify the orthogonal group $\OO(B_n)$ (containing the image of~$\tau_n$) with $\OO(k_n,\ell_n)$, then there is a unique $\tau_n$-equivariant embedding $\iota_n : \partial_{\infty}\HH^2\hookrightarrow\partial_{\scriptscriptstyle\PP}\HH^{k_n,\ell_n-1}$, and an easy computation shows that its image $\Lambda_n := \iota(\partial_\infty \HH^2)$ is negative.

For $p \geq k_n$ and $q \geq \ell_n$, the representation $\tau_n : \SL_2(\RR)\to\OO(B_n)\simeq\OO(k_n,\ell_n)$ and the natural embedding $\RR^{k_n,\ell_n}\hookrightarrow\RR^{p,q}$ induce a representation $\tau : H = \SL_2(\RR) \to \PO(p,q)$ which is proximal, and a $\tau$-equivariant embedding $\iota : \partial_{\infty}\HH^2\hookrightarrow\partial_{\scriptscriptstyle\PP}\HH^{k_n,\ell_n-1}\hookrightarrow\partial_{\scriptscriptstyle\PP}\HH^{p,q-1}$.
The set $\Lambda:=\iota(\partial_{\infty}\HH^2)\subset\partial_{\scriptscriptstyle\PP}\HH^{p,q-1}$ is negative by construction.

Let $\Gamma$ be the fundamental group of a convex cocompact orientable hyperbolic surface, with holonomy $\sigma_0 : \Gamma\to\PSL_2(\RR)$.
The representation $\sigma_0$ lifts to a representation $\widetilde{\sigma}_0 : \Gamma\to H:=\SL_2(\RR)$.
Let $\rho_0 := \tau\circ\widetilde{\sigma}_0 : \Gamma \to G:=\PO(p,q)$.
The proximal limit set $\Lambda_{\rho_0(\Gamma)} = \iota(\Lambda_{\sigma_0(\Gamma)}) \subset \Lambda$ is negative.
Thus Corollary~\ref{cor:T-Gamma-rho-0} implies the following.

\begin{proposition} \label{prop:gen-Hitchin}
In this setting, the representation $\rho_0 : \Gamma\to G=\PO(p,q)$ is $P_1^{p,q}$-Anosov.

The connected component $\mathcal{T}_{\rho_0}$ of~$\rho_0$ in the space of $P_1^{p,q}$-Anosov representations from $\Gamma$ to~$G$ is a neighborhood of $\rho_0$ in $\Hom(\Gamma,G)$ consisting entirely of $P_1^{p,q}$-Anosov representations with negative proximal limit~set.

For any irreducible $\rho\in\mathcal{T}_{\rho_0}$, the group $\rho(\Gamma)$ is $\HH^{p,q-1}$-convex cocompact, hence strongly convex cocompact in $\PP(\RR^{p+q})$.
\end{proposition}

It follows from \cite{lab06,fg06} (see \eg \cite[\S\,6.1]{biw14}) that when $(p,q)=(k_n,\ell_n)$ as in \eqref{eqn:kn-ln} or $(p,q) = (m+1,m+1)$ and when $\Gamma$ is a closed surface group, the space $\mathcal{T}_{\rho_0}$ of Proposition~\ref{prop:gen-Hitchin} is a full connected component of $\Hom(\Gamma,\PO(p,q))$, called the \emph{Hitchin component} of $\Hom(\Gamma,\PO(p,q))$.
Proposition~\ref{prop:gen-Hitchin} specializes in that case to Proposition~\ref{prop:Hitchin}.

By \cite{biw10,biw}, when $p=m+1=2$ and $\Gamma$ is a closed surface group, the space $\mathcal{T}_{\rho_0}$ is a full connected component of $\Hom(\Gamma,\PO(2,q))$, consisting of so-called \emph{maximal representations}.

%%%%%%%%%%%%%%%%%%%%%%%%%%%%%%%%%%%%%%%%%%%%%%%%%%%
\section{New examples of Anosov representations} \label{sec:ex-Ano}

In this section we use Theorem~\ref{thm:main-negative} to give new examples of Anosov representations, for any hyperbolic right-angled Coxeter group.

%%%%%%%%%%%%%%%%%%%%%%%%%
\subsection{Representations of Coxeter groups into orthogonal groups} \label{subsec:deform-Tits-repr}

By definition, a \emph{right-angled Coxeter group} is a group $W_S$ generated by a finite set of involutions $S = \{s_1, \ldots, s_n\}$, with presentation
$$W_S = \big\langle s_1,\dots,s_n ~|~ (s_i s_j)^{m_{i,j}}=1\quad \forall\,1\leq i,j\leq n\big\rangle$$
where $m_{i,i}=1$ and $m_{i,j} = m_{j,i} \in\{2, \infty\}$ for $i\neq j$.
It is said to be \emph{irreducible} if $S$ cannot be written as the disjoint union of two proper subsets $S'$ and~$S''$ such that $W_{S'}$ and $W_{S''}$ commute (\ie $m_{i,j}=2$ for all $s_i\in S'$ and $s_j\in S''$).

The following construction gives representations of $W_S$ into orthogonal groups, and may be formulated for arbitrary Coxeter groups.
It is a deformation of the well-known geometric representation due to Tits (see Krammer~\cite{kra94}).
Let $(e_1,\dots,e_n)$ be a basis of~$\RR^n$ and $B$ a symmetric bilinear form on~$\RR^n$ satisfying
\begin{equation} \label{item:def-B}
B(e_i,e_j)=\left\{ \begin{array}{rl}
1 & \mathrm{if}\ i=j,\\
0 & \mathrm{if}\ m_{i,j} = 2,\\
-\alpha_{i,j} & \mathrm{if}\ m_{i,j}=\infty,
\end{array} \right.
\end{equation}
where $\alpha_{i,j} = \alpha_{j,i}$ are any real numbers $\geq 1$.
Consider the representation $\rho: W_S \to \mathrm{Aut}_{\RR}(B) \subset \GL_n(\RR)$ sending any generator $s_i$ to the $B$-orthogonal reflection of~$\RR^n$ with respect to~$e_i$:
$$\rho(s_i) = \big(x \longmapsto x - 2 B(e_i,x)\,e_i\big).$$
It is possible that $B$ is degenerate.
To avoid this inconvenience, we can perturb the coefficients $\alpha_{i,j}$ slightly and $B$ becomes nondegenerate.
Indeed, $\det(B)$ is a polynomial in the variables $\alpha_{i,j}$ which is not identically zero (it would take value~$1$ if all $\alpha_{i,j}$ were set to zero).

\begin{remark}
If $B$ is degenerate and one wishes to keep the chosen values of the~$\alpha_{i,j}$, then one may work instead in the vector space $\RR^n/\mathrm{Ker}(B)$, where $\mathrm{Ker}(B)$ is the kernel of~$B$.
Note that $B$ descends to a nondegenerate symmetric bilinear form $\overline{B}$ on $\RR^n/\mathrm{Ker}(B)$ and the representation $\rho$ to a representation $\overline{\rho}$ into the general linear group of $\RR^n/\mathrm{Ker}(B)$ that preserves~$\overline{B}$.
The following arguments easily transpose to this setting.
\end{remark}

From now on, we assume that $B$ is nondegenerate.
We identify $B$ with $\langle\cdot,\cdot\rangle_{p,q}$ and $\mathrm{Aut}_{\RR}(B)$ with $\OO(p,q)$ for some $p,q\in\NN$.
The basis $(e_1,\dots,e_n)$ becomes a basis $(x_1,\dots,x_n)$ of~$\RR^{p,q}$ with $\langle x_i,x_j\rangle_{p,q} = B(e_i,e_j)$ for all $i,j$.

%%%%%%%%%%%%%%%%%%%%%%%%%
\subsection{Conditions for $\HH^{p,q-1}$-convex cocompactness}

By work of Tits and Vinberg \cite{vin71}, the representation $\rho$ is injective and discrete, and $W_S$ acts properly discontinuously via~$\rho$ on the interior $\widetilde{\Omega}$ of the $\rho(W_S)$-orbit of the fundamental closed polyhedral cone
\begin{equation} \label{eqn:Delta-tilde}
\widetilde{\Delta} = \big\{ v \in \RR^{p,q} ~|~ \langle v,x_i\rangle_{p,q} \leq 0 \quad\forall 1\leq i\leq n\big\}
\end{equation}
in~$\RR^{p,q}$.
Since $B$ is nondegenerate, $\widetilde{\Delta}$ has nonempty interior.
The elements $\rho(s_i)$ are reflections in the faces of~$\widetilde{\Delta}$.
Let $\Omega$ be the image of $\widetilde{\Omega}$ in $\PP(\RR^{p,q})$.
We shall prove the following.

\begin{theorem}\label{thm:coxeter}
In the setting of Section~\ref{subsec:deform-Tits-repr}, suppose that $W_S$ is infinite and irreducible, and that the following conditions are both satisfied:
\begin{enumerate}
  \item\label{item:no-Z2} there does not exist disjoint subsets $S',S''$ of~$S$ such that $W_{S'}$ and $W_{S''}$ are both infinite and commute;
  \item\label{item:param>1} the parameters $\alpha_{ij}$ of \eqref{item:def-B}, which define $B$ and~$\rho$, are all $>1$.
\end{enumerate}
Then $\Omega$ is properly convex and the group $\rho(W_S)\subset\mathrm{Aut}_{\RR}(B)\simeq\OO(p,q)$ is $\HH^{p,q-1}$-convex cocompact: it acts properly discontinously and cocompactly on $\C := \Omega \cap \overline{\Omega^*}$, which is a nonempty properly convex closed subset of~$\HH^{p,q-1}$, and the ideal boundary $\partial_i \C$ does not contain any nontrivial segment.
\end{theorem}

Here we denote by $\Omega^*$ the dual convex to~$\Omega$ (see Section~\ref{subsec:prop-conv-proj}), viewed as a subset of $\PP(\RR^{p,q})$ using the nondegenerate bilinear form $\langle\cdot,\cdot\rangle_{p,q}$ (see \eqref{eqn:Omega*}).

\begin{remarks}
\begin{enumerate}
  \item Condition~\eqref{item:no-Z2} of Theorem~\ref{thm:coxeter} neither implies, nor is implied by, the irreducibility of~$W_S$.
  \item Let $W_S$ be an irreducible, word hyperbolic, right-angled Coxeter group.
In \cite{dgk-racg-cc}, we shall generalize Theorem~\ref{thm:coxeter} to representations $\rho : W_S\to\GL(\RR^n)$ which do not necessarily preserve a quadratic form.
This will enable us to completely describe the moduli space of $P_1$-Anosov representations $\rho : W_S\to\GL(\RR^n)$ realizing $W_S$ as a reflection group.
  \item In this more general context, Marquis \cite{mar17} considered groups generated by reflections in the faces of a polytope which is not necessarily right-angled, but \emph{2-perfect} (a condition on the codimension-$2$ faces).
He gave a full criterion \cite[Th.\,A]{mar17} for a notion of convex cocompactness which is a priori weaker than strong convex cocompactness in $\PP(\RR^n)$.
The $2$-perfect assumption forces the Zariski closure of~the~re\-flection group to be conjugate to $\OO(n-1,1)$ or
$\SL^{\pm}(\RR^n)$ \cite[Th.\,B]{mar17}.
\end{enumerate}
\end{remarks}

%%%%%%%%%%%%%%%%%%%%%%%%%
\subsection{Anosov representations for right-angled Coxeter groups}

Note that in the setting of Theorem~\ref{thm:coxeter} the group $\rho(W_S)\subset\OO(B)\simeq\OO(p,q)$ is irreducible \cite{bh04}.
Here is an immediate consequence of Theorems \ref{thm:main}.\eqref{item:ccc-implies-Anosov}~and~\ref{thm:coxeter}.

\begin{corollary} \label{cor:cox-Anosov}
In the setting of Theorem~\ref{thm:coxeter}, the group $W_S$ is word hyperbolic and the representation $\rho : W_S\to\OO(p,q)$ is $P_1^{p,q}$-Anosov with negative proximal limit set $\Lambda_{\rho(W_S)}\subset\partial_{\scriptscriptstyle\PP}\HH^{p,q-1}$.
\end{corollary}

In particular, this yields a new proof of Moussong's hyperbolicity criterion \cite{mou87} in the case of right-angled Coxeter groups.

\begin{corollary}[Moussong \cite{mou87}] \label{cor:Moussong}
A right-angled Coxeter group $W_S$ is word hyperbolic if and only if it satisfies condition~\eqref{item:no-Z2} of Theorem~\ref{thm:coxeter}, or equivalently if and only if the generating set $S$ does not contain elements $s_{i_1},s_{i_2},s_{i_3},s_{i_4}$ with $m_{i_1,i_3}=m_{i_2,i_4}=\infty$ and $m_{i_j,i_{j+1}}=2$ for all $j\in\ZZ/4\ZZ$.
\end{corollary}

This last condition is sometimes known as the \emph{no empty square} condition.

\begin{proof}
If $W_S$ is word hyperbolic, then it does not contain any subgroup isomorphic to~$\ZZ^2$, and so condition~\eqref{item:no-Z2} of Theorem~\ref{thm:coxeter} is satisfied.
Corollary~\ref{cor:cox-Anosov} states that the converse is true.
\end{proof}

We can now prove Theorem~\ref{thm:exist-Ano-RACG}.

\begin{proof}[Proof of Theorem~\ref{thm:exist-Ano-RACG}]
Let $W=W_S$ be a right-angled Coxeter group with $n$ generators as above.
Since finite groups trivially satisfy Theorem~\ref{thm:exist-Ano-RACG}, we assume that $W_S$ is infinite.
We also assume that $W_S$ is word hyperbolic; then condition~\eqref{item:no-Z2} of Theorem~\ref{thm:coxeter} is clearly satisfied.
If the Coxeter group $W_S$ is irreducible, then Corollary~\ref{cor:cox-Anosov} provides a $P_1^{p,q}$-Anosov representation $\rho : W_S\to\OO(p,q)$.
Otherwise there is a nontrivial partition $S = S' \sqcup S''$ of the generating set such that $W_S$ is the direct product of its subgroups $W_{S'}$ and $W_{S''}$ generated respectively by $S'$ and~$S''$.
Up to switching $S'$ and~$S''$, we may assume that $W_{S''}$ is a finite group and $W_{S'}$ an infinite, irreducible word hyperbolic Coxeter group, still satisfying condition~\eqref{item:no-Z2} of Theorem~\ref{thm:coxeter}.
Corollary~\ref{cor:cox-Anosov} yields a $P_1^{p',q'}$-Anosov representation $\rho': W_{S'} \to \OO(p',q')$.
The composition of $\rho'$ with the natural projection $W_S \to W_{S'}\simeq W_S/W_{S''}$ is also $P_1^{p',q'}$-Anosov since its restriction to the finite-index subgroup $W_{S'}$ is (see \cite[Cor.\,1.3]{gw12}).
\end{proof}

\begin{remark}
In the context of our work \cite{dgk-cox} on proper affine actions of right-angled Coxeter groups, the possibility that $\rho$ might be $P_1^{p,q}$-Anosov in Corollary~\ref{cor:cox-Anosov} was first suggested to us by Anna Wienhard.
\end{remark}

%%%%%%%%%%%%%%%%%%%%%%%%%
\subsection{Proof of Theorem~\ref{thm:coxeter}}

We assume that conditions \eqref{item:no-Z2} and~\eqref{item:param>1} of Theorem~\ref{thm:coxeter} are both satisfied.
Condition~\eqref{item:param>1} ensures that the irreducible reflection group $\rho(W_S)$ is \emph{of negative type} in the terminology of \cite{vin71}, \ie the Cartan matrix $(2B(e_i,e_j))_{1\leq i,j\leq n}$ has at least one negative eigenvalue.
Therefore $\Omega$ is properly convex by \cite[Lem.\,15]{vin71}.
The dual convex set $\Omega^*$, seen as an open subset of $\PP(\RR^{p,q})$ via $\langle\cdot,\cdot\rangle_{p,q}$, is given by
\begin{equation} \label{eqn:Omega*}
\Omega^* = \PP\big(
\big\{ x'\in\RR^{p,q} ~|~ \langle x,x'\rangle_{p,q} < 0 \quad\forall x\in\overline{\widetilde{\Omega}}\big\}\big),
\end{equation}
where $\overline{\widetilde{\Omega}}$ is the closure of $\widetilde{\Omega}$ in $\RR^{p,q}\smallsetminus\{0\}$.

\begin{lemma}\label{lem:non-empty}
The properly convex set $\C = \Omega \cap \overline{\Omega^*}$ is nonempty.
\end{lemma}

\begin{proof}
By Fact~\ref{fact:Yves}, the proximal limit set $\Lambda_{\rho(W_S)}$ is nonempty and contained in the respective closures $\overline{\Omega}$ and~$\overline{\Omega^*}$ of $\Omega$ and~$\Omega^*$ in $\PP(\RR^{p,q})$, hence in the intersection $\overline{\Omega} \cap \overline{\Omega^*}$, which is invariant under $\rho(W_S)$.
Since $\rho$ is irreducible, the interior $\Omega \cap \Omega^* \subset \C$ is nonempty.
\end{proof}

Let
$$\widetilde \Delta^* := \bigg\{ x = \sum_{i=1}^n t_i x_i \in \RR^{p,q} ~\Big|~ t_i \geq 0 \quad\forall 1\leq i \leq n \bigg\}$$
be the dual simplex to~$\widetilde \Delta$, and let $\widetilde{\Sigma} := \widetilde{\Delta} \cap \widetilde{\Delta}^*$, \ie
$$\widetilde{\Sigma} = \bigg\{ x = \sum_{i=1}^n t_i x_i \in \RR^{p,q} ~\Big|~ t_i \geq 0 \ \mathrm{and}\ \langle x, x_j\rangle_{p,q} \leq 0 \quad\forall 1\leq i,j \leq n \bigg\}.$$
The set $\overline{\Omega^*}\subset\PP(\RR^{p,q})$ is the projectivization of the intersection of all $\rho(W_S)$-translates of $\widetilde \Delta^*\smallsetminus\{0\}$.
Therefore $\C = \Omega \cap \overline{\Omega^*}$ is contained in the $\rho(W_S)$-orbit of the projectivization $\Sigma\subset\PP(\RR^{p,q})$ of $\widetilde{\Sigma}\smallsetminus\{0\}$.

\begin{lemma} \label{lem:Sigma-in-Hpq}
The compact set $\Sigma\subset\PP(\RR^{p,q})$ is contained in $\HH^{p,q-1}$. 
\end{lemma}

\begin{proof}
Let $x = \sum_{i=1}^n t_i x_i \in \widetilde{\Sigma}$ where $t_i \geq 0$ and $\langle x, x_j \rangle_{p,q} \leq 0$ for all $1\leq i, j\leq n$.
We have
$$\langle x, x\rangle_{p,q} = \sum_{i=1}^n t_i \langle x, x_i \rangle_{p,q} \leq 0,$$
hence $x$ projects to a point of $\HH^{p,q-1}\cup\partial_{\scriptscriptstyle\PP}\HH^{p,q-1}$.
Suppose by contradiction that $x$ projects to a point of $\partial_{\scriptscriptstyle\PP}\HH^{p,q-1}$, \ie $\sum_{i=1}^n t_i \langle x, x_i \rangle_{p,q} = 0$.
Then $t_i \langle x, x_i \rangle_{p,q} = 0$ for all $1\leq i\leq n$.
Let $I \subset \{1,\ldots,n\}$ be the (nonempty) collection of indices $k$ such that $t_k > 0$.
For any $k \in I$ we have $\langle x, x_k\rangle_{p,q} = 0$, hence $t_k = \sum_{i\in I_k} t_i \alpha_{i,k}$ where $I_k := \{ i\in I \,|\, m_{i,k}=\infty\}$.
Recall that $\alpha_{i,k}>1$ for all $i\in I_k$.
Therefore we reach a contradiction by considering $k \in I$ such that $t_k$ is minimal.
\end{proof} 

\begin{lemma}\label{lem:P-in-Omega}
The stabilizer in~$W_S$ of any point of~$\Sigma$ is finite.
In particular (see \cite[Th.\,2]{vin71}), the compact set $\Sigma$ is contained in~$\Omega$.
\end{lemma}

\begin{proof}
Let $x = \sum_{i=1}^n t_i x_i \in \widetilde{\Sigma}$ where $t_i \geq 0$ and $\langle x, x_j \rangle_{p,q} \leq 0$ for all $1\leq i, j\leq n$.
The stabilizer of $[x]\in\PP(\RR^{p,q})$ in~$W_S$ is the subgroup $W_{S_x}$ generated by the subset
$$S_x := \big\{ s_j \in S ~|~ \langle x, x_j \rangle_{p,q} = 0\big\}.$$
We aim to show $W_{S_x}$ is finite.
For this we split $S_x$ into the disjoint union of its two subsets $S_x^0 := \{s_j\in S_x \,|\, t_j = 0\}$ and $S_x^> := \{s_j\in S_x \,|\, t_j > 0\}$.

We claim that any element of~$S_x^0$ commutes with any element of~$S_x^>$; in particular, $W_{S_x}$ is the direct product of its subgroups $W_{S_x^0}$ and~$W_{S_x^>}$ generated respectively by $S_x^0$ and~$S_x^>$.
Indeed, for any $s_j \in S_x$ we have by definition
\begin{equation} \label{eqn:Sx-def}
0 = \langle x,x_j\rangle_{p,q} = \sum_{i=1}^n t_i \langle x_i, x_j\rangle_{p,q} = \sum_{s_i \in S^>} t_i \langle x_i, x_j\rangle_{p,q},
\end{equation}
where
$$S^> := \{s_i\in S \,|\, t_i > 0\}.$$
If $s_j\in S_x^0$, then each term of the right-hand sum in \eqref{eqn:Sx-def} is nonpositive, hence must be zero.
Thus for any $s_i\in S^>$ and $s_j\in S_x^0$ we have $\langle x_i,x_j\rangle_{p,q} = 0$, which means that $s_i$ and~$s_j$ commute.
Therefore $W_{S_x}=W_{S_x^0}\times W_{S_x^>}$.

Let us prove that $W_{S_x^>}$ is finite.
For this it is sufficient to prove that $m_{j,k} = 2$ for all distinct $s_j, s_k \in S_x^>$.
Suppose by contradiction that $m_{j,k}=\infty$ for some $s_j, s_k \in S_x^>$.
By definition, we have
\begin{align*}
0 = \langle x,x_j \rangle_{p,q} &= t_j + \sum_{s_i \in S^>,\ s_i\neq s_j} t_i \langle x_i, x_j \rangle_{p,q} \\
& \leq t_j - \alpha_{j,k} t_k < t_j - t_k,
\end{align*}
where the last inequality uses condition~\eqref{item:param>1} of Theorem~\ref{thm:coxeter}.
But similarly by considering $\langle x,x_k \rangle_{p,q} = 0$, we find $t_k - t_j < 0$ which is impossible.
Thus $m_{j,k} = 2$ for all distinct $s_j, s_k \in S_x^>$ and $W_{S_x^>}$ is a finite group (a product of finitely many copies of $\ZZ/2\ZZ$).

We now check that $W_{S_x^0}$ is finite.
For this it is sufficient to check that $W_{S^>}$ is infinite, since $S_x^0$ and $S^>$ are disjoint and condition~\eqref{item:no-Z2} of Theorem~\ref{thm:coxeter} is satisfied.
By Lemma~\ref{lem:Sigma-in-Hpq} we have
$$\langle x,x\rangle_{p,q} = \sum_{s_i, s_\ell \in S^>} t_i t_k \langle x_i,x_k\rangle_{p,q} < 0.$$
The diagonal terms in the sum are positive and so there must be a nonzero nondiagonal term.
In other words, there are two distinct elements $s_i, s_\ell \in\nolinebreak S^>$ generating an infinite dihedral group, proving that $W_{S^>}$ is infinite.
\end{proof}

For $r > 0$, let $\C_r$ be the closed uniform $r$-neighborhood of $\C$ in $\Omega$ with respect to the Hilbert metric $d_\Omega$.
It is properly convex \cite[(18.12)]{bus55}.
By Lemma~\ref{lem:unif-neighb}, since the action of $W_S$ on $\C$ via~$\rho$ is cocompact, we may choose $r$ small enough so that $\C_r \subset \HH^{p,q-1}$; the action of $W_S$ via~$\rho$ is still properly discontinuous and cocompact on $\C_r$. We let $\mathcal{U}_r$ denote the interior of $\C_r$.

\begin{lemma} \label{lem:cox-no-ray}
There is no infinite straight ray contained in the boundary $\partial_{\scriptscriptstyle\HH} \C_r := \C_r \smallsetminus \mathcal{U}_r$ of $\C_r$ in~$\HH^{p,q-1}$.
\end{lemma}

\begin{proof}
Suppose by contradiction that there is such a ray $\mathcal{R} = [y,z) \subset \partial_{\scriptscriptstyle\HH} \C_r$, with $y \in \partial_{\scriptscriptstyle\HH} \C_r$ and $z \in \partial_i \C_r$.
Let $\Delta\subset\PP(\RR^{p,q})$ be the fundamental polytope of the reflection group $\rho(W_S)$, namely the image of \eqref{eqn:Delta-tilde} in $\PP(\RR^{p,q})$; it is bounded by the reflection hyperplanes $\PP(x_i^\perp)$ for $1\leq i\leq n$.

Up to replacing $\mathcal{R}$ with some $\rho(W_S)$-translate, we may assume that $\mathcal{R}$ crosses $\PP(x_i^\perp)$ and $\PP(x_j^\perp)$ transversely for some generators $s_i\neq s_j$ that do not commute.
Indeed, let $(\rho(s_{i_1}\dots s_{i_m})\cdot\Delta)_{m\geq M}$ be a sequence of $\rho(W_S)$-translates of~$\Delta$ that  meet~$\mathcal{R}$, where $i_1,\dots,i_m\in\{1,\dots,n\}$ and $M\in\NN^*$.
Since $\mathcal{R}$ is infinite, the elements $s_{i_{\ell}}$ for $\ell\geq M$ do not all commute: there exist $M\leq\ell<m$ such that $s_{i_{\ell}}$ does not commute with~$s_{i_m}$ but does with $s_{i_{\ell+1}},\dots,s_{i_{m-1}}$.
Then $s_{i_{\ell}}\dots s_{i_{m-1}} = s_{i_{\ell+1}}\dots s_{i_{m-1}}s_{i_{\ell}}$, and so up to renumbering we may assume that $s_{i_{m-1}}$ and $s_{i_m}$ do not commute.
Thus $\rho(s_{i_1}\dots s_{i_{m-1}})^{-1}\cdot\mathcal{R}$ meets $\Delta$ and its translates $\rho(s_{i_{m-1}})\cdot\Delta = \rho(s_{i_{m-1}})^{-1}\cdot\Delta$ and $\rho(s_{i_m})\cdot\Delta$.
It follows that $\rho(s_{i_1}\dots s_{i_{m-1}})^{-1}\cdot\mathcal{R}$ crosses the hyperplanes $\PP(x_{i_{m-1}}^\perp)$ and $\PP(x_{i_m}^\perp)$ transversely.
Indeed, $\rho(s_{i_{m-1}})\cdot \Delta$ and $\rho(s_{i_m}) \cdot \Delta$ are separated in~$\Omega$ by the hyperplanes $\PP(x_{i_{m-1}}^\perp)$ and $\PP(x_{i_m}^\perp)$, whose intersection lies outside of $\Omega$ because it is pointwise fixed by the infinite subgroup of $\rho(W_S)$ generated by $\rho(s_{i_{m-1}})$ and $\rho(s_{i_m})$.
Therefore, up to replacing $\mathcal{R}$ with some $\rho(W_S)$-translate, we may assume that $\mathcal{R}$ crosses $\PP(x_i^\perp)$ and $\PP(x_j^\perp)$ transversely for some generators $s_i\neq s_j$ that do not commute.

Let $y_i$ be the intersection point of $\mathcal{R}$ with the hyperplane $\PP(x_i^\perp)$ and let $H$ be a supporting hyperplane to $\C_r$ at~$y_i$.
Then $H$ must contain~$\mathcal{R}$.
Similarly, $\rho(s_i)\cdot H$ must contain~$\mathcal{R}$.
Note that $H \cap \PP(x_i^\perp) = (\rho(s_i)\cdot H) \cap \PP(x_i^\perp)$.
Since $H$ is spanned by $H \cap \PP(x_i^\perp)$ and $\mathcal{R}$, we deduce that $\rho(s_i)\cdot H = H$.
On the other hand, since $H$ contains~$\mathcal{R}$, it is also a supporting hyperplane to $\C_r$ at the intersection point $y_j$ of $\mathcal{R}$ with the hyperplane $\PP(x_j^{\perp})$, and similarly $\rho(s_j)\cdot H = H$.
Therefore $H$ is invariant under $\rho(s_i s_j)$.

By a straightforward calculation (see \cite[\S\,2]{vin71}), condition~\eqref{item:param>1} of Theorem~\ref{thm:coxeter} implies that $g:=\rho(s_is_j)\in\PO(p,q)$ is proximal in $\PP(\RR^{p,q})$, hence in $\partial_{\scriptscriptstyle\PP}\HH^{p,q-1}$ (see Remark~\ref{rem:lim-set-POpq}).
If $\xi_g^+,\xi_g^-\in\partial_{\scriptscriptstyle\PP}\HH^{p,q-1}$ are the attracting and repelling fixed points of~$g$, then $\C$ meets the projective line spanned by $\xi_g^+$ and~$\xi_g^-$ in an open interval $(\xi_g^+,\xi_g^-)$: indeed, we have $\xi_g^+,\xi_g^-\in\partial_i\C$ and $\langle\xi_g^+,\xi_g^-\rangle_{p,q}\neq 0$ (see Remark~\ref{rem:lim-set-POpq} and Fact~\ref{fact:Yves}).
Since $y_i\in\C_r\subset\Omega$ we have $y_i\notin(\xi_g^-)^{\perp}$ (see Proposition~\ref{prop:Lambda-non-pos-neg}), and so $g^m\cdot y_i\to\xi_g^+$ as $m\to +\infty$; similarly, $g^{-m}\cdot y_i\to\xi_g^-$ as $m\to +\infty$.
Thus $H$ contains the interval $(\xi_g^-, \xi_g^+)$: contradiction since $H$ does not meet~$\C$.
\end{proof}

Before giving the proof of Theorem~\ref{thm:coxeter}, we prove one more general lemma.

\begin{lemma} \label{lem:C-closed}
Let $\Gamma$ be a discrete subgroup of $\OO(p,q)$ preserving a properly convex open subset $\Omega \subset \PP(\RR^{p,q})$.
Any accumulation point of the $\Gamma$-orbit of a compact subset $\mathcal{K}$ of $\Omega \cap \HH^{p,q-1}$ is contained in $\partial_{\scriptscriptstyle\PP}\HH^{p,q-1}$.
\end{lemma}

\begin{proof}
Suppose by contradiction that there are sequences $(y_m)\in\mathcal{K}^{\NN}$ and $(\gamma_m)\in\Gamma^{\NN}$ such that the $\gamma_m$ are pairwise distinct and $z_m := \gamma_m\cdot y_m$ converges to some $z\in\HH^{p,q-1}$.
We can lift the $y_m\in \HH^{p,q-1}$ to vectors $x_m\in\RR^{p,q}$ with $\langle x_m,x_m\rangle_{p,q}=-1$; both the $x_m$ and the $\gamma_m\cdot x_m$ stay in a compact subset of~$\RR^{p,q}$. 
On the other hand, since $\Gamma$ is discrete, there exists $x \in \RR^{p,q}\smallsetminus\{0\}$ such that $(\gamma_m\cdot x)_{n\in\NN}$ leaves every compact subset of~$\RR^{p,q}$.
(Indeed, if we fix a basis of $\RR^{p,q}$, then at least one element $x$ of this basis must satisfy this property.)
Up to passing to a subsequence, we may assume that the direction of $\gamma_m\cdot x$ converges to some null direction $\ell$.
There exists $\varepsilon > 0$ such that all segments $[x_m - \varepsilon x,x_m+\varepsilon x] \subset \RR^{p,q}\smallsetminus\{0\}$ project to segments $\sigma_m$ contained in~$\Omega$.
The images $\gamma_m\cdot\sigma_m$, which are again contained in~$\Omega$, converge to the full projective line spanned by $x$ and~$\ell$.
This contradicts the proper convexity of~$\Omega$.
Thus the $\Gamma$-orbit of~$\mathcal{K}$ does not have any accumulation point in~$\HH^{p,q-1}$.
\end{proof}

\begin{proof}[Proof of Theorem~\ref{thm:coxeter}]
Let $\mathcal{C}'\subset\PP(\RR^{p,q})$ be the $\rho(W_S)$-orbit of~$\Sigma$.
By Lemmas \ref{lem:Sigma-in-Hpq} and~\ref{lem:P-in-Omega}, we have $\mathcal{C}'\subset\HH^{p,q-1}\cap\Omega$.
In particular, the action of $W_S$ on~$\mathcal{C}'$ via~$\rho$ is properly discontinuous, and cocompact since $\Sigma$ is a compact fundamental domain.

The set $\C$ is nonempty by Lemma~\ref{lem:non-empty}.
Since $\C \subset \mathcal{C}'$ and $\C$ is closed in $\Omega$, the action of $W_S$ on $\C$ via~$\rho$ is also properly discontinuous and cocompact.
By Lemma~\ref{lem:C-closed}, the set $\C$ is closed in $\HH^{p,q-1}$.
We now complete the proof by showing that $\partial_i \C$ does not contain any nontrivial projective segment.

Suppose by contradiction that there is a nontrivial segment $[a',b']$ in $\partial_i \C$.
Since $\partial_i\C\subset\partial_{\scriptscriptstyle\PP}\HH^{p,q-1}$, we have $\langle a',b'\rangle_{p,q}=0$.
By Lemma~\ref{lem:unif-neighb}, there exists $r>0$ such that the closed uniform neighborhood $\C_r$ of $\C$ in $(\Omega,d_{\Omega})$ is properly convex, contained in $\HH^{p,q-1}$, and the action of $\Gamma$ on~$\C_r$ is properly discontinuous and cocompact.
Extend $[a',b']$ to a segment $[a,b]$ which is maximal in $\partial_i \C_r$.
Note that $[a,b]$ is also a maximal segment of $\partial_{\scriptscriptstyle\PP} \C_r$ (since any segment of $\partial_{\scriptscriptstyle\PP} \C_r$ containing $[a,b]$ is contained in $\partial_{\scriptscriptstyle\PP} \HH^{p,q-1} \cap \partial_{\scriptscriptstyle\PP} \C_r = \partial_i \C_r$).
Consider a point $c \in \C$ and a sequence of points $y_m\in\C$ in the triangle with vertices $a,b,c$, such that $(y_m)_m$ converges to an interior point $y$ of $[a,b] \subset \partial_i \C_r$.

We claim that in $\mathcal{U}_r:=\mathrm{Int}(\C_r)$, the Hilbert distance $d_{\mathcal{U}_r}$ from $y_m$ to either projective interval $(a,c]$ or $(b,c]$ tends to infinity with~$m$.
Indeed, consider a sequence $(z_m)_m$ of points of $(a,c]$ converging to $z \in [a,c]$ and let us check that $d_{\mathcal{U}_r}(y_m,z_m) \to +\infty$ (the proof for $(b,c]$ is the same).
If $z \in (a,c]$, then $z\in\C_r$ and so $d_{\mathcal{U}_r}(y_m,z_m) \to +\infty$ by properness of the Hilbert metric.
Otherwise $z = a$.
In that case, for each~$m$, consider $y'_m, z'_m \in \partial_{\scriptstyle\PP} \mathcal{U}_r$ such that $y'_m, y_m, z_m, z'_m$ are aligned in that order.
Up to taking a subsequence, we may assume $y'_m \to y'$ and $z'_m \to z'$ for some $y',z' \in \partial_{\scriptstyle\PP} \mathcal{U}_r$, with $y', y, a, z'$ aligned in that order.
By maximality of $[a,b]$ in $\partial_{\scriptscriptstyle\PP}\mathcal{U}_r$, we must have $z = a = z'$, hence $d_{\mathcal{U}_r}(y_m,z_m) \to +\infty$ in this case as well, proving the claim.

Since $\Gamma$ acts cocompactly on~$\C$, there is a sequence $(\gamma_m)\in\Gamma^{\NN}$ such that $\gamma_m\cdot\nolinebreak y_m$ remains in a fixed compact subset of $\C\subset\mathcal{U}_r$.
Up to passing to a subsequence, we may assume that $(\gamma_m\cdot y_m)_m$ converges to some $y_{\infty}\in\C$, and $(\gamma_m\cdot a)_m$ and $(\gamma_m\cdot b)_m$ and $(\gamma_m\cdot\nolinebreak c)_m$ converge respectively to some $a_{\infty},b_{\infty},c_{\infty}\in\partial_i\C_r$, with $[a_{\infty},b_{\infty}]\subset\partial_i\C_r$.
The triangle with vertices $a_\infty, b_\infty, c_\infty$ is nondegenerate since it contains $y_\infty \in \C$.
Further, $y_\infty$ is infinitely far (for the Hilbert metric $d_{\mathcal{U}_r}$) from the edges $[a_\infty, c_\infty]$ and $[b_\infty, c_\infty]$, and so these edges are fully contained in $\partial_{\scriptscriptstyle\PP}\mathcal{U}_r$.
By Lemma~\ref{lem:cox-no-ray}, there is no straight infinite ray in $\partial_{\scriptscriptstyle\HH} \C_r  := \C_r \smallsetminus \mathcal{U}_r$, hence $[b_\infty, c_\infty]$ and $[a_\infty, c_\infty]$ do not intersect $\partial_{\scriptscriptstyle\HH} \C_r$ and are contained in $\partial_i \C_r\subset \partial_{\scriptscriptstyle\PP} \HH^{p,q-1}$. 
But this means that $\langle a_\infty, b_\infty\rangle_{p,q} = \langle b_\infty, c_\infty\rangle_{p,q} = \langle a_\infty, c_\infty\rangle_{p,q} = 0$, hence every point of the triangle with vertices $a_\infty, b_\infty, c_\infty$ is null, contradicting the fact that $y_\infty$ lies in $\C \subset \HH^{p,q-1}$.
\end{proof}

\begin{remark}
In the proof of Theorem~\ref{thm:coxeter}, we do not assume that the set  $\mathcal{C}' = \rho(W_S)\cdot\Sigma$ is convex.
We only use that $\mathcal{C}'$ is contained in~$\Omega$ and contains $\C = \Omega \cap \overline{\Omega^*}$.
In fact, by studying the local convexity of $\mathcal{C}'$ along its boundary faces, it is possible to show that $\mathcal{C}'$ is convex and equal to~$\C$; this is done in a general setting by Greene--Lee--Marquis \cite{glm17}.
\end{remark}

%%%%%%%%%%%%%%%%%%%%%%%%%%%%%%%%%%%%%%%%%%%%%%%%%%%
%%%%%%%%%%%%%%%%%%%%%%%%%%%%%%%%%%%%%%%%%%%%%%%%%%%
\appendix
\section{Connectedness in the space of unordered tuples} \label{app:connectivity}

The following general statement, on which Proposition~\ref{prop:conn-transv} relies, is probably well known.
We provide a proof for the reader's convenience.

\begin{fact} \label{fact:conn-ktuple}
Let $\Lambda$ be a connected Hausdorff topological space.
For $k\geq 1$, the space $\Lambda^{(k)}$ of unordered $k$-tuples of pairwise distinct points of~$\Lambda$ is also connected (for the restriction of the product topology).
\end{fact}

Given a finite subset $X$ of~$\Lambda$ and a point $x\in \Lambda \smallsetminus X$, we denote by $\Lambda_x^X$ the connected component of $\Lambda\smallsetminus X$ containing~$x$.
Since $\Lambda$ is Hausdorff, its finite subsets are closed, and so $\Lambda_x^X$ is an open subset of $\Lambda$ and its closure $\overline{\Lambda_x^X}$ is contained in $\Lambda_x^X\cup X$.

\begin{lemma} \label{lem:conn-1}
Let $\Lambda$ be a connected topological space with closed singletons.
For any $k\geq 1$ and $\{x_0, \dots, x_k\}\in\Lambda^{(k+1)}$, there exists $0\leq i_0 < k$ such that the unordered $k$-tuples $\{x_0, \dots, x_k\} \smallsetminus \{x_{i_0}\} $ and $\{x_0, \dots, x_{k-1}\}$ belong to the same connected component of~$\Lambda^{(k)}$.
\end{lemma}

\begin{proof}
For $0\leq i<k$, let $X_i:=\{x_0, \dots, x_{k-1}\}\smallsetminus \{x_i\}$.
It is sufficient to prove the existence of $0\leq i_0 < k$ such that $x_{i_0}$ and $x_k$ belong to the same connected component of $\Lambda\smallsetminus X_{i_0}$, \ie $x_{i_0}\in \Lambda_{x_k}^{X_{i_0}}$.
We have
\begin{equation} \label{eqn:inters-1}
\overline{\bigcap_{0\leq i <k} \Lambda_{x_k}^{X_i}} \subset \bigcap_{0\leq i <k} \overline{\Lambda_{x_k}^{X_i}} \subset \bigcap_{0\leq i <k} (\Lambda_{x_k}^{X_i} \cup X_i).
\end{equation}
Suppose by contradiction that $x_i\notin \Lambda_{x_k}^{X_i}$ for all $0\leq i<k$: then $x_i\notin\Lambda_{x_k}^{X_i} \cup X_i$, so the right-hand intersection in \eqref{eqn:inters-1} is disjoint from all $X_i$ and can be rewritten $\bigcap_{0\leq i <k} \Lambda_{x_k}^{X_i}$.
This set is therefore open and (by~\eqref{eqn:inters-1}) closed, and contains $x_k$ but no other $x_i$, contradicting the fact that $\Lambda$ is connected.
\end{proof}

\begin{lemma} \label{lem:conn-2}
Let $\Lambda$ be a connected topological space with closed singletons.
For any $k\geq 1$ and $\{x_0, \dots, x_k\}\in\Lambda^{(k+1)}$, there exists $1\leq j_0 \leq k$ such that $x_i \in \Lambda_{x_0}^{\{x_{j_0}\}}$ for all $i\in\{1,\dots,k\}\smallsetminus\{ j_0\}$.
\end{lemma}

We call this property $\mathrm{H}_k$, or $\mathrm{H}_k(x_0, \{x_1,\dots, x_k\})$ to be specific.

\begin{proof}
We argue by induction.
Property~$\mathrm{H}_1$ is vacuously true.
Assuming $\mathrm{H}_{k-1}$ where $k\geq 2$, let us prove $\mathrm{H}_k$ by contradiction.
We have
\begin{equation} \label{eqn:inters-2}
\overline{\bigcap_{1\leq j \leq k} \Lambda_{x_0}^{\{x_j\}}} \subset \bigcap_{1\leq j \leq k} \overline{\Lambda_{x_0}^{\{x_j\}}} \subset \bigcap_{1\leq j \leq k} (\Lambda_{x_0}^{\{x_j\}} \cup \{x_j\}).
\end{equation}
Suppose $\mathrm{H}_k(x_0, \{x_1,\dots, x_k\})$ fails: that is, for all $1\leq j \leq k$, 
$$\{x_1,\dots,x_k\}\smallsetminus\{ x_j\} \not\subset \Lambda_{x_0}^{\{x_{j}\}}.$$ 

We claim that the right member of \eqref{eqn:inters-2} then cannot contain any $x_i$ for $1\leq i \leq k$: indeed that would imply $x_i\in \Lambda_{x_0}^{\{x_j\}}$ for all $j\in\{1,\dots,k \} \smallsetminus \{i\}$, hence the above relationship would yield
$$\{x_1,\dots,x_k\}\smallsetminus\{x_i, x_j\} \not\subset \Lambda_{x_0}^{\{x_{j}\}}$$ 
for all $j\in \{1,\dots, k\}\smallsetminus \{i\}$, contradicting $\mathrm{H}_{k-1}(x_0, \{x_1,\dots, x_k\}\smallsetminus \{x_i\})$.

Therefore the right-hand side of \eqref{eqn:inters-2} can be written $\bigcap_{1\leq j \leq k} \Lambda_{x_0}^{\{x_j\}}$, which by \eqref{eqn:inters-2} turns out to be closed.
It is also open, and contains $x_0$ but no other~$x_i$: this contradicts connectedness of $\Lambda$.
Therefore $\mathrm{H}_k$ holds.
\end{proof}

\begin{proof}[Proof of Fact~\ref{fact:conn-ktuple}]
We argue by induction on~$k$ to prove that $\Lambda^{(k)}$ is connected for any connected topological space~$\Lambda$ with closed singletons.
The case $k=1$ is obviously true.
For $k\geq 2$, suppose that $(\Lambda')^{(k-1)}$ is connected for any connected~$\Lambda'$ with closed singletons, and let us prove that $\Lambda^{(k)}$ is connected for any connected~$\Lambda$ with closed singletons.

Consider $\{x_0, \dots, x_k\}\in\Lambda^{(k+1)}$.
By Lemma~\ref{lem:conn-2}, up to exchanging the labels $j_0$ and~$k$, we have $x_i \in \Lambda_{x_0}^{\{x_k\}}$ for all $1\leq i\leq k-1$, \ie all points $x_0, \dots, x_{k-1}$ belong to the same connected component $\Lambda'$ of $\Lambda\smallsetminus \{x_k\}$.
Since $\Lambda'^{(k-1)}$ is connected, all $(k-1)$-tuples $\{x_0, \dots, x_{k-1}\}\smallsetminus \{x_i\}$ for $0\leq i < k$ belong to the same connected component of $(\Lambda\smallsetminus \{x_k\})^{(k-1)}$, and so all $k$-tuples $\{x_0, \dots, x_k\}\smallsetminus \{x_i\}$ for $0\leq i < k$ belong to the same component of $\Lambda^{(k)}$.
But by Lemma~\ref{lem:conn-1} one of these $k$-tuples (for $i=i_0$) belongs to the same component as $\{x_0, \dots, x_{k-1}\}$.
Therefore \emph{all} $k$-tuples contained in $\{x_0, \dots, x_k\}$ belong to the same component of $\Lambda^{(k)}$.
This is true for all $\{x_0, \dots, x_k\}\in\Lambda^{(k+1)}$, hence $\Lambda^{(k)}$ is connected.
\end{proof}

%%%%%%%%%%%%%%%%%%%%%%%%%%%%%%%%%%%%%%%%%%%%%%%%%%%
%%%%%%%%%%%%%%%%%%%%%%%%%%%%%%%%%%%%%%%%%%%%%%%%%%%

\end{document}